\documentclass[11pt]{amsart}

\usepackage{amssymb}
\usepackage{amsmath}
\usepackage{amsthm}
\usepackage{graphicx}
\usepackage{macros}
\usepackage{bm}
\usepackage[usenames,dvipsnames]{xcolor}
\usepackage{subcaption}
\usepackage{tikz}
\usepackage{pgfplots} 
\usetikzlibrary{3d,calc}
\usepackage{mwe}

\pgfplotsset{compat=newest}

\usepackage{float}
\usepackage{color}
\usepackage{hyperref}
\hypersetup{
  colorlinks,
  citecolor= Brown,
  linkcolor= link,
  urlcolor= link
}

\usepackage{appendix}
\usepackage{array}
\usepackage{footnote}
\usepackage{booktabs}
\usepackage{hhline}
\makesavenoteenv{tabular}

\usepackage{algpseudocode}
\usepackage{algorithm}

\usepackage{caption}

\usepackage{bbm}

\usepackage{scrextend}
 \deffootnote{0em}{1.6em}{\enskip}

\definecolor{link}{rgb}{0.18,0.25,0.63}
\definecolor{myred}{rgb}{0.7,0.25,0.2}
\definecolor{mygray}{rgb}{0.8,0.8,0.8}
\usepackage[top=3cm,bottom=3cm,left=2.5cm,right=2.5cm]{geometry}

\numberwithin{equation}{section}

\newcommand{\balpha}{\boldsymbol{\alpha}}

\newcommand{\cmmnt}[1]{}

\makeatletter
\g@addto@macro{\endabstract}{\@setabstract}
\newcommand{\authorfootnotes}{\renewcommand\thefootnote{\@fnsymbol\c@footnote}}%
\makeatother

\begin{document}
\definecolor{link}{rgb}{0,0,0}
\definecolor{mygrey}{rgb}{0.34,0.34,0.34}
\def\blue #1{{\color{blue}#1}}

 \begin{center}
 \large
 % \textbf{ Structure of solutions and asymptotic analysis of image denoising models with infimal convolution of data discrepancy}. \par \bigskip \bigskip
   \textbf{Dualization and Automatic Distributed Parameter Selection of Total Generalized Variation via Bilevel Optimization} \par \bigskip \bigskip
   \normalsize
  \textsc{Michael Hinterm\"uller}\textsuperscript{$\,1,$$2$}, \textsc{Kostas Papafitsoros}\textsuperscript{$\,2$},  \textsc{Carlos N. Rautenberg}\textsuperscript{$\,3$},\\
   \textsc{Hongpeng Sun}\textsuperscript{$\,4$}
\let\thefootnote\relax\footnote{
\textsuperscript{$1$}Humboldt-Universit\"at zu Berlin, Unter den Linden 6, 10999 Berlin, Germany
}
\let\thefootnote\relax\footnote{
\textsuperscript{$2$}Weierstrass Institute for Applied Analysis and Stochastics (WIAS), Mohrenstrasse 39, 10117 Berlin, Germany}

\let\thefootnote\relax\footnote{
\textsuperscript{$3$}Department of Mathematical Sciences, George Mason University, Fairfax, VA 22030, USA}

\let\thefootnote\relax\footnote{
\textsuperscript{$4$}Institute for Mathematical Sciences, Renmin University of China, 100872 Beijing, People's Republic of China}

\let\thefootnote\relax\footnote{
%\hspace{3.2pt}Emails: \href{mailto:Hintermueller@wias-berlin.de}{\nolinkurl{hintermueller@wias-berlin.de}}
%
% \hspace{26pt}\href{mailto: Papafitsoros@wias-berlin.de}{\nolinkurl{papafitsoros@wias-berlin.de}}
% 
% \hspace{26pt}\href{mailto: Rautenberg@wias-berlin.de}{\nolinkurl{crautenb@gmu.edu}}
%
% \hspace{26pt}\href{mailto: hpsun@amss.ac.cn}{\nolinkurl{hpsun@amss.ac.cn}}}
\hspace{3.2pt}Emails: \href{mailto:Hintermueller@wias-berlin.de}{\nolinkurl{hintermueller@wias-berlin.de}},
 \href{mailto: Papafitsoros@wias-berlin.de}{\nolinkurl{papafitsoros@wias-berlin.de}},
 \href{mailto: Rautenberg@wias-berlin.de}{\nolinkurl{crautenb@gmu.edu}},
 \href{mailto: hpsun@amss.ac.cn}{\nolinkurl{hpsun@amss.ac.cn}}
}
\end{center}
\vspace{-0.8cm}

\begin{abstract}
Total Generalized Variation (TGV)  regularization in image reconstruction %an established high quality regularizer for a variety of variational image reconstruction tasks. By incorporating 
relies on an infimal convolution type combination of generalized first- and second-order derivatives. %in the regularization process it 
This helps to avoid the staircasing effect of Total Variation (TV) regularization, while still preserving sharp contrasts in images. The associated regularization effect crucially hinges on two parameters whose proper adjustment represents a challenging task. In this work, a bilevel optimization framework with a suitable statistics-based upper level objective is proposed in order to automatically select these parameters. The framework allows for spatially varying parameters, thus enabling better recovery in high-detail image areas. A rigorous dualization framework is established, and for the numerical solution, two Newton type methods for the solution of the lower level problem, i.e. the image reconstruction problem, and  two bilevel TGV algorithms are introduced, respectively. Denoising tests confirm that automatically selected distributed regularization parameters lead in general to improved reconstructions when compared to results for scalar parameters.
 \end{abstract}

%\tableofcontents

\definecolor{link}{rgb}{0.18,0.25,0.63}

\section{Introduction}
In this work we analyze and implement a bilevel optimization framework for automatically selecting spatially varying regularization parameters $\balpha:=(\alpha_{0}, \alpha_{1})^\top\in C(\overline{\om})^2$, $\balpha>0$, in the following image reconstruction problem:
\begin{equation}\label{weightedTGV_min_intro}
\text{minimize}\quad \frac{1}{2}\int_{\om}(Tu-f)^{2}dx+\tgv_{\balpha}^{2}(u)\quad\text{over }u\in\bv(\om),
\end{equation}
where the second-order Total Generalized Variation (TGV) regularization is given by
 \begin{equation}\label{weighted_TGV_def_intro}
 \begin{split}
\tgv_{\balpha}^{2}(u)=\sup \Big\{\int_{\om} u\,\di^{2} \phi\,dx:\phi\in C_{c}^{\infty}(\om,\mathcal{S}^{d\times d}),\; & |\phi(x)|_{r}\le \alpha_{0}(x), \\
& |\di\phi(x)|_{r}\le \alpha_{1}(x), \text{ for all } x\in\om  \Big\}.
\end{split}
\end{equation}
Here, $\om\subseteq \RR^{d}$ is a bounded, open image domain with Lipschitz boundary,  $\mathcal{S}^{d\times d}$ denotes the space of $d\times d$ symmetric matrices,  $T: L^{d/d-1}(\om)\to L^{2}(\om)$ is a bounded linear (output) operator, and $f$ denotes given data which satisfies
\begin{equation}
f=Tu_{true}+\eta.
\end{equation}
In this context, $\eta$ models a highly oscillatory (random) component with zero mean and known quadratic deviation (variance) $\sigma^{2}$ from the mean.  Further, $L^2(\om)$ and $L^{d/d-1}(\om)$ denote standard Lebesgue spaces \cite{Adams}, and $|\cdot|_r$, $1\leq r\leq+\infty$, represents the $\ell^r$ vector norm or its associated matrix norm. The space of infinitely differentiable functions with compact support in $\om$ and values in $\mathcal{S}^{d\times d}$ is denoted by $C_{c}^{\infty}(\om,\mathcal{S}^{d\times d})$. Further, we refer to Section \ref{sec:weighted_TV_functional} for the definition of the first- and second-order divergences $\operatorname{div}$ and $\operatorname{div}^2$, respectively.
 
Originally, the TGV functional was introduced for scalar parameters $\alpha_{0}, \alpha_{1}>0$ only; see \cite{TGV}. It serves as a higher order extension of the well-known Total Variation (TV) regularizer \cite{ChambolleLions, rudin1992nonlinear}, preserves edges (i.e., sharp contrast) \cite{Papafitsoros_Bredies, valkonen_jump_2}, and promotes piecewise affine reconstructions while avoiding the often adverse staircasing effect (i.e., piecewise constant structures) of TV \cite{poon_TV_geometric, Jalalzai2015jmiv, ring2000structural}; see Figure \ref{fig:tv_tgv_comp} for an illustration.
\begin{figure}[t!]
	%\centering
	\begin{subfigure}[t]{0.20\textwidth}\centering
	\includegraphics[height=3.05cm]{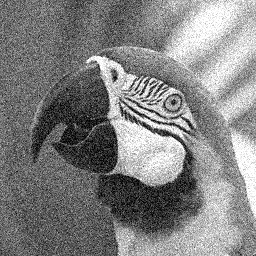}\\% BEST TV PSNR
	\tiny{Noisy}
	\end{subfigure}
	\begin{subfigure}[t]{0.20\textwidth}\centering
	\includegraphics[height=3.05cm]{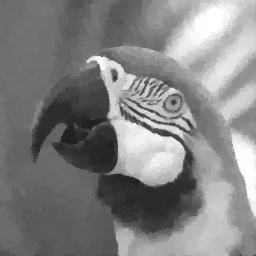}\\
	\tiny{TV}
	\end{subfigure} 
	\begin{subfigure}[t]{0.20\textwidth}\centering
	\includegraphics[height=3.05cm]{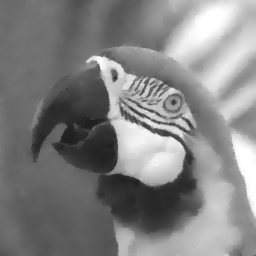}\\% BEST TGV PSNR
	\tiny{ TGV}
	\end{subfigure}
	\caption{Gaussian denoising: Typical difference between TV  (piecewise constant) and TGV reconstructions (piecewise affine)}
	\label{fig:tv_tgv_comp}
\end{figure}
These properties of TGV have made it a successful regularizer in variational image restoration for a variety of applications \cite{bredies2012artifact,   bredies2013tgv, TGV_decompression_part1, TGV_decompression_part2,TGV, BredValk, knoll2011second, diff_tens}. Extensions to manifold-valued data, multimodal and dynamic problems \cite{TGV_dynamic_PET, TGV_manifold, TGV_video, TGV_electron, TGV_MR_PET, TGV_dynamic_MRI} %\cite{tgvcolour, bredies2012artifact,   bredies2013tgv, TGV_decompression_part1, TGV_decompression_part2,TGV, BredValk, knoll2011second, diff_tens}. Extensions to manifold-valued data, multimodal and dynamic problems \cite{TGV_dynamic_PET, TGV_manifold, TGV_video, TGV_electron, TGV_MR_PET, TGV_dynamic_MRI} 
have been proposed, as well. In all of these works, the choice of the scalar parameters $\alpha_{0}, \alpha_{1}$ is  made ``manually'' via a direct grid search. Alternatively, selection schemes relying on a known ground truth $u_{true}$ have been studied; see \cite{bilevellearning, DELOSREYES2016464, TGV_learning2}. The latter approach, however, is primarily of interest when investigating the mere capabilities of TGV regularization.

While there exist automated parameter choice rules for TV regularization, see for instance \cite{hintermuellerPartI} and the references therein, analogous techniques and results for the TGV parameters are very scarce. One of the very few contributions is \cite{TGV_weightedfid} where, however, a spatially varying fidelity weight rather then regularization parameter is computed. Compared to the choice of the regularization weight in TV-based models, the infimal convolution type regularization incorporated into the TGV functional significantly complicates the selection; compare the equivalent definition \eqref{TGV_definition_min} below. Further difficulties arise when these parameters are spatially varying as in \eqref{weighted_TGV_def_intro}. In that case, by appropriately choosing $\balpha=(\alpha_0,\alpha_1)^\top$, one wishes to smoothen homogeneous areas in the image while preserving fine scale details. The overall target is then to not only  select the parameters in order to reduce noise while avoiding oversmoothing, as in the TV case, but also to ensure that the interplay of $\alpha_0$ and $\alpha_1$ will not produce any staircasing.  %In that case the regularization parameters need to adapt to the specific image itself.

% The aim would be here, not only to select them in order for the noise to be reduced while oversmoothing is avoided, as it is in the case of TV, but also to ensure that their combined values will not produce any staircasing. We mention here the works \cite{TGV_learning2, DELOSREYES2016464} where the authors, with means of a bilevel scheme and a training set, learn the optimal parameters for Gaussian denoising with respect to different image quality measures. While this approach is suitable for denoising images of the same noise level with the training set, these training sets, being pairs of ground truths and noisy versions are not always available. Further challenges arise, when one wishes these parameters to be spatially distributed, in order to smooth out large homogeneous areas in the image while preserving fine scale details. In that case the regularization parameters need to adapt to the specific image itself.

For this delicate selection task and inspired by  \cite{hintermuellerPartI, hintermuellerPartII} for TV, in this work we propose a bilevel minimization framework for an automated selection of $\balpha$ in the TGV case. Formally, the setting can be characterized as follows: 
  \begin{equation}\label{bilevel_framework}\left \{
 \begin{aligned}
 &\text{\emph{minimize} a  statistics-based (upper level) objective over } (u,\balpha)\\
  %&\text{\emph{over }} u \text{ and regularization parameters, }  (\alpha_{0},\alpha_{1}),\\[2pt]
  &\text{\emph{subject to }} u \text{ solving \eqref{weightedTGV_min_intro} for a regularization weight } \balpha=(\alpha_{0},\alpha_{1}).
 \end{aligned}\right.
\end{equation}
Note here that the optimization variable $\balpha$ enters the lower level minimization problem \eqref{weightedTGV_min_intro} as a parameter, thus giving rise to $u=u(\balpha)$.
We also mention that this optimization format falls into the general framework which is discussed in our review paper \cite{bilevel_handbook} where the general opportunities and mathematical as well as algorithmic aspects of bilevel optimization in generating structured non-smooth regularization functionals are discussed in detail.

As our statisical set-up parallels the one in \cite{hintermuellerPartI, hintermuellerPartII}, here we resort to the upper level objective proposed in that work. It is based on localized residuals $\mathrm{R}: L^{d/d-1}(\om)\to L^{\infty}(\om)$ with
\begin{equation}\label{loc_res_intro}
\mathrm{R}u(x)=\int_{\om}w(x,y)(Tu-f)^{2}(y)\,dy,
\end{equation}
where $w\in L^{\infty}(\om\times \om)$ with $\int_{\om}\int_{\om}w(x,y)dxdy=1$.
Note that $\mathrm{R}u(x)$ can be interpreted as a local variance keeping in mind that, assuming Gaussian noise of variance $\sigma^{2}$, we have that $\int_{\om}(Tu_{true}-f)^{2}\,dx=\int_{\om}\eta^{2}\,dx=\sigma^{2}|\om|$.  Consequently, if a reconstructed image $u$ is close to $u_{true}$ then it is expected that for every $x\in\om$ the value of $\mathrm{R}u(x)$ will be close to $\sigma^{2}$. Hence it is natural to consider an upper level objective which aims to approximately keep $\mathrm{R}u$ within a corridor $\underline{\sigma}^{2}\leq\sigma^{2}\leq\overline{\sigma}^{2}$ with positive bounds $\underline{\sigma}^{2}$, $\overline{\sigma}^{2}$. This can be achieved by minimizing $F:L^{2}(\om)\to \RR$ with
\begin{equation}\label{upper_level_intro}
F(v):=\frac{1}{2} \int_{\om} \max (v-\overline{\sigma}^{2},0)^{2}dx+\frac{1}{2} \int_{\om} \min (v-\underline{\sigma}^{2},0)^{2}dx.
\end{equation}
%Hence by minimizing the term $F(\mathrm{R}u)$ one aims to force the value of $\mathrm{R}u$ to be close to the variance $\sigma^{2}$ which for our purposes will be considered a known quantity. 
The function $F(\mathrm{R}\cdot)$ is indeed suitable as an upper level objective. This is demonstrated in Figure \ref{fig:psnr_ul_comp}, where we show (in the middle and right plots) the objective values for a series of scalar TGV denoising results and for a variety of parameters $(\alpha_{0},\alpha_{1})$ for the image depicted on the left. Regarding the choices of $\underline{\sigma}, \overline{\sigma}, w$ we refer to Section \ref{sec:numerics}.
  \begin{figure}
 \begin{minipage}[c]{0.2\textwidth}\vspace{1em}
	\begin{minipage}[c]{0.95\textwidth}
	 \includegraphics[width=0.65\textwidth]{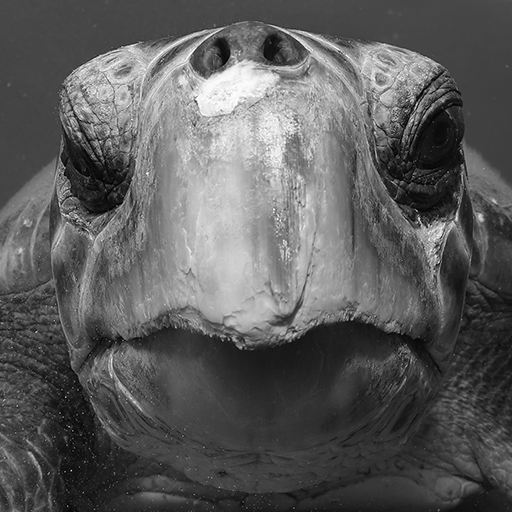}\\[-3pt]\small{clean image}
	\end{minipage}
	
	\begin{minipage}[c]{0.95\textwidth}
	 \includegraphics[width=0.65\textwidth]{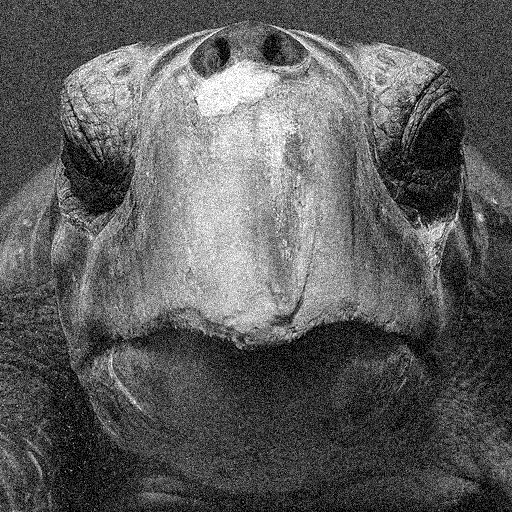}\\[-3pt]\small{noisy image}
	\end{minipage}
\end{minipage}
\begin{minipage}[c]{0.65\textwidth}
\vspace{10pt}
\resizebox{0.49\textwidth}{!}{
\input{psnr_tgv_aniso_v3.tex}}
\resizebox{0.49\textwidth}{!}{
\input{upperlevel_tgv_aniso.tex}}
\end{minipage}
\caption{Suitability of the functional $F( \mathrm{R}\cdot)$ as an upper level objective. Evaluation of $F( \mathrm{R}u)$ where $u$ solves the  TGV  denoising problem \eqref{weightedTGV_min_intro} ($T=Id$), for a variety of  scalar  parameters $(\alpha_{0},\alpha_{1})$ %The functional is minimized  close to the points that maximize the PSNR. The minimum (for $F( \mathrm{R}\cdot)$) and maximum (for PSNR) points are denoted with a bullet.
%We used the values $\underline{\sigma}=0.00798$, $\overline{\sigma}=0.01202$ and $w$ a spatially invariant averaging filter of size $7\times 7$ in the formulas for \eqref{loc_res_intro} and \eqref{upper_level_intro}
}
\label{fig:psnr_ul_comp}		
\end{figure}
Upon inspection of Figure \ref{fig:psnr_ul_comp} we find that the functional $F(\mathrm{R}\cdot)$ is minimized for a pair of scalar parameters $(\alpha_{0},\alpha_{1})$ that is close to the one maximizing the peak-signal-to-noise-ratio (PSNR). Note, however, that in order to truly optimize the PSNR, one would need the ground truth image $u_{true}$, which is course typically not available. In contrast to this, we emphasize that $F(\mathrm{R}\cdot)$ does {\it not} involve any ground truth information. Rather, it only relies on statistical properties of the noise. 

%In order to fully take advantage of the concept of localized residuals, one must allow the regularization parameters to be (non-constant) functions, an approach that will follow here. Hence the specific form of the image reconstruction problem \eqref{variational_intro} which we will consider  is 
%\begin{equation}\label{weightedTGV_min_intro}
%\min_{u\in\bv(\om)} \frac{1}{2}\int_{\om}(Tu-f)^{2}dx+\tgv_{\balpha}^{2}(u)
%\end{equation}
%where 
% \begin{equation}\label{weighted_TGV_def_intro}
%\tgv_{\balpha}^{2}(u)=\sup \left \{\int_{\om} u\,\di^{2} \phi\,dx:\phi\in C_{c}^{\infty}(\om,\mathcal{S}^{d\times d}),\; |\phi(x)|_{r}\le \alpha_{0}(x),\; |\di\phi(x)|_{r}\le \alpha_{1}(x), \text { for all }x\in\om  \right\}
%\end{equation}
%where $\alpha_{0},\alpha_{1}\in C(\overline{\om})$ are bounded away from zero.
%Here $|\cdot|_{r}$ denotes the standard finite dimensional norm, see the next section for details.

For analytical and numerical reasons, rather than having \eqref{weightedTGV_min_intro} as the lower level problem for the bilevel minimization framework \eqref{bilevel_framework}, we
use its Fenchel predual. This yields a bilevel problem which is expressed in terms of dual variables and is equivalent to the one stated in terms of the primal variable $u$. A similar approach was taken in \cite{hintermuellerPartI, hintermuellerPartII} for TV models.
%and substitute it by its corresponding Fenchel predual and in particular a regularized version of it, see Sections \ref{sec:dual_TGV_problem} and \ref{sec:regularized_problems} for more details. 
In this way, one has to treat a more amenable variational inequality of the first kind rather than one of second kind in the primal setting in the constraint system of the resulting bilevel optimization problem. Numerically, one may then utilize very efficient and resolution independent, function space based solution algorithms, like (inexact) semismooth Newton methods \cite{hint_kun}. 
The other option that will also consider here, is to minimize the upper level objective subject to the primal-dual optimality conditions, for which Newton methods can also be applied for their solution, 
see for instance \cite{stadler} for an inexact semismooth Newton solver which operates on the primal-dual optimality conditions for TV regularization.

%The reason for this, is that as the predual problem becomes a constraint of the overall minimization problem, we can handle a problem of obstacle type resulting on a variational inequality of the first kind, instead of the second kind representing the first order optimality condition of the primal problem \eqref{weightedTGV_min_intro}. Moreover as in the case of TV,  the predual problem  is more amenable to efficient, resolution independent, function space solution algorithms, like semismooth Newton methods \cite{hint_kun}, that can solve the problem up to a very high accuracy, in a relatively short amount of time. The predual approach was the one that was also adopted in \cite{hintermuellerPartI, hintermuellerPartII}.
% The other option that will also consider here, is to minimize the upper level objective subject to the primal-dual optimality conditions, for which Newton methods can also be applied for their solution, see for instance \cite{stadler} for the TV case. In both cases, rigorous dualization frameworks need to be established. 
 
Summarizing, this work provides not only a user-friendly and novel hierarchical variational framework for automatic selection of the TGV regularization parameters, but by making these parameters spatially dependent it leads to an overall performance improvement; compare, e.g., the results in Section \ref{sec:numerics}. 

\subsection*{The structure of the paper}
Basic facts on the TGV functional with spacially varying parameters along with functional analytic foundations needed for (pre)dualization are the subjects of Section \ref{sec:weighted_TV_functional}.
%, we recall some basic facts about the TGV functional. We are particularly interested in the role of the parameters at least in the scalar case and their asymptotics. We also remind the reader of the definition and the basic properties of the space $W_{0}^{q}(\di^{2};\om)$ which plays a central role in the dualization. In particular we show that, as in the scalar case, the weighted  $\tgv$ functional can be defined equivalently as in \eqref{weighted_TGV_def_intro} with the dual variables belonging to $W_{0}^{q}(\di^{2};\om)$ instead of $C_{c}^{\infty}(\om,\mathcal{S}^{d\times d})$, something that facilitates the dualization. The proofs here are simple adaptations of the corresponding results for the scalar case.
 %Having established this result, in  
Section \ref{sec:dual_TGV_problem} is concerned with the derivation of the predual problem of \eqref{weightedTGV_min_intro} and the corresponding primal-dual optimality conditions.
 Regularized versions of the primal problem \eqref{weightedTGV_min_intro} and its predual are in the focus of Section \ref{sec:regularized_problems}. %, we introduce a series of regularized problems both for the primal problem \eqref{weightedTGV_min_intro} and its predual, and we derive the corresponding 
Besides respective primal-dual optimality conditions, we study the asymptotic behavior of these problems and their associated solutions under vanishing regularization. It is also argued that every regularized instance can be solved efficiently by employing an (inexact) semismooth Newton method. 
 %
% The regularization of the primal problem consists of $H^{1}$ smoothings of the primal variables as well as Huberization of the corresponding Radon norms, in the spirit of \cite{stadler}. 
%We show convergence of the regularized problem to the original one in a suitable sense. The  regularization of the predual problem, consists of a higher order smoothing of the dual variable as well as a Moreau-Yosida penalization of the corresponding box constraints for that variable and its divergence. This makes the regularized predual problem amenable to be solved using  Newton methods. 
 %
 Section \ref{sec:Bilevel} introduces two bilevel TGV problems %, \emph{Bilevel dual}  and  \emph{Bilevel primal-dual}, where 
 for which the first-order optimality conditions of the predual problem and the first-order primal-dual optimality conditions serve as constraints, respectively. For these problems, based on Karush-Kuhn-Tucker theory in Banach space associated first-order optimality conditions are derived. %We also discuss two Newton-type methods for the solutions of these two sets of optimality conditions as well as aspects of their discretization.
 The numerical solution of the proposed bilevel problems is the subject of Section \ref{sec:numerics}. Finally, the paper ends by a report on extensive numerical tests along with conclusions drawn from theses computational results.
 
% , we  describe the  algorithms that solve the two bilevel TGV problems, based on a projected gradient descent scheme as it was done for TV in \cite{ hintermuellerPartII} and we provide numerical examples for the two versions of bilevel TGV  algorithms in the case of image denoising.

\section{The dual form of the weighted TGV functional}\label{sec:weighted_TV_functional}
\subsection{Total Generalized Variation}
We recall here some basic facts about the TGV functional \eqref{weighted_TGV_def_intro} with constant parameters $\alpha_0,\alpha_1$ and assume throughout that the reader is familiar with the basic concepts of functions of bounded variation (BV); see \cite{AmbrosioBV} for a detailed account.
%We write once  again the definition of TGV for scalar parameters
% \begin{equation}\label{TGV_definition}
% \tgv_{\balpha}^{2}(u)=\sup \left \{\int_{\om} u\,\di^{2} \phi\,dx:\phi\in C_{c}^{\infty}(\om,\mathcal{S}^{d\times d}),\; \|\phi\|_{2,\infty}\le \alpha_{0},\; \|\di\phi\|_{2,\infty}\le \alpha_{1} \right\}.
% \end{equation}
%
For a function  $\phi\in C_{c}^{\infty}(\om,\mathcal{S}^{d\times d})$ the first- and second-order divergences are respectively given by
\[(\di \phi)_{i}=\sum_{j=1}^{d} \frac{\partial \phi_{ij}}{\partial x_{j}}, \;i=1,\ldots,d,\quad \text{and}\quad \di^{2}\phi=\sum_{i=1}^{d}\frac{\partial ^{2} \phi_{ii}}{\partial x_{i}^{2}} +2\sum_{i<j} \frac{\partial ^{2} \phi_{ij}}{\partial x_{i} \partial x_{j}}. \]
When $r=2$ in \eqref{weighted_TGV_def_intro} then we obtain the isotropic version of the TGV functional; otherwise the functional is anisotropic. Among all anisotropic versions, $r=+\infty$ is of particular interest to us, primarily for computational reasons.
  
%Similarly to TV, the anisotropic version of TGV can be defined by substituting the finite dimensional norms $\|\cdot\|_{2,\infty}$ by the corresponding $\|\cdot\|_{\infty,\infty}$. 
%%see for instance Figure \ref{fig:tv_tgv_square} for some comparison.

In \cite{BredValk} it was shown that a function $u\in L^{1}(\om)$ has finite $\tgv$ value if and only if it belongs to $\bv(\om)$. Here $\bv(\om)$ denotes the Banach space of function of bounded variation over $\om$ with associated norm $\|\cdot\|_{\bv(\om)}$. Moreover,
 the bounded generalized variation norm $\|\cdot\|_{\mathrm{BGV}}:=\|\cdot\|_{L^{1}(\om)}+\tgv_{\balpha}^{2}(\cdot)$ is equivalent to $\|\cdot\|_{\bv(\om)}$. Similarly to TV, TGV is a convex functional which is lower semicontinuous with respect to the strong $L^{1}$ convergence. In \cite{bredies2014regularization, BredValk} it is demonstrated that the TGV functional can be equivalently written as 
 \begin{equation}\label{TGV_definition_min}
\tgv_{\balpha}^{2}(u)=\min_{w\in\mathrm{BD}(\om)} \alpha_{1} |Du-w|(\om) +\alpha_{0}|\mathcal{E}w|(\om),
\end{equation}
%\begin{figure}[t!]
%	%\centering
%	\begin{subfigure}[t]{0.20\textwidth}\centering
%	\includegraphics[height=3.05cm]{figures/square/noisy.png}\\% BEST TV PSNR
%	\tiny{Noisy}
%	\end{subfigure}
%	\begin{subfigure}[t]{0.20\textwidth}\centering
%	\includegraphics[height=3.05cm]{figures//square/TGV_01_026_aniso.png}\\
%	\tiny{Anisotropic TGV}
%	\end{subfigure} 
%	\begin{subfigure}[t]{0.20\textwidth}\centering
%	\includegraphics[height=3.05cm]{figures/square/TGV_012_042_iso.png}\\% BEST TGV PSNR
%	\tiny{Isotropic TGV}
%	\end{subfigure}
%	\caption{Comparison between  the isotropic and anisotropic versions of TGV. The (scalar) regularization parameters have been manually adjusted for optimal SSIM}
%	\label{fig:tv_tgv_square}
%\end{figure}
where $\mathrm{BD}(\om)$ is the space of functions of bounded deformation, with $\mathcal{E}$ denoting the distributional symmetrized gradient \cite{bd}.
The asymptotical behavior of  the TGV model in image restoration with respect to  scalars $\alpha_{0},\alpha_{1}$ was studied in \cite{tgv_asymptotic}; see also in \cite{valkonen_jump_2}. %\cite{bredies2014regularization, valkonen_jump_2}. 
For instance, when  $T=Id$ and either $\alpha_{0}$ or $\alpha_{1}$ converges to zero, then the corresponding solutions of \eqref{weightedTGV_min_intro} converge (weakly$^\ast$ in $\bv(\om)$) to $f$. When both of the parameters are sent to infinity, then the solutions converge weakly$^\ast$ to the $L^{2}$-linear regression solution for $f$.
We further note that the set of affine functions constitutes the kernel of the TGV functional.

There exist combinations of $\alpha_{0},\alpha_{1}$ such that $\tgv_{\balpha}(u)=\alpha_{1}\tv(u)$. This happens for specific functions $u$, and in general one can show that there exists a constant $C>0$ such that if $\alpha_{0}/\alpha_{1}>C$, then the $\tgv$ value does not depend on $\alpha_{0}$ and, up to an affine correction, it is equivalent to $\tv$. In that case the reconstructed images still suffer from a kind of (affine) staircasing effect  \cite{tgv_asymptotic}.
%We summarize in the following, see also \cite{Papafitsoros_Bredies, tgv_asymptotic}.
%
% \newtheorem{Wdiv}{Theorem}[section]
% \begin{Wdiv}\label{lbl:Wdiv}
%There exists a constant $C$ depending only on the domain $\om$ such that if with $\balpha=(\alpha_{0},\alpha_{1})$ are such that  $\alpha_{0}/\alpha_{1}>C$ then
%\begin{equation}\label{TGV_almostTV}
%\tgv_{\balpha}^{2}(u)= \alpha_{1} |Du-m_{\mathcal{E}}(\nabla u)|(\om),\quad \text{for all }u\in \bv(\om),  
%\end{equation}
%where for a function $g\in L^{1}(\om,\RR^{d})$ we define
%\begin{equation}\label{mE}
%m_{\mathcal{E}}(g):=\underset{w\in \mathrm{Ker}\mathcal{E}}{\operatorname{argmin}}\; \|g-w\|_{L^{1}(\om,\RR^{d})}.
%\end{equation}
%\end{Wdiv}
%Note here that the kernel of the symmerized gradient $ \mathrm{Ker}\mathcal{E}$ consists exactly of all the functions of the form $r(x)=Ax+b$ where $b\in \RR^{d}$ and $A\in \RR^{d\times d}$ is a skew symmetric matrix. Theorem \ref{lbl:Wdiv} thus says that for a large ratio $\alpha_{0}/\alpha_{1}$ $\tgv$ is almost equivalent to $\tv$ up to an affine correction. In practice this means that for such a combination of parameters the reconstructed images still suffer from a type of (affine) staircasing effect, see \cite{tgv_asymptotic}. 
	
The fine structure of TGV reconstructions has been studied analytically mainly in dimension one in \cite{TGVbregman, BrediesL1, Papafitsoros_Bredies, poschl2013exact}. Under some additional regularity assumptions (compare \cite{valkonen_jump_2}) it can be shown that for TGV denoising the jump set of the solution is essentially contained in the jump set of the data; see \cite{caselles2007discontinuity} for the TV case.
%, i.e., whether the jump set of the solution is essentially contained in the jump set of the data, also holds for TGV denoising is rather unknown but nevertheless it holds 
\subsection{The space $W_{0}^{q}(\di^{2};\om)$}
Next we introduce several function spaces which will be useful in our subsequent development.
\newtheorem{Wdiv}{Definition}[section]
%\begin{Wdiv}\label{lbl:Wdiv}
For this purpose, let $1\le q\le \infty$ and $p\in L^{q}(\om,\RR^{d})$. Recall that $\di p\in L^{q}(\om)$ if there exists $w\in L^{q}(\om)$ such that
\[\int_{\om} \nabla \phi \cdot p\,dx = -\int_{\om}\phi w\,dx, \quad \text{for all }\phi\in C_{c}^{\infty}(\om).\]
Based on this first-order divergence, we define the Banach space
\[W^{q}(\di;\om):= \left \{p\in L^{q}(\om,\RR^{d}):\;\di p\in L^{q}(\om) \right \},\]
endowed with the norm $\|p\|_{W^{q}(\di;\om)}^{q}:=\|p\|_{L^{q}(\om,\RR^{d})}^{q}+\|\di p\|_{L^{q}(\om)}^{q}$.  Similarly one obtains the Banach space $W^{q}(\di^{2};\om)$ as the space of all functions $p\in L^{q}(\om,\mathcal{S}^{d\times d})$ whose first- and second-order divergences, $\di p$ and $\di^{2} p$, respectively, belong to $L^{q}(\om)$. Note that $\di^{2}p\in L^{q}(\om)$ if there exists a function $v\in L^{q}(\om)$ such that
\[\int_{\om} \nabla \phi\cdot \di p\,dx=-\int_{\om} \phi v\,dx, \quad \text{for all }\phi\in C_{c}^{\infty}(\om).\]
This space is equipped with the norm $\|p\|_{W^{q}(\di^{2};\om)}^{q}:=\|p\|_{L^{q}(\om)}^{q}+\|\di p\|_{L^{q}(\om,\RR^{d})}^{q}+\|\di^{2} p\|_{L^{q}(\om)}^{q}$.
%\end{Wdiv}
%
%It can be shown using density of $C_{c}^{\infty}$ functions in $L^{q}$ that the first and second order divergences are unique. Moreover both spaces above are Banach under these norms. 
We refer to \cite{TGV_decompression_part1} for a more general definition of these spaces. Note that when $q=2$ these spaces are Hilbertian and then the standard notation is $H(\di;\om)$ and $H(\di^{2};\om)$; see \cite{Girault}. The Banach spaces $W_{0}^{q}(\di;\om)$ and $W_{0}^{q}(\di^{2};\om)$ are defined as 
%\begin{align*}
\[
W_{0}^{q}(\di;\om)= \overline{C_{c}^{\infty}(\om,\RR^{d})}^{\|\cdot\|_{W^{q}(\di;\om)}},\qquad
W_{0}^{q}(\di^{2};\om)= \overline{C_{c}^{\infty}(\om,\mathcal{S}^{d\times d})}^{\|\cdot\|_{W^{q}(\di^{2};\om)}}.
\]
%\end{align*}
Using the definitions above, the following integration by parts formulae hold true:
\begin{align}
\int_{\om} \nabla \phi\cdot p\,dx&=-\int_{\om} \phi \, \di p \, dx,\quad \text{for all }p\in W_{0}^{q}(\di;\om),\; \phi\in C^{\infty}(\overline{\om},\RR),\label{ibp_1}\\
\int_{\om} E \phi\cdot p\,dx&=-\int_{\om} \phi \cdot \, \di p \, dx,\quad \text{for all }p\in W_{0}^{q}(\di^{2};\om),\; \phi\in C^{\infty}(\overline{\om},\RR^{d}),\label{ibp_2}\\
\int_{\om} \nabla \phi\cdot \di p\,dx&=-\int_{\om} \phi  \, \di^{2} p \, dx,\quad \text{for all }p\in W_{0}^{q}(\di^{2};\om),\; \phi\in C^{\infty}(\overline{\om},\RR),\label{ibp_3}
\end{align}
with $E\phi$ denoting the symmetrized gradient of $\phi$.

\subsection{Weighted TGV}
Throughout the remainder of ths work we use the weighted TGV functional \eqref{weighted_TGV_def_intro} with
%
%We define the weighted TGV functional as follows: 
%\begin{equation}\label{weighted_TGV_def}
%\tgv_{\balpha}^{2}(u)=\sup \left \{\int_{\om} u\,\di^{2} \phi\,dx:\phi\in C_{c}^{\infty}(\om,\mathcal{S}^{d\times d}),\; |\phi(x)|_{r}\le \alpha_{0}(x),\; |\di\phi(x)|_{r}\le \alpha_{1}(x), \text { for all }x\in\om  \right\}.
%\end{equation}
%Here  $\balpha=(\alpha_{0},\alpha_{1})$ with 
$\alpha_{0},\alpha_{1}\in C(\overline{\Omega})$ and $\alpha_{0}(x),\alpha_{1}(x)>\underline{\alpha}>0$, $\underline{\alpha}\in \RR$, $x\in\om$.  %Note that 
Concerning $|\cdot|_{r}$ %denotes the finite dimensional $r$-norm for $1\le r\le \infty$, and 
let $r^{\ast}$ with $1/r+1/r^{\ast}=1$ and the obvious definitions for $r=1,\infty$.

We will show that the space $C_{c}^{\infty}(\om,\mathcal{S}^{d\times d})$ in \eqref{weighted_TGV_def_intro} can be substituted by  $W_{0}^{d}(\di^{2};\om)$. This fact will be instrumental when deriving the predual of the TGV minimization problem.
For this we need  the following result, which involves the Banach space of functions of bounded deformation here denoted by $\mathrm{BD}(\om)$; see, e.g., \cite{temam1985mathematical} for more details.
\newtheorem{prop_weighted_TGV_min}[Wdiv]{Proposition}
\begin{prop_weighted_TGV_min}\label{lbl:prop_weighted_TGV_min}
Then weighted $\tgv_{\balpha}^{2}$ functional \eqref{weighted_TGV_def_intro} admits the equivalent expression
\begin{equation}\label{weighted_TGV_min}
\tgv_{\balpha}^{2}(u)=\min_{w\in\mathrm{BD}(\om)} \int_{\om} \alpha_{1}\,d |Du-w|_{r^{\ast}}+ \int_{\om}\alpha_{0}\, d|\mathcal{E}w|_{r^{\ast}}.
\end{equation}
\end{prop_weighted_TGV_min}
\begin{proof}
The proof is analogous to the one for the scalar TGV functional; see for instance \cite[Proposition 2.8]{TGV_decompression_part1} or \cite[Theorem 3.5]{bredies2014regularization}. Here, we highlight only the significant steps. Indeed, given $u\in L^{1}(\om)$, the idea is to define
\begin{align*}
U&=C_{0}^{1}(\om,\RR^{d})\times C_{0}^{2}(\om;\mathcal{S}^{d\times d}),\quad V=C_{0}^{1}(\om,\RR^{d}),\\
\Lambda&:U\to V,\quad \Lambda (u_{1},u_{2})=-u_{1}-\di u_{2},\\
F_{1}&: U\to \overline{\RR},\quad F_{1}(u_{1},u_{2})=-\int_{\om} u\,\di u_{1}+ \mathcal{I}_{\{|\cdot(x)|_{r}\le \alpha_{1}(x)\}}(u_{1})+ \mathcal{I}_{\{|\cdot(x)|_{r}\le \alpha_{0}(x)\}}(u_{2}),\\
F_{2}&: V\to \overline{\RR},\quad F_{2}(v)=\mathcal{I}_{\{0\}}(v).
\end{align*}
Now, after realizing that 
\begin{equation}\label{primal_weighted_TGV}
\tgv_{\balpha}^{2}(u)=\sup_{(u_{1},u_{2})\in U} - F_{1}(u_{1},u_{2})-F_{2}(\Lambda (u_{1},u_{2})),
\end{equation}
the proof proceeds by next showing that the dual problem of \eqref{primal_weighted_TGV} is equivalent to \eqref{weighted_TGV_min} and then applying the Fenchel duality result \cite{ekeland1999convex}. The only subtle point is the following density result which is required in order to show that \eqref{primal_weighted_TGV} is indeed equal to \eqref{weighted_TGV_def_intro}. In fact, it suffices to show that 
\begin{equation}\label{density_Cc_C0_weighted}
\begin{aligned}
&\overline{\left \{\phi\in C_{c}^{\infty}(\om,\mathcal{S}^{d\times d}):\; |\phi(x)|_{r}\le \alpha_{0}(x),\; |\di\phi(x)|_{r}\le \alpha_{1}(x), \text { for all }x\in\om \right\}}^{\|\cdot\|_{C_{0}^{2}}}\\
&=\left \{\psi\in C_{0}^{2}(\om,\mathcal{S}^{d\times d}):\; |\psi(x)|_{r}\le \alpha_{0}(x),\; |\di\psi(x)|_{r}\le \alpha_{1}(x), \text { for all }x\in\om \right\}.
\end{aligned}
\end{equation}
Indeed let $\psi$ belong to the second set in \eqref{density_Cc_C0_weighted}, and let $\epsilon>0$. Choose $0<\lambda_{\epsilon}<1$ such that 
\begin{equation}\label{lambdaeps_psi}
\|\psi-\lambda_{\epsilon}\psi\|_{C_{0}^{2}}<\epsilon/2.
\end{equation}
 Since $\alpha_{0}$ and $\alpha_{1}$ are continuous and bounded away from zero there exists $\alpha_{\epsilon}>0$,  smaller than the minimum of $\alpha_{0},\alpha_{1}$, such that 
 \[|\lambda_{\epsilon}\psi(x)|_{r}\le \alpha_{0}(x)-\alpha_{\epsilon},\quad |\di \lambda_{\epsilon}\psi(x)|_{r}\le \alpha_{1}(x)-\alpha_{\epsilon},\quad \text{ for all }x\in\om.\]
From  standard density properties there exists a function $\phi_{\epsilon}\in C_{c}^{\infty}(\om,\mathcal{S}^{d\times d})$ such that the following conditions hold for all $x\in\om$:
 \begin{equation}\label{phieps}
 \|\phi_{\epsilon}-\lambda_{\epsilon}\psi\|_{C_{0}^{2}}<  \epsilon/2,\quad |\phi_{\epsilon}(x)-\lambda_{\epsilon}\psi(x)|_{r}\le \alpha_{\epsilon}/2, \quad |\di \phi_{\epsilon}(x)-\di \lambda_{\epsilon}\psi(x)|_{r}\le \alpha_{\epsilon}/2,
 \end{equation}
 which implies
 \begin{equation}\label{a_epsi_over2}
 |\phi_{\epsilon}(x)|_{r}\le \alpha_{0}(x)-\alpha_{\epsilon}/2,\quad |\di \phi_{\epsilon}(x)|_{r}\le \alpha_{1}(x)-\alpha_{\epsilon}/2,\quad \text{ for all }x\in\om.
 \end{equation}
 Then, from \eqref{a_epsi_over2} it follows that $\phi_{\epsilon}$ belongs to the first set in \eqref{density_Cc_C0_weighted} and from \eqref{lambdaeps_psi} and \eqref{phieps} we get that $\|\psi-\phi_{\epsilon}\|_{C_{0}^{2}}< \epsilon$.
\end{proof}

Now we are ready to establish the density result needed for dualization. For the sake of the flow of presentation we defer the proof, which parallels the one of \cite[Proposition 3.3]{TGV_decompression_part1}, to the appendix; see Appendix \ref{sec:app.proof}. Below ``a.e.'' stands for ``almost every'' with respect to the Lebesgue measure.
\newtheorem{weighted_TGV_Wdiv}[Wdiv]{Proposition}
\begin{weighted_TGV_Wdiv}\label{lbl:weighted_TGV_Wdiv}
Let $u\in L^{d/d-1}(\om)$, $\balpha=(\alpha_{0},\alpha_{1})$ with $\alpha_{0},\alpha_{1}\in C(\overline{\Omega})$ and $\alpha_{0},\alpha_{1}>\underline{\alpha}>0$. Then  the weighted TGV functional \eqref{weighted_TGV_def_intro}
%\begin{equation}\label{weighted_TGV_phi}
%\tgv_{\balpha}^{2}(u)=\sup \left \{\int_{\om} u\,\di^{2} \phi\,dx:\phi\in C_{c}^{\infty}(\om,\mathcal{S}^{d\times d}),\; |\phi(x)|_{r}\le \alpha_{0}(x),\; |\di\phi|_{r}\le \alpha_{1}(x), \text { for all }x\in\om  \right\}
%\end{equation}
can be equivalently written as 
\begin{equation}\label{weighted_TGV_p}
\begin{split}
\mathrm{TGV}_{\balpha}^{2}(u)=\sup \Big\{\int_{\om} u\,\di^{2} p\,dx:p  \in  W_{0}^{d}(\di^{2};\om),\; &|p(x)|_{r}\le \alpha_{0}(x),\; \\&|\di p (x)|_{r}\le \alpha_{1}(x), \text { for a.e. }x\in\om  \Big\}.\end{split}
\end{equation}
\end{weighted_TGV_Wdiv}

\noindent
\emph{Remark}: By slightly amending the proof of Proposition \ref{lbl:weighted_TGV_Wdiv} one can also show that 
\begin{equation}\label{density_weighted_TGV_H0}
\overline{C_{\balpha}}^{L^{2}(\om)}=K_{\balpha},
\end{equation}
where $K_{\balpha}$ is defined over $H_{0}(\di^{2};\om)$ rather than $W_{0}^{d}(\di ^{2};\om)$.

\section{The predual weighted TGV problem}\label{sec:dual_TGV_problem}

%While the predual problem for TV regularization has been studied \cite{hint_kun}, this is not the case for   TGV.  We do that  for the 

Now we study the predual problem for the weighted TGV model with continuous weights, i.e.,  we use the regularization functional \eqref{weighted_TGV_def_intro} or equivalently  \eqref{weighted_TGV_p}. For $T\in \mathcal{L}(L^{d/d-1}(\om),L^{2}(\om))$ we assume for simplicity that $B:=T^{\ast}T$ is invertible and define $\|v\|_{B}^{2}=\int_{\om} v B^{-1}v$, which induces a norm in $L^{d}(\om)$; compare \cite{hint_kun}.

\newtheorem{quadfid_duality_TGV}[Wdiv]{Proposition}
\begin{quadfid_duality_TGV}\label{lbl:quadfid_duality_TGV}
Let $f\in L^{2}(\om)$, $\balpha=(\alpha_{0},\alpha_{1})$, $\alpha_{0},\alpha_{1}\in C(\overline{\om})$ with $\alpha_{0},\alpha_{1}>\underline{\alpha}>0$ and  $T\in \mathcal{L}(L^{d/d-1}(\om),L^{2}(\om))$ with $T^{\ast}T$ invertible. Then there exists a solution to the primal problem
\begin{equation}\label{duality_TGV_primal}
\text{minimize}\quad \frac{1}{2}\|Tu -f\|_{L^{2}(\om)}^{2}+\tgv_{\balpha}^{2}(u)\quad\text{over }u\in\bv(\om),
\end{equation}
as well as to its predual problem
\begin{equation}\label{duality_TGV_predual}
\begin{split}
&\text{minimize}\quad \frac{1}{2}\|T^{\ast}f-\di^{2}p\|_{B}^{2}-\frac{1}{2}\|f\|_{L^{2}}^{2}\quad\text{over }p\in W_{0}^{d}(\di^{2};\om)\\
&\text{subject to}\quad |p(x)|_{r}\le\alpha_{0}(x),\; |\di p(x)|_{r}\le \alpha_{1}(x),\text{ for a.e. }x\in\om, 
\end{split}
\end{equation}
and there is no duality gap, i.e., the primal and predual optimal objective values are equal.
Moreover, the solutions $u$ and $p$ of these problems satisfy
\begin{align}
&Bu=T^{\ast}f-\di^{2}p.\label{TGV_duality_opt1}
\end{align}
\end{quadfid_duality_TGV}

\begin{proof}
We set $U=W_{0}^{d}(\di^{2};\om)$, $V=L^{d}(\om)$, $\Lambda: U\to V$ with $\Lambda p=\di ^{2}p$, and also $F_{1}:U\to \overline{\RR}$ and $F_{2}: V\to \overline{\RR}$ with 
\begin{align}
F_{1}(p)&=\mathcal{I}_{\{|\cdot(x)|_{r}\le\alpha_{0}(x),\text{ for a.e. }x\}}(p)+\mathcal{I}_{\{|\di \cdot(x)|_{r}\le\alpha_{1}(x),\text{ for a.e. }x\}}(p),\\
F_{2}(\psi)&= \frac{1}{2} \|T^{\ast}f-\psi\|_{B}^{2}-\frac{1}{2} \|f\|_{L^{2}(\om)}^{2}.
\end{align}
Here, $\mathcal{I}_S(\cdot)$ denotes the indicator function of a set $S$.
Immediately one gets that
\begin{equation}\label{equality_TGV_predual}
\inf_{p\in U} F_{1}(p)+F_{2}(\Lambda p)=\inf_{\substack{p\in W_{0}^{d}(\di^{2};\om)\\ |p(x)|_{r}\le\alpha_{0}(x) \\  |\di p(x)|_{r}\le \alpha_{1}(x)}} \frac{1}{2} \|T^{\ast}f-\di^{2}p\|_{B}^{2}-\frac{1}{2} \|f\|_{L^{2}(\om)}^{2}.
\end{equation}
The problem in \eqref{equality_TGV_predual} admits a solution. Indeed, first observe that the objective is bounded from below. Then note that since $\frac{1}{2}\|T\cdot-f\|_{L^{2}(\om)}^{2}$ is continuous at $0\in L^{d/d-1}(\om)$, its convex conjugate (see \cite{ekeland1999convex} for a general definition) which is equal to $ \frac{1}{2} \|T^{\ast}f+\cdot\|_{B}^{2}-\frac{1}{2} \|f\|_{L^{2}(\om)}^{2}$ is coercive in $L^{d}(\om)$; see \cite[Theorem 4.4.10]{Borwein}. Hence, any infimizing sequence $(p_{n})_{n\in\NN}$ is bounded in $W_{0}^{d}(\di^{2};\om)$, and thus there exist an (unrelabeled) subsequence and $p\in W^{d}(\di^{2};\om)$ such that $p_{n}\rightharpoonup p$, $\di p_{n} \rightharpoonup \di p$ and $\di^{2} p_{n} \rightharpoonup \di^{2} p$ weakly in $L^{d}$. We also have that $p$ is a feasible point since the set 
\[\left \{(h,\di h,\di^{2}h):\; h\in W_{0}^{d}(\di^{2};\om), \; |h(x)|_{r}\le\alpha_{0}(x),\; |\di h(x)|_{r}\le \alpha_{1}(x),\text{ for a.e. }x\in\om  \right \},\]
is  weakly closed. Then  $p$ is a minimizer of \eqref{equality_TGV_predual} as $ \frac{1}{2} \|T^{\ast}f-\cdot\|_{B}^{2}$ is weakly lower semicontinuous in $L^{d}(\om)$. 

We now calculate the expression $F_{1}^{\ast}(\Lambda^{\ast} u)+F_{2}^{\ast}(-u)$ for $u\in Y^{\ast}=L^{d/d-1}(\om)$. As before one verifies by direct computation that $F_{2}^{\ast}(-u)=\frac{1}{2}\|Tu-f\|_{L^{2}(\om)}^{2}$. Moreover,
\begin{align*}
F_{1}^{\ast}(\Lambda^{\ast} u)
&=\sup_{p\in X} \{\langle \Lambda^{\ast} u,p \rangle_{X^{\ast},X}- F_{1}(p)\}
=\sup_{p\in X}  \{\langle u,\Lambda p \rangle_{L^{d/d-1}(\om),L^{d}(\om)} -F_{1}(p)\}\\
&=\sup_{\substack{p\in W_{0}^{d}(\di^{2};\om)\\ |p(x)|_{r}\le\alpha_{0}(x) \\  |\di p(x)|_{r}\le \alpha_{1}(x)}} \int_{\om}\di^{2}p\,dx
=\tgv_{\balpha}^{2}(u).
\end{align*}
In order to prove that there is no duality gap, it suffices to show that the set $\bigcup_{\lambda\ge 0} \lambda (\mathrm{dom}(F_{2})-\Lambda(\mathrm{dom}(F_{1})))$ is a closed subspace of $V$. Then the so-called Attouch-Brezis condition is satisfied; see \cite{attouch1986duality}. It is immediate to see that $\mathrm{dom}(F_{2})=L^{d}(\om)$, and hence the condition holds true. Thus, we also get existence of a solution for the primal problem \eqref{duality_TGV_primal}.
  Finally \eqref{TGV_duality_opt1} follows from the optimality condition (Euler-Lagrange system) that corresponds to $\Lambda p \in \partial F_{2}^{\ast}(-u)$.
\end{proof}

The assumptions on $T$ in the above proposition are invoked throughout the rest of this work. In the special case when $T=Id$ (corresponding to image denoising), then we can only get existence of a solution to the predual problem in the Hilbert space $H_{0}(\di^{2};\om)$. The proof of this fact is similar to the one above.

The primal-dual optimality conditions for the problems \eqref{duality_TGV_primal} and \eqref{duality_TGV_predual} read
\begin{align}
p&\in \partial F_{1}^{\ast}(\Lambda^{\ast} u),\label{opt_cond_1}\\
\Lambda p&\in \partial F_{2}^{\ast}(-u),\label{opt_cond_2}
\end{align}
and we note once again that \eqref{TGV_duality_opt1} corresponds to \eqref{opt_cond_2} with $F_{2}$ and $\Lambda$ as in the proof of Proposition \ref{duality_TGV_primal}. Instead of making the optimality condition that corresponds to \eqref{opt_cond_1} explicit, we are interested in the analogous optimality conditions written in the variables $u$ and $w$ of the equivalent primal weighted TGV problem
\begin{equation}\label{TGV_primal_u_w}
\min_{\substack{u\in\bv(\om)\\ w\in \mathrm{BD}(\om)}} \frac{1}{2} \|Tu-f\|_{L^{2}(\om)}^{2}+
\int_{\om}\alpha_{1}\, d|Du-w|_{r^{\ast}} + \int_{\om} \alpha_{0}\, d|\mathcal{E}w|_{r^{\ast}}.
\end{equation}
For this purpose note first that the predual problem \eqref{duality_TGV_predual} can be equivalently written as
\begin{equation}\label{TGV_dual_q_p}\left\{
\begin{split}
&\text{minimize}\quad \frac{1}{2} \|T^{\ast}f+\di q\|_{B}^{2}-\frac{1}{2}\|f\|_{L^{2}(\om)}^{2}\quad \text{ over } (q,p)\in W_{0}^{d}(\di;\om)\times W_{0}^{d}(\di^{2},\om),\\
&\text{subject to} -\di p=q,
\, |p(x)|_{r}\le \alpha_{0}(x), \, |q(x)|_{r}\le \alpha_{1}(x), \, \text{for a.e. }x\in\om.
\end{split}\right.
\end{equation}
Then the solutions of the above two problems can be characterized as follows.
\newtheorem{optimality_u_w_q_p}[Wdiv]{Proposition}
\begin{optimality_u_w_q_p}\label{lbl:optimality_u_w_q_p}
The pair $(p,q)\in W_{0}^{d}(\di^{2};\om)\times W_{0}^{d}(\di;\om)$ is a solution to \eqref{TGV_dual_q_p}, and $(w,u)\in  \mathrm{BD}(\om)\times\bv(\om) $ is a solution to \eqref{TGV_primal_u_w} if and only if the following optimality conditions are satisfied:
\begin{align}
&Bu=T^{\ast}f+\di q,\label{opt_1}\\
&q=-\di p, \label{opt_2}\\
&|q(x)|_{r}\le \alpha_{1}(x) \text{ for a.e. } x\in\om \label{opt_3}\\
& \text{and }  \langle Du-w, \tilde{q}-q\rangle \le 0 \text{ for every } \tilde{q}\in W_{0}^{d}(\di;\om), \text{ with } |\tilde{q}(x)|_{r}\le \alpha_{1}(x) \text{ for a.e. } x\in\om,\nonumber\\
&|p(x)|_{r}\le \alpha_{0}(x)  \text{ for a.e. } x\in\om \label{opt_4}\\
&\text{ and } \langle \mathcal{E}w,\tilde{p}-p\rangle \le 0 \text{ for every }  \tilde{p}\in W_{0}^{d}(\di^{2};\om)  \text{ with } |\tilde{p}(x)|_{r}\le \alpha_{0}(x) \text{ for a.e. } x\in\om.\nonumber
\end{align}
\end{optimality_u_w_q_p}

\begin{proof}
Define $X=(X_{1},X_{2})=W_{0}^{d}(\di^{2},\om)\times W_{0}^{d}(\di,\om)$, $Y=(Y_{1},Y_{2})=W_{0}^{d}(\di;\om)\times L^{d}(\om)$, $\Lambda: X\to Y$ with $\Lambda(p,q)=(q+\di p,\di q)$, and $F_{1}: X\to \overline{\RR}$, $F_{2}: Y\to \overline{\RR}$ with 
\begin{align}
F_{1}(p,q)&=\mathcal{I}_{\{|\cdot(x)|_{r}\le\alpha_{0}(x),\text{ for a.e. }x\}}(p)+\mathcal{I}_{\{| \cdot(x)|_{r}\le\alpha_{1}(x),\text{ for a.e. }x\}}(q),\\
F_{2}(\phi,\psi)&=\mathcal{I}_{\{0\}}(\phi)+ \frac{1}{2} \|T^{\ast}f+\psi\|_{B}^{2}-\frac{1}{2} \|f\|_{L^{2}(\om)}^{2}.
\end{align}
One checks immediately that $\min_{(p,q)\in X} F_{1}(p,q)+F_{2}(\Lambda (p,q))$ corresponds to \eqref{TGV_dual_q_p}
%\begin{equation}\label{equality_TGV_predual_q_p}
%\min_{(p,q)\in X} F_{1}(p,q)+F_{2}(\Lambda (p,q))
%=\min_{\substack{p\in W_{0}^{d}(\di^{2};\om)\\q\in W_{0}^{d}(\di;\om),\,q=-\di p \\ |p(x)|_{r}\le\alpha_{0}(x) \\  |q (x)|_{r}\le \alpha_{1}(x)}} \frac{1}{2} \|T^{\ast}f+\di q\|_{B}^{2}-\frac{1}{2} \|f\|_{L^{2}(\om)}^{2}.
%\end{equation}
with the dual problem reading
%\begin{equation}\label{dual_TGV_w_v}
$
\min_{(w,u)\in Y^{\ast}} F_{1}^{\ast}(-\Lambda^{\ast}(w,u))+ F_{2}^{\ast}(w,u).
$
%\end{equation}
Observe that since
\[-\langle \Lambda^{\ast}(w,u), (p,q) \rangle_{X^{\ast},X}=-\langle (w,u),\Lambda(p,q)\rangle_{Y^{\ast},Y}=- \langle w,\di p \rangle_{Y_{1}^{\ast},Y_{1}} - \langle w,q \rangle_{Y_{1}^{\ast}, Y_{1}} -\langle u,\di q \rangle_{Y_{2}^{\ast}, Y_{2}},\]
 we have 
\[F_{1}^{\ast}(-\Lambda^{\ast}(w,u))
= \sup_{\substack{p\in W_{0}^{d}(\di^{2};\om) \\ |p(x)|_{r}\le \alpha_{0}(x)}} -\langle w,\di p \rangle_{Y_{1}^{\ast},Y_{1}}
+\sup_{\substack{q\in W_{0}^{d}(\di;\om) \\ |q(x)|_{r}\le \alpha_{1}(x)}}
- \langle w,q \rangle_{Y_{1}^{\ast}, Y_{1}} -\langle u,\di q \rangle_{Y_{2}^{\ast}, Y_{2}}.\]

Note that the suprema above are always greater or equal to the corresponding suprema over $C_{c}^{\infty}(\om,\mathcal{S}^{d\times d})\subset W_{0}^{d}(\di^{2};\om)$ and $C_{c}^{\infty}(\om,\RR^{d})\subset W_{0}^{d}(\di;\om)$.
Moreover, as we focus on a minimization problem, we are interesing in those $(w,u)\in Y^{\ast}$ that render the suprema finite. This implies in particular that $w$ has a distributional derivative $\mathcal{E}w$ with bounded Radon norm, and hence it is a Radon measure. It follows that $w\in L^{1}(\om,\RR^{d})$ yielding $w\in \mathrm{BD}(\om)$; see \cite{bredies2014regularization}. This also implies $\langle w, \di p \rangle_{Y_{1}^{\ast}, Y_{1}}=\langle w, \di p \rangle_{L^{d}(\om)^{\ast}, L^{d}(\om)}$ and similarly   $\langle w, q \rangle_{Y_{1}^{\ast}, Y_{1}}=\langle w, q \rangle_{L^{d}(\om)^{\ast}, L^{d}(\om)}$. Using now density results analogous to \eqref{density_weighted_TGV} we have
\begin{align*}
F_{1}^{\ast}(-\Lambda^{\ast}(w,u))
&= \sup_{\substack{p\in W_{0}^{d}(\di^{2};\om) \\ |p(x)|_{r}\le \alpha_{0}(x)}} -\langle w,\di p \rangle_{L^{d}(\om)^{\ast}, L^{d}(\om)}\\&\qquad
+\sup_{\substack{q\in W_{0}^{d}(\di;\om) \\ |q(x)|_{r}\le \alpha_{1}(x)}}
- \langle w,q \rangle_{L^{d}(\om)^{\ast}, L^{d}(\om)} -\langle u,\di q \rangle_{L^{d}(\om)^{\ast}, L^{d}(\om)}\\
&= \sup_{\substack{\phi\in C_{c}^{\infty}(\om,\mathcal{S}^{d\times d}) \\ |\phi(x)|_{r}\le \alpha_{0}(x)}} -\langle w,\di \phi \rangle_{L^{d}(\om)^{\ast}, L^{d}(\om)}\\&\qquad
+\sup_{\substack{\psi\in C_{c}^{\infty}(\om,\RR^{d})\\ |\psi(x)|_{r}\le \alpha_{1}(x)}}
- \langle w,\psi \rangle_{L^{d}(\om)^{\ast}, L^{d}(\om)} -\langle u,\di \psi \rangle_{L^{d}(\om)^{\ast}, L^{d}(\om)}\\
&=\sup_{\substack{\phi\in C_{c}^{\infty}(\om,\mathcal{S}^{d\times d}) \\ |\phi(x)|_{r}\le \alpha_{0}(x)}}
\langle \mathcal{E}w, \phi\rangle 
+\sup_{\substack{\psi\in C_{c}^{\infty}(\om,\RR^{d})\\ |\psi(x)|_{r}\le \alpha_{1}(x)}}\langle Du-w,\psi \rangle,  \\
&=\int_{\om} \alpha_{0}\, d|\mathcal{E}w|_{r^{\ast}} + \int_{\om} \alpha_{1}\, d|Du-w|_{r^{\ast}}.
\end{align*}
Here we used the fact that since the distribution $Du-w$ has a finite Radon norm, it can be represented by an $\RR^{d}$-valued finite Radon measure and in particular by $u\in\bv(\om)$. Furthermore, as in the proof of Proposition \ref{lbl:quadfid_duality_TGV} we have $F_{2}^{\ast}(w,u)=\frac{1}{2}\|Tu-f\|_{L^{2}(\om)}^{2}$.
%\begin{align*}
%F_{2}^{\ast}(w,u)
%&=\sup_{(\phi,\psi)\in Y} \langle w,\phi \rangle_{Y_{1}^{\ast}, Y_{1}} + \langle u,\psi \rangle _{Y_{2}^{\ast}, Y_{2}}- F_{2}(\phi,\psi)\\
%&= \langle u,\psi \rangle _{L^{d}(\om)^{\ast}, L^{d}(\om)}-\left (\frac{1}{2} \|T^{\ast}f+\psi\|_{B}^{2}-\frac{1}{2} \|f\|_{L^{2}(\om)}^{2} \right )\\
%&=\frac{1}{2}\|Tu-f\|_{L^{2}(\om)}^{2},
%\end{align*}
%similarly to the proof of Proposition \ref{lbl:quadfid_duality_TGV}. 

The fact that there is no duality gap is ensured by Propositions \ref{lbl:prop_weighted_TGV_min}, \ref{lbl:weighted_TGV_Wdiv} and \ref{lbl:quadfid_duality_TGV}. We now turn our attention to the optimality conditions 
\begin{align}
(p,q)&\in \partial  F_{1}^{\ast} (-\Lambda^{\ast}(w,u)),\label{opt_cond_1_pqwu}\\
\Lambda(p,q) &\in \partial F_{2}^{\ast}((w,u)).\label{opt_cond_2_pqwu}
\end{align}
It can be checked again that \eqref{opt_cond_2_pqwu} gives \eqref{opt_1} and \eqref{opt_2}. We now expand on \eqref{opt_cond_1_pqwu}. We have that $(p,q)\in   \partial  F_{1}^{\ast} (-\Lambda^{\ast}(w,u))$ which is equivalent to $-\Lambda^{\ast}(w,u)\in \partial F_{1}(p,q)$, that is $F_{1}(p,q)=0$ and 
\begin{alignat*}{3}
%& &F_{1}(p,q)+\langle -\Lambda^{\ast}(w,u), (\tilde{p}-p, \tilde{q}-q) \rangle_{X^{\ast},X} &\le F_{1}(\tilde{p},\tilde{q})\\
%\iff  
& &\langle -\Lambda^{\ast}(w,u), (\tilde{p}-p, \tilde{q}-q) \rangle_{X^{\ast},X} &\le  F_{1}(\tilde{p},\tilde{q})\\
\iff & & -\langle w, \di (\tilde{p}-p) \rangle
- \langle w, \tilde{q}-q \rangle- \langle u,\di  \tilde{q}-\di q \rangle & \le  F_{1}(\tilde{p},\tilde{q})\\
\iff & & \langle \mathcal{E}w, \tilde{p}-p \rangle & \le  \mathcal{I}_{\{|\cdot(x)|_{r}\le\alpha_{0}(x),\,\text{f.a.e.}\,x\}}(\tilde{p})\\
&&\langle Du-w, \tilde{q}-q \rangle & \le  \mathcal{I}_{\{|\cdot(x)|_{r}\le\alpha_{1}(x),\,\text{f.a.e.}\,x\}}(\tilde{q})\\
\iff & & \langle \mathcal{E}w, \tilde{p}-p \rangle & \le  0\\
&&\langle Du-w, \tilde{q}-q \rangle & \le 0,
\end{alignat*}
with the last two inequalities holding for any $\tilde{p}\in W_{0}^{d}(\di^{2};\om)$ with $|\tilde{p}(x)|_{r}\le \alpha_{0}(x)$ for a.e. $x\in\om$ and for any $\tilde{q}\in W_{0}^{d}(\di;\om)$ with $|\tilde{q}(x)|_{r}\le \alpha_{1}(x)$ for a.e. $x\in\om$. Hence we obtain \eqref{opt_3} and \eqref{opt_4}.
\end{proof}

Note that in the proof above we made use of the following density results:
\begin{align*}
\overline{C_{\alpha_{0}}}^{L^{d}(\om)}=K_{\alpha_{0}},\quad
\overline{C_{\alpha_{1}}}^{W_{0}^{d}(\di;\om)}=K_{\alpha_{1}},
\end{align*}
where
\begin{align}
C_{\alpha_{0}}&:=\left \{\di \phi:\; \phi\in C_{c}^{\infty}(\om,\mathcal{S}^{d\times d}),\; |\phi(x)|_{r}\le \alpha_{0}(x),\text{ for all }x\in\om \right \},\\
K_{\alpha_{0}}&:=\left \{\di p:\; p\in W_{0}^{d}(\di^{2};\om),\; |p(x)|_{r}\le \alpha_{0}(x),\text{ for a.e. }x\in\om \right \},\\
C_{\alpha_{1}}&:= \left\{ \psi: \psi\in C_{c}^{\infty}(\om,\RR^{d}),\; |\psi(x)|_{r}\le \alpha_{1}(x), \text{ for all }x\in\om \right \},\\
K_{\alpha_{1}}&:= \left \{q: q\in W_{0}^{d}(\di;\om),\; |q(x)|_{r}\le \alpha_{1}(x), \text{ for a.e. }x\in\om \right \}.
\end{align}
These results can be proven by using the duality arguments of the proof of Proposition \ref{lbl:weighted_TGV_Wdiv}, which originate from \cite{TGV_decompression_part1}, or with the use of mollification techniques; see \cite{sing_mol, Hint_Rau_density, hint_rau_ros}.

\section{A series of regularized problems}\label{sec:regularized_problems}

\subsection{Regularization of the primal problem}
With the aim of lifting the regularity of $u$ and $w$ to avoid measure-valued derivatives, we next consider the following regularized version of the primal weighted TGV problem \eqref{TGV_primal_u_w}:
\begin{equation}\label{TGV_primal_u_w_reg1}
\begin{split}
\text{minimize}\quad %_{\substack{ u\in H^{1}(\om) \\ w\in H^{1}(\om,\RR^{d})}}
\frac{1}{2} \|Tu-f\|_{L^{2}(\om)}^{2}
 &+\int_{\om}\alpha_{1}|\nabla u-w|_{r^{\ast}}dx+
  \int_{\om}\alpha_{0}|Ew|_{r^{\ast}}dx\\
&+\frac{\mu}{2}\|\nabla u\|_{L^{2}(\om)}^{2}+ \frac{\alpha}{2} \|w\|_{H^{1}(\om,\RR^{d})}^{2}\quad\text{over }(u,w)\in H^{1}(\om) \times H^{1}(\om,\RR^{d}),
\end{split}
\end{equation}
for some  constants $0< \mu,\alpha \ll 1$. Existence of solutions for \eqref{TGV_primal_u_w_reg1} follows from standard arguments.

Observe that \eqref{TGV_primal_u_w_reg1} is equivalent to $\min_{(w,u)\in \hat{X}} Q_{1}(w,u)+ Q_{2}(R(w,u))$ where  $\hat{X}=H^{1}(\om,\RR^{d})\times H^{1}(\om)$, $\hat{Y}= L^{2}(\om,\mathcal{S}^{d\times d})\times L^{2}(\om,\RR^{d})$, $R:\hat{X}\to \hat{Y}$ with $R(w,u)=(Ew,\nabla u-w)$, $Q_{1}: X\to \RR$, $Q_{2}: Y\to \RR$ with $Q(w,u)=\frac{1}{2}\|Tu-f\|_{L^{2}(\om)}^{2}+\frac{\mu}{2}\|\nabla u\|^2_{L^{2}(\om,\RR^{d})}+\frac{\alpha}{2}\|w\|^2_{H^{1}(\om,\RR^{d})}$
and $Q_{2}(\psi,\phi)=\int_{\om}\alpha_{1}|\phi|_{r^{\ast}}dx+\int_{\om}\alpha_{0}|\psi|_{r^{\ast}}dx$. %Here we are not interested in the corresponding predual problem but  rather in the primal-dual optimality conditions. 
Note that the Attouch-Brezis condition is satisfied since $\mathrm{dom}(Q_{2})=Y$.

\newtheorem{optimality_reg1}[Wdiv]{Proposition}
\begin{optimality_reg1}\label{lbl:optimality_reg1}
The pairs $(w,u)\in H^{1}(\om,\RR^{d})\times H^{1}(\om)$ and $(p,q)\in L^{2}(\om,\RR^{d\times d})\times L^{2}(\om,\RR^{d})$ are solutions to \eqref{TGV_primal_u_w_reg1}  and its  predual problem, respectively, if and only if the following optimality conditions are satisfied:
\begin{align}
&\text{$Bu-\mu \Delta u+\nabla^{\ast}q -T^{\ast}f=0\;$ in $ H^{1}(\om) ^{\ast}$,}\label{optimality_reg1_1}\\
&\text{$\alpha w- \alpha \Delta w -q+E^{\ast}p=0\;$ in $ H^{1}(\om,\RR^{d}) ^{\ast}$,} \label{optimality_reg1_2}\\
&\begin{cases}
\alpha_{1}(\nabla u -w)-q|\nabla u -w|=0 & \text{ if}\quad |q(x)|_{r}=\alpha_{1}(x), \\
\nabla u-w=0 						& \text{ if}\quad |q(x)|_{r}<\alpha_{1}(x),
 \end{cases}\label{optimality_reg1_3}\\
 &\begin{cases}
\alpha_{0}Ew-p|Ew|=0 & \text{ if}\quad |p(x)|_{r}=\alpha_{0}(x), \\
Ew=0 						& \text{ if}\quad |p(x)|_{r}<\alpha_{0}(x).
 \end{cases}\label{optimality_reg1_4}
\end{align}
\end{optimality_reg1}
\begin{proof}
The proof follows again easily by calculating the corresponding primal-dual optimality conditions.
\end{proof}

%One can compare the optimality conditions \eqref{optimality_reg1_1}--\eqref{optimality_reg1_4} to the ones of the unregularized problem \eqref{opt_1}--\eqref{opt_4}.

Next we study the relationship between the solutions of \eqref{TGV_primal_u_w} and \eqref{TGV_primal_u_w_reg1} as the parameters $\mu$, $\alpha$ tend to zero.
% then the solution pair of \eqref{TGV_primal_u_w_reg1} converge in an appropriate sense to a solution pair of \eqref{TGV_primal_u_w}. For that we need a further assumption on the forward operator $T$.
\newtheorem{limit_mu_a}[Wdiv]{Proposition}
\begin{limit_mu_a}\label{lbl:limit_mu_a}
In addition to the standing assumptions on $T$, let $T$ also be injective on the set of affine functions. Further, let $\mu_{n},\alpha_{n}\to 0$ and let $(w_{n},u_{n})_{n\in\NN}$ be a sequence of  solution pairs of the problem \eqref{TGV_primal_u_w_reg1}. 
Then $u_{n}\stackrel{\ast}{\rightharpoonup}  u^{\ast}$ and $w_{n}\stackrel{\ast}{\rightharpoonup}  w^{\ast}$ in $\bv(\om)$ and $\mathrm{BD}(\om)$ respectively, where $(w^{\ast},u^{\ast})$ is a solution pair for  \eqref{TGV_primal_u_w}. The convergence is up to subsequences.
\end{limit_mu_a}

\begin{proof}
For convenience of notation, define the energies 
\begin{align*}
E_{n}(w,u)
&= \frac{1}{2}\|Tu-f\|_{L^{2}(\om)}^{2}+\int_{\om}\alpha_{1} |\nabla u-w|_{r^{\ast}}dx+\int_{\om}\alpha_{0}| Ew|_{r^{\ast}}dx+ \frac{\mu_{n}}{2} \|\nabla u\|_{L^{2}(\om)}^{2}+ \frac{\alpha_{n}}{2}\|w\|_{H^{1}(\om,\RR^{d})}^{2},\\
E(w,u)&=\frac{1}{2}\|Tu-f\|_{L^{2}(\om)}^{2}+\int_{\om}\alpha_{1}d|Du-w|_{r^{\ast}}+\int_{\om}\alpha_{0}d|\mathcal{E} w|_{r^{\ast}}.
\end{align*}
We have
\begin{equation}\label{E_bound}
\begin{split}
\frac{1}{2}\|Tu_{n}-f\|_{L^{2}(\om)}^{2}+\int_{\om}\alpha_{1}|\nabla u_{n}-w_{n}|_{r^{\ast}}dx+\int_{\om}\alpha_{0}| Ew_{n}|_{r^{\ast}}dx
&\le E_{n}(w_{n},u_{n})\\
& \le E_{n}(0,0)\le \frac{1}{2} \|f\|_{L^{2}(\om)}^{2}.
\end{split}
\end{equation}
Thus, the sequences $(u_{n})_{n\in\NN}$ and $(w_{n})_{n\in\NN}$ are bounded in $\bv(\om)$ and $\mathrm{BD}(\om)$, respectively. In order to see this, note that by setting $\underline{\alpha}_{i}:= \min_{x\in\overline{\om}} \alpha_{i}(x)$, $i=0,1$, we get
\begin{align*}
\tgv_{\underline{\alpha}_{0},\underline{\alpha}_{1}}^{2}(u_{n})&=\min_{w\in\mathrm{BD}(\om)} \underline{\alpha}_{1} \|\nabla u_{n}-w\|_{\cM}+\underline{\alpha}_{0} \|\mathcal{E}w\|_{\cM}\\
&\le \int_{\om}\alpha_{1}|\nabla u_{n}-w_{n}|_{r^{\ast}}dx+\int_{\om}\alpha_{0}| Ew_{n}|_{r^{\ast}}dx
\le \frac{1}{2}\|f\|_{L^{2}(\om)}^{2}.
\end{align*}
Hence,  $(u_{n})_{n\in\NN}$ is bounded in the sense of second-order TGV. Using the fact that $T$ is injective on the set of affine functions, one can further derive a uniform $L^{1}$ bound on $(u_{n})_{n\in\NN}$; see for instance \cite[Theorem 4.2]{BredValk}. This implies further that this sequence is bounded on $\bv(\om)$. The bound on $(w_{n})_{n\in\NN}$ in $\mathrm{BD}(\om)$ then follows from \eqref{E_bound}.

 From  compactness theorems in those spaces (for $\mathrm{BD}(\om)$ see for instance \cite{bd}) we have that there exist $u^{\ast}\in\bv(\om)$ and $w^{\ast}\in \mathrm{BD}(\om)$ such that $u_{n_{k}}\stackrel{\ast}{\rightharpoonup}  u^{\ast}$ and $w_{n_{k}}\stackrel{\ast}{\rightharpoonup}  w^{\ast}$ in $\bv(\om)$ and $\mathrm{BD}(\om)$ respectively along suitable subsequences. Due to the lower semicontinuity of the functional $E$ with respect to these convergences, we have for any pair $(\tilde{w},\tilde{u})\in H^{1}(\om,\RR^{d})\times  H^{1}(\om)$
\begin{align}
E(w^{\ast},u^{\ast})\le \liminf_{k\to\infty} E(w_{n_{k}},u_{n_{k}})\le \liminf_{k\to\infty} E_{n_{k}}(w_{n_{k}},u_{n_{k}})\le \liminf_{k\to\infty} E_{n_{k}}(\tilde{w},\tilde{u})=E(\tilde{w},\tilde{u}). \label{opt_H}
\end{align}
Recall now that $\mathrm{LD}(\om)=\{w\in L^{1}(\om,\RR^{d}):\; Ew\in L^{1}(\om,\RR^{d\times d}) \}$ is a Banach space endowed with the norm $\|w\|_{\mathrm{LD}(\om)}=\|w\|_{L^{1}(\om,\RR^{d})}+ \|Ew\|_{L^{1}(\om,\RR^{d\times d})}$ and that $C^{\infty}(\overline{\om},\RR^{d})$ is dense in that space; see \cite{temam1985mathematical}. From this, in combination with the fact that $C^{\infty}(\overline{\om})$ is dense in $W^{1,1}(\om)\subset L^{d/d-1}(\om)$ we have that for every $(\hat{w},\hat{u})\in \mathrm{LD}(\om)\times  W^{1,1}(\om)$ there exists a sequence 
\[(\hat{w}_{h},\hat{u}_{h})_{h\in\NN}\in C^{\infty}(\overline{\om},\RR^{d})\times C^{\infty}(\overline{\om})\subseteq H^{1}(\om,\RR^{d})\times  H^{1}(\om),\]
such that
$E(\hat{w}_{h},\hat{u}_{h})\to E(\hat{w},\hat{u})$.
Hence, since \eqref{opt_H} holds we have that
\begin{equation}\label{opt_LD_W}
E(w^{\ast},u^{\ast})\le E(\hat{w},\hat{u}),\quad \text{ for all }(\hat{w},\hat{u})\in \mathrm{LD}(\om)\times W^{1,1}(\om).
\end{equation}
Finally, by following similar steps as in the proof of \cite[Thm. 3]{valkonen_jump_2}, we can show that  for every $(w,u)\in \mathrm{BD}(\om)\times \bv(\om)$ there exists a sequence $(w_{h},u_{h})_{h\in\NN}\in \mathrm{LD}(\om)\times W^{1,1}(\om)$ such that
\[\|u_{h}-u\|_{L^{d/d-1}(\om)}\to 0,\;\;
 \int_{\om}\alpha_{1}|\nabla u_{h}-w_{h}|_{r^{\ast}}dx\to \int_{\om}\alpha_{1}d|Du-w|_{r^{\ast}},\;\; 
 \int_{\om}\alpha_{0}|Ew_{h}|_{r^{\ast}}dx\to \int_{\om}\alpha_{0}d|\mathcal{E}w|_{r^{\ast}}, \]
which implies again  that
$E(w_{h},u_{h})\to E(w,u)$.
This, together with \eqref{opt_LD_W} yields
\[E(w^{\ast},u^{\ast})\le E(w,u),\quad \text{ for all }(w,u)\in \mathrm{BD}(\om)\times \bv(\om).\]
This yields that $(w^{\ast},u^{\ast})$ is a solution pair for \eqref{TGV_primal_u_w}.
% Finally from the uniqueness of the solution $u^{\ast}$ for \eqref{TGV_primal_u_w}, we have that for the unitial sequence we have $u_{n}\stackrel{\ast}{\rightharpoonup}  u^{\ast}$.
\end{proof}

Note that if the solution  $u^{\ast}$ of \eqref{TGV_primal_u_w} is unique, then we have $u_{n}\stackrel{\ast}{\rightharpoonup}  u^{\ast}$ along the entire sequence.

We now proceed to the second level of regularization of the problem \eqref{TGV_primal_u_w_reg1}, which, in addition to lifting the regularity of $u$ and $w$, respectively, also smoothes the non-differentiable constituents. For this purpose, we define the following primal problem which will also be treated numerically below:
\begin{equation}\label{tgv_primal_reg2}
\begin{split}
\text{minimize}\quad  \frac{1}{2}\|Tu-f\|_{L^{2}(\om)}^{2}&+\int_{\om}\alpha_{1}\varphi_{\gamma,r^{\ast}}(\nabla u-w)dx+\int_{\om}\alpha_{0} \varphi_{\gamma,r^{\ast}}(Ew)dx\\&+ \frac{\mu}{2} \|\nabla u\|_{L^{2}(\om)}^{2}+ \frac{\alpha}{2}\|w\|_{H^{1}(\om,\RR^{d})}^{2}\text{ over }(u,w)\in H^{1}(\om)\times H^{1}(\om,\RR^{d}).%\tag{$P_{\gamma}$}.
\end{split}
\tag{$P_{\gamma}$}
\end{equation}
Here $\varphi_{\gamma,r^{\ast}}$ denotes the Huber-regularized version of the $|\cdot|_{r^{\ast}}$ norm. In what follows, for notational convenience we will focus on $\varphi_{\gamma}:=\varphi_{\gamma,2}$, i.e.,
for a vector $v\in X$, $S=\RR^{d}$ or $\RR^{d\times d}$ and $\gamma>0$ we use  
\begin{equation}\label{Huber}
\varphi_{\gamma}(v)(x)=
\begin{cases}
|v(x)|-\frac{1}{2}\gamma         & \text{ if}\quad |v(x)|\ge \gamma,\\
\frac{1}{2\gamma} |v(x)|^{2}  & \text{ if}\quad |v(x)|< \gamma,
\end{cases}
\end{equation}
with $|\cdot|$ denoting either the Euclidean norm in $\RR^{d}$ or the Frobenius norm in $\RR^{d\times d}$. 
We mention that this type of Huber regularization of $\tv$-type terms in the primal problem corresponds to an $L^{2}$ regularization of the dual variables in the predual  \cite{journal_tvlp, stadler}. In order to illustrate this consider the following denoising problem \eqref{tgv_primal_reg2} without any $H^{1}$ regularization:
\begin{equation}\label{tgv_primal_Huber}
\text{minimize}\quad  \frac{1}{2}\|u-f\|_{L^{2}(\om)}^{2}+\int_{\om}\alpha_{1}d|Du-w|_{\gamma_{1}}+\int_{\om}\alpha_{0}d|\mathcal{E}w|_{\gamma_{2}}\quad\text{over }(u,w)\in \bv(\om)\times \mathrm{BD}(\om) ,
\end{equation}
%Here $\|Du-w\|_{\gamma_{1},\cM}=\|\nabla u-w\|_{\gamma_{1}, L^{1}(\om,\RR^{d})}+\|D^{s}u\|_{\cM}$ and $\|\mathcal{E}w\|_{\gamma_{2},\cM}=\|Ew\|_{\gamma_{2}, L^{1}(\om,\RR^{d\times d})}+\|\mathcal{E}^{s}w\|_{\cM}$.
where
\begin{align*}
\int_{\om}\alpha_{1}d|Du-w|_{\gamma_{1}}
&=\int_{\om} \alpha_{1} \varphi_{\gamma_{1}}(\nabla u -w)dx + \int_{\om}\alpha_{1}d|D^{s}u|,\\
\int_{\om}\alpha_{0}d|\mathcal{E}w|_{\gamma_{2}}
&=\int_{\om}\alpha_{0} \varphi_{\gamma_{2}}(Ew)dx + \int_{\om} \alpha_{0}d|\mathcal{E}^{s}w|.
\end{align*}
Its corresponding predual problem is given by
\begin{equation}\label{tgv_predual_Huber}
\begin{split}
&\text{maximize}\quad
%_{\substack{p\in W_{0}^{d}(\di^{2};\om) \\ q\in W_{0}^{d}(\di;\om),\, q=-\di p  \\ |p(x)|\le \alpha_{0}(x),\, | q(x)|\le \alpha_{1}(x) }} 
-\frac{1}{2} \|f+\di q\|_{L^{2}(\om)}^{2} - \frac{\gamma_{0}}{2} \int_{\om}\frac{1}{\alpha_{0}}|p|^{2}dx-\frac{\gamma_{1}}{2} \int_{\om}\frac{1}{\alpha_{1}}|q|^{2}dx+\frac{1}{2}\|f\|_{L^{2}(\om)}^{2},\\
&\text{over }(p,q)\in W_{0}^{d}(\di^{2};\om)\times  W_{0}^{d}(\di;\om),\\
&\text{subject to }q=-\di p,\, |p(x)|\le \alpha_{0}(x),\, | q(x)|\le \alpha_{1}(x) .
\end{split}
\end{equation}
The proof is similar to the one of Proposition \ref{lbl:optimality_u_w_q_p} with
\[F_{1}(p,q)=\mathcal{I}_{\{|\cdot(x)|\le \alpha_{0}(x)\}}(p)+\mathcal{I}_{\{|\cdot(x)|\le\alpha_{1}(x)\}}(q)
 - \frac{\gamma_{0}}{2} \int_{\om}\frac{1}{\alpha_{0}}|p|^{2}dx-\frac{\gamma_{1}}{2} \int_{\om}\frac{1}{\alpha_{1}}|q|^{2}dx,\]
and in the dualization process we use the fact that for an $S$-valued measure $\mu$ we have,
\begin{align*}
\int_{\om}\alpha d\varphi_{\gamma}(\mu)
%&=\sup_{\phi\in C_{c}^{\infty}(\om,S)} \int_{\om}\phi\,d\mu-\int_{\om}\Phi_{\gamma,\alpha}^{\ast}(\phi)\,dx\\
&=\sup \left\{\int_{\om}\phi\,d\mu-\mathcal{I}_{\{|\cdot(x)|\le \alpha(x)\}}(\phi)-\frac{\gamma}{2}\|\phi\|_{L^{2}(\om)}^{2}:\phi\in C_{c}^{\infty}(\om,S)\right\};
\end{align*}
%where $\Phi_{\gamma,\alpha}: \RR^{d}\to \RR$, $\Phi_{\gamma,\alpha}(v)=\alpha(x)$
see for instance \cite{Dem}.

Returning to the (doubly) regularized primal problem \eqref{tgv_primal_reg2}, we are primarily interested in its associated first-order optimality conditions. 

\newtheorem{optimality_reg2}[Wdiv]{Proposition}
\begin{optimality_reg2}\label{lbl:optimality_reg2}
We have that the pairs $(w,u)\in H^{1}(\om,\RR^{d})\times H^{1}(\om)$ and $(p,q)\in L^{2}(\om,\RR^{d\times d})\times L^{2}(\om,\RR^{d})$ are solution to \eqref{tgv_primal_reg2} and its predual problem, respectively, if and only if the following optimality conditions are satisfied:
\begin{align}
&\text{$Bu-\mu \Delta u+\nabla^{\ast}q -T^{\ast}f=0\;$ in $ H^{1}(\om) ^{\ast}$,} \label{reg_opt1}\tag{$Opt_{1}$}\\
&\text{$\alpha w- \alpha \Delta w -q+E^{\ast}p=0\;$ in $ H^{1}(\om,\RR^{d}) ^{\ast}$,} \label{reg_opt2}\tag{$Opt_{2}$}\\
&\text{$ \max (|\nabla u -w|,\gamma_{1})q-\alpha_{1}(\nabla u-w)=0\;$ in $L^{2}(\om,\RR^{d})$, } \label{reg_opt3}\tag{$Opt_{3}$} \\
&\text{$\max(|Ew|,\gamma_{0})p-\alpha_{0}Ew=0\;$ in $L^{2}(\om,\mathcal{S}^{d\times d})$.} \label{reg_opt4}\tag{$Opt_{4}$} 
\end{align}
\end{optimality_reg2}

The proof of Proposition \ref{lbl:optimality_reg2} follows from calculating the corresponding primal-dual optimality conditions as in Proposition \ref{lbl:optimality_reg1}.
The analogous approximation result follows, where we have set $\gamma_{0}=\gamma_{1}=\gamma$ and $T=Id$ for simplicity.

\newtheorem{gamma_approximation}[Wdiv]{Proposition}
\begin{gamma_approximation}\label{lbl:gamma_approximation}
Let $(w,u,q,p)$ and $(w_{\gamma},u_{\gamma}, p_{\gamma},q_{\gamma})$ satisfy the optimality conditions \eqref{optimality_reg1_1}--\eqref{optimality_reg1_4} and \eqref{reg_opt1}--\eqref{reg_opt4}, respectively. Then, as $\gamma\to 0$, we have $u_{\gamma}\to u$ strongly in $H^{1}(\om)$, $w_{\gamma}\to w$ strongly in $H^{1}(\om,\RR^{d})$ as well as $\di q_{\gamma} \to \di q$ and $q_{\gamma}+\di p_{\gamma}\to q+ \di p$ weakly$^\ast$ in $H^{1}(\om)^{\ast}$ and $H^{1}(\om,\RR^{d})^{\ast}$, respectively.
\end{gamma_approximation}

\begin{proof}
By subtracting first two equations of the optimality system of Proposition \ref{lbl:optimality_reg1} and \ref{lbl:optimality_reg2}, respectively, we get for all $v\in H^{1}(\om)$, $\omega\in H^{1}(\om,\RR^{d})$
\begin{align}
&\int_{\om}(u-u_{\gamma})v\, dx+ \mu\int_{\om}\nabla (u-u_{\gamma})\nabla v\, dx=\int_{\om}(q_{\gamma}-q)\nabla v\,dx,\label{weak_u}\\
&\alpha\int_{\om} (w-w_{\gamma})\omega\,dx+ \alpha\int_{\om} \nabla (w-w_{\gamma})\nabla \omega\,dx=\int_{\om}(q-q_{\gamma})\omega\,dx+\int_{\om}(p_{\gamma}-p)E\omega\, dx.\label{weak_w}
\end{align}
When using $v=u-u_{\gamma}$ and $\omega=w-w_{\gamma}$ in the equations above and adding them up we get
\begin{equation}\label{add_u_w}
\|u-u_{\gamma}\|_{L^{2}(\om)}^{2}+\mu \|\nabla u-\nabla u_{\gamma}\|_{L^{2}(\om,\RR^{d})}^{2}
+\alpha \|w-w_{\gamma}\|_{H^{1}(\om,\RR^{d})}^{2}=R_{1}+R_{2},
\end{equation}
where 
\[R_{1}:=\int_{\om} (q_{\gamma}-q)^\top [\nabla u-w-(\nabla u_{\gamma}-w_{\gamma})]\,dx,\qquad
R_{2}:=\int_{\om}(p_{\gamma}-p)^\top E(w-w_{\gamma})\,dx.
\]
We now estimate $R_{1}$ and $R_{2}$. Consider the partitions of $\om$ into disjoint sets (up to sets of measure zero) $\om=\mathcal{A}\cup \mathcal{I}=\mathcal{A}_{\gamma}\cup \mathcal{I}_{\gamma}$, where
\begin{alignat*}{3}
\mathcal{A}&=\{x\in\om: \;|\nabla u-w|>0\}, \qquad &&\mathcal{I}&&=\om\setminus \mathcal{A},\\
\mathcal{A}_{\gamma}&=\{x\in\om: \;|\nabla u_{\gamma}-w_{\gamma}|>\gamma\},\qquad &&\mathcal{I}_{\gamma}&&=\om\setminus \mathcal{A}_{\gamma}.
\end{alignat*}
We estimate $R_{1}$ separately on the  disjoint sets $\mathcal{A}_{\gamma}\cap \mathcal{A}$, $\mathcal{A}_{\gamma}\cap \mathcal{I}$, $\mathcal{I}_{\gamma}\cap\mathcal{A}$ and $\mathcal{I}_{\gamma}\cap \mathcal{I}$.  Recall that  $|q(x)|\le \alpha_{1}(x)$, $|q_{\gamma}(x)|\le \alpha_{1}(x)$ for almost every $x\in\om$. Starting from $\mathcal{A}_{\gamma}\cap \mathcal{A}$ and noticing that 
\[q=\alpha_{1}\frac{\nabla u-w}{|\nabla u -w|},\quad q_{\gamma}=\alpha_{1}\frac{\nabla u_{\gamma}-w_{\gamma}}{|\nabla u_{\gamma}-w_{\gamma}|},\]
it follows  that pointwise on $\mathcal{A}_{\gamma}\cap \mathcal{A}$ (with argument $x$ left off for ease of notation) we have
\begin{align*}
 (q_{\gamma}-q)^\top[\nabla u-w-(\nabla u_{\gamma}-w_{\gamma})]
 &=q_{\gamma} (\nabla u-w)-\alpha_{1}|\nabla u_{\gamma} -w_{\gamma} |- \alpha_{1}|\nabla u-w|+q(\nabla u_{\gamma}-w_{\gamma})\\
 &\le \alpha_{1}|\nabla u-w|-\alpha_{1}|\nabla u_{\gamma} -w_{\gamma} |- \alpha_{1}|\nabla u-w|+\alpha_{1}|\nabla u_{\gamma} -w_{\gamma} |\\
 &=0.
\end{align*}
Turning now to the set $\mathcal{A}_{\gamma}\cap \mathcal{I}$ and recalling $\nabla u-w=0$ we have
\begin{align*}
 (q_{\gamma}-q)^\top [\nabla u-w-(\nabla u_{\gamma}-w_{\gamma})]\le -\alpha_{1}|\nabla u_{\gamma}-w_{\gamma}|+|q||\nabla u_{\gamma}-w_{\gamma}|\le 0.
\end{align*}
For the set $\mathcal{I}_{\gamma}\cap\mathcal{A}$, note that 
\[|\nabla u_{\gamma}-w_{\gamma}|\le \gamma,\quad \nabla u_{\gamma} - w_{\gamma}=\frac{\gamma}{\alpha_{1}}q_{\gamma}. \]
Thus, we can estimate
\begin{align*}
 (q_{\gamma}-q)^\top [\nabla u-w-(\nabla u_{\gamma}-w_{\gamma})]
 &\le q_{\gamma}(\nabla u -w)-\alpha_{1}|\nabla u - w| -q_{\gamma}(\nabla u_{\gamma}-w_{\gamma})+q(\nabla u_{\gamma}-w_{\gamma})\\
 &\le \alpha_{1}|\nabla u-w|-\alpha_{1}|\nabla u - w| -q_{\gamma}(\nabla u_{\gamma}-w_{\gamma})+q(\nabla u_{\gamma}-w_{\gamma})\\
 &\le -\frac{\gamma}{\alpha_{1}}|q_{\gamma}|^{2}+\alpha_{1}\frac{\gamma}{\alpha_{1}}|q_{\gamma}|=\gamma|q_{\gamma}|\left (1-\frac{|q_{\gamma}|}{\alpha_{1}} \right )\le \gamma\alpha_{1}.
\end{align*}
Similarly, for the set $\mathcal{I}_{\gamma}\cap\mathcal{I}$ we get
\begin{align*}
 (q_{\gamma}-q)^\top [\nabla u-w-(\nabla u_{\gamma}-w_{\gamma})]\le \gamma\alpha_{1}.
 \end{align*}
Combining the above estimates we have
\[R_{1}\le \int_{\om}\gamma\alpha_{1}\,dx\to 0\]
and for $R_{2}$ we get 
\[R_{2}\le \int_{\om}\gamma\alpha_{0}\,dx\to 0.\]
Hence, from \eqref{add_u_w} we obtain the desired convergences for $u_{\gamma}$ and $w_{\gamma}$.  From this result and using \eqref{weak_u} and \eqref{weak_w} we get that for every $v\in H^{1}(\om)$ and for every $\omega\in H^{1}(\om,\RR^{d})$ we have \begin{align*}
\int_{\om}v\di q_{\gamma} \,dx &\to \int_{\om} v\di q \,dx\quad \text{ and }\quad 
\int_{\om}\omega(q_{\gamma}+\di p_{\gamma})\,dx\to \int_{\om}\omega(q+\di p)\,dx,
\end{align*}
as $\gamma\to 0$. This completes the proof.
\end{proof}

Finally, the following approximation result holds true, when $\alpha$, $\mu$ and $\gamma$ tend to zero. 
\newtheorem{full_approximation}[Wdiv]{Proposition}
\begin{full_approximation}\label{lbl:full_approximation}
Let $T=Id$,  $\mu_{n},\alpha_{n},\gamma_{n}\to 0$, and denote by $u_{\mu_{n},\alpha_{n},\gamma_{n}}\in H^{1}(\om)$ the solution of  \eqref{tgv_primal_reg2} with $(\mu,\alpha,\gamma)=(\mu_{n},\alpha_{n},\gamma_{n})$. Then $u_{\mu_{n},\alpha_{n},\gamma_{n}}\stackrel{\ast}{\rightharpoonup}  u^{\ast}$ in $\bv(\om)$, where $u^{\ast}$ solves \eqref{TGV_primal_u_w}.
\end{full_approximation}

\begin{proof}
It is easy to show that $u_{\mu_{n},\alpha_{n},\gamma_{n}}\to u^{\ast}$ in $L^{1}(\om)$. Indeed, we have 
\[\|u_{\mu_{n},\alpha_{n},\gamma_{n}}-u^{\ast}\|_{L^{1}(\om)}
\le \|u_{\mu_{n},\alpha_{n},0}-u^{\ast}\|_{L^{1}(\om)}+ \|u_{\mu_{n},\alpha_{n},\gamma_{n}}-u_{\mu_{n},\alpha_{n},0}\|_{L^{1}(\om)}.\]
According to Proposition \ref{lbl:limit_mu_a} it holds that $ \|u_{\mu_{n},\alpha_{n},0}-u^{\ast}\|_{L^{1}(\om)}\to 0$. The other term tends to zero according to equation \eqref{add_u_w} of Proposition \ref{lbl:gamma_approximation}. There, the estimates for $R_{1}, R_{2}$ are not affected if we substitute $u$ and $u_{\gamma}$ by $u_{\mu_{n},\alpha_{n},0}$ and $u_{\mu_{n},\alpha_{n},\gamma_{n}}$, respectively. In other words, the estimate
\[\|u_{\mu_{n},\alpha_{n},0}-u_{\mu_{n},\alpha_{n},\gamma_{n}}\|_{L^{2}(\om)}^{2}\le \gamma_{n}|\om|\|\alpha_{0}+\alpha_{1}\|_{\infty}\]
holds and hence $\|u_{\mu_{n},\alpha_{n},\gamma_{n}}-u_{\mu_{n},\alpha_{n},0}\|_{L^{1}(\om)}\to 0$. 

To finish the proof and show that the convergence is weak$^{\ast}$ in $\bv(\om)$, it suffices to establish that $\int_{\om}|\nabla u_{\mu_{n},\alpha_{n},\gamma_{n}}\,dx |$ is uniformly bounded in $n$; see \cite[Prop. 3.13]{AmbrosioBV}. Observe first that as in the proof of Proposition \ref{lbl:limit_mu_a} we get 
\begin{equation}\label{tgv_gamma_bd}
\int_{\om}\alpha_{1}\varphi_{\gamma}(\nabla u_{\mu_{n},\alpha_{n},\gamma_{n}}-w_{\mu_{n},\alpha_{n},\gamma_{n}})dx +\int_{\om}\alpha_{0}\varphi_{\gamma}(Ew_{\mu_{n},\alpha_{n},\gamma_{n}})dx\le \frac{1}{2}\|f\|_{L^{2}(\om)}^{2}.
\end{equation}
From \eqref{Huber} we have that $\varphi_{\gamma}(\cdot)\ge |\cdot|-\frac{1}{2}\gamma$, and hence we obtain
\begin{equation}\label{tgv_nogamma_bd}
\int_{\om}\alpha_{1}|\nabla u_{\mu_{n},\alpha_{n},\gamma_{n}}-w_{\mu_{n},\alpha_{n},\gamma_{n}}|dx +\int_{\om}\alpha_{0}|Ew_{\mu_{n},\alpha_{n},\gamma_{n}}| dx\le \frac{1}{2} \|f\|_{L^{2}(\om)}^{2}+\frac{(\|\alpha_{1}\|_{\infty}+\|\alpha_{0}\|_{\infty})|\om|\gamma_{n}}{2}\le K,
\end{equation}
for some constant $K>0$.
Then, as in the proof of Proposition \ref{lbl:limit_mu_a}, we get that  $(u_{\mu_{n},\alpha_{n},\gamma_{n}})_{n\in\NN}$ is bounded in TGV which, together with the $L^{1}$ bound, gives the desired bound in TV.
\end{proof}

\subsection{Regularization of the predual problem}

We now consider the following regularization of the predual problem \eqref{duality_TGV_predual} for $\epsilon>0$: 
\begin{equation}\label{tgv_predualE}%\tag{$P_\epsilon^*$}
\min_{p\in H_0^2(\Omega, \mathcal{S}^{d\times d})}\frac{\epsilon}{2}\|\Delta p\|_{L^2(\Omega, \mathcal{S}^{d\times d})}^2+\frac{\epsilon}{2}\|p\|_{L^{2}(\om,\mathcal{S}^{d\times d})}^{2}+\frac{1}{2} \|T^{\ast}f-\di^{2} p\|_{B}^{2}+\frac{1}{\epsilon}\mathrm{M}(p),
\end{equation}
%where
%\begin{align*}
%H_0^2(\Omega; \mathcal{S}^{d\times d})&:=\overline{C_{c}^{\infty}(\om,\mathcal{S}^{d\times d})}^{\|\cdot\|_{H^2(\om; \mathbb{R}^{d\times d})}}=\{q\in H^2(\Omega, \mathcal{S}^{d\times d}): q=\frac{\partial q}{\partial n}=0 \:\text{ on }\: \partial \Omega\},
%\end{align*}
where $H_0^2(\Omega, \mathcal{S}^{d\times d})$ denotes the usual Sobolev space with homogeneous first-order trace on the boundary \cite{Adams}, and the  map $\mathrm{M}:H_0(\di^{2};\Omega) \to \mathbb{R}^+_0$ is convex and continuous, with $\mathrm{M}(p)=0$ if and only if $|p(x)|_{r}\le \alpha_{0}(x)$ and $|\di p (x)|_{r}\le \alpha_{1}(x)$ for almost every $x\in\om$. We also assume that $M$ is coercive in the sense that $M(p_{n})\to \infty$ if $\max\{\|p_{n}\|_{L^{2}(\om)},\|\di p_{n}\|_{L^{2}(\om)}\}\to\infty$ for some sequence $(p_{n})_{n\in\NN}$. Further, $\Delta$ denotes the vector Laplacian operator, which is the standard Laplacian applied component-wise. For the sake of discussion, we mention that more sophisticated regularizations securing $\di p\in L^r(\om)$ with $r>2$ for the subsequent application of (function space versions of) generalized Newton methods for solving this problem are possible as well.

\newtheorem{consistency}[Wdiv]{Proposition}
\begin{consistency}\label{lbl:consistency}
Problem \eqref{tgv_predualE} admits a unique solution $p_\epsilon\in H_0^2(\Omega, \mathcal{S}^{d\times d})$, and \begin{equation}\label{conv}
\di^{2}p_\epsilon\to \di^{2}p, 
\end{equation}
in $L^2(\Omega)$ as $\epsilon\to 0$, up to subsequences, where $p$ solves  problem \eqref{duality_TGV_predual}.
\end{consistency}

\begin{proof}
By $J(\cdot)$ we denote the optimal objective of \eqref{duality_TGV_predual}, where we ignore the term $\frac{1}{2}\|f\|_{L^{2}(\om)}$, and let $K_{\balpha}$ be the corresponding constraint set. Let $\epsilon_{n}\to 0$.
Note that  $\|\cdot\|_{L^2}+ \|\Delta \cdot\|_{L^2}$ is a norm on $H^2(\om, \mathcal{S}^{d\times d})$ \cite{grisvard2011elliptic}. Thus, the minimizing functional in \eqref{tgv_predualE}, denoted by $J_{n}(\cdot)$, is coercive over $H_0^2(\Omega, \mathcal{S}^{d\times d})$ for every $n\in\mathbb{N}$. Hence, any infimizing sequence of  \eqref{tgv_predualE} has a weakly convergent subsequence in $H_0^2(\Omega, \mathcal{S}^{d\times d})$.  
%since the non-negative functional $J_\epsilon(p):=J(p)+ G_\epsilon(p)$, with $J(p):=\frac{1}{2} \|f-\di^{2} p\|_{L^{2}(\om)}^{2}$, and $G_\epsilon(p):=\frac{\epsilon}{2}\|\Delta p\|_{L^2(\Omega; \mathcal{S}^{d\times d})}^2+\frac{1}{\epsilon}\mathrm{M}(p)$, is coercive.
  Further, $J_n$ is weakly lower semicontinuous and, thus,    \eqref{tgv_predualE} has a solution $p_n$, which is unique due to strict convexity.

%A standard optimization technique for penalty methods (see \cite[\S 10.11. Lemma 1]{MR0238472}) shows that $ G_\epsilon(p_\epsilon)\to 0$, and 

Since $J_n(p_n)\leq J_{n}(0)$ for all $n\in\mathbb{N}$, by using the coercivity assumptions on $M$, we have that $(p_n)_{n\in\NN}$ is bounded in $H_{0}(\di^{2};\om)$. Hence, there exists $p^{\ast}\in H_{0}(\di^{2};\om)$ and an (unrelabeled) subsequence of $(p_{n})_{n\in\NN}$ converging weakly to $p^{\ast}$. We then have
\begin{equation}\label{ineq}
J(p^{\ast})\leq \liminf_{n\to\infty} J (p_n)\leq \liminf_{n\to\infty} J_n(p_n)\leq \limsup_{n\to\infty} J_n(p_n)\leq \limsup_{n\to\infty} J_n(\tilde{p})= J(\tilde{p}),
\end{equation}
for all $\tilde{p}\in H_0^2(\Omega, \mathcal{S}^{d\times d})$ with $\tilde{p}\in K_{\balpha}$. Note that necessarily $p^{\ast}\in K_{\balpha}$ as well since
\[M(p^{\ast})\le \liminf_{n\to\infty} M(p_{n})\le \liminf_{n\to\infty} \frac{\epsilon_{n}}{2} \|T^{\ast}f\|_{B}^{2}=0.\]
%This together with \eqref{ineq} holds,  implies 
%\begin{equation*}
%J(p)\leq J(\tilde{p}),\qquad \forall \tilde{p}\in H_0^2(\Omega, \mathcal{S}^{d\times d}), \text{ such that } p\in K_{\balpha}.
%\end{equation*}
We claim that $p^{\ast}$ actually solves \eqref{duality_TGV_predual}. 
Indeed, for every $p\in H_{0}(\di^{2};\om)$ with $p\in K_{\balpha}$, we get
 from the density \eqref{density_weighted_TGV_H0} that there exists $(\tilde{p}_{n})_{n\in\NN}\subset C_{c}^{\infty}(\om,\mathcal{S}^{d\times d})\subset H_0^2(\Omega, \mathcal{S}^{d\times d})$ and $\tilde{p}_{n} \in K_{\balpha}$, such that $\di ^{2}\tilde{p}_{n}\to \di^{2} p$ in $L^{2}(\om)$. Hence, from the continuity of $J$ we get
$J(p^{\ast})\leq \lim_{n\to \infty} J(\tilde{p}_{n})=J(p)$,
and thus $p^{\ast}$ solves \eqref{duality_TGV_predual}.
%  Further, every weak limit point   $p$ of $(p_n)_{n\in\NN}$ satisfies $\di^{2}p=\di^{2}\tilde{p}$, for the entire sequence we have
%\begin{equation*}
%\di^{2}p_n\rightharpoonup\di^{2}\tilde{p},
%\end{equation*}
%in $L^2(\Omega)$.
 Finally, from \eqref{ineq} we observe that $\|\di^{2}p_n\|_{L^2(\Omega)}\to \|\di^{2}p\|_{L^2(\Omega)}$, and hence \eqref{conv} holds.% (\textbf{MORE JUSTIFICATION HERE}).
\end{proof}

For this problem we take $r=\infty$ leading to the anisotropic version of TGV and use
%\begin{equation}\label{Q_P}
$M(p)=\mathcal{Q}_{\delta}(p,\alpha_{0})+ \mathcal{P}_{\delta}(\di p, \alpha_{1})$,
%\end{equation}
where
\begin{align}
\mathcal{P}_{\delta}(q,\alpha_{1})
&=\int_{\om} \sum_{i=1}^{d}(G_{\delta}(-(q_{i}+\alpha_{1}))+ G_{\delta}(q_{i}-\alpha_{1}))dx, \label{mathcal_P}\\
\mathcal{Q}_{\delta}(p,\alpha_{0})&=\int_{\om} \sum_{\substack{i,j=1\\i\le j}} (G_{\delta}(-(p_{ij}+\alpha_{0}))+G_{\delta}(p_{ij}-\alpha_{0}))dx, \label{mathcal_Q}
\end{align}
with $G_{\delta}:\RR\to \RR$ acting component-wise and defined by
\begin{equation}\label{G_d}
G_{\delta}(t)=
\begin{cases}
\frac{1}{2}t^{2}-\frac{\delta}{2}t+\frac{\delta^{2}}{6}, & \text{ if }t\ge \delta,\\
\frac{t^{3}}{6\delta}, & \text{ if } 0<t<\delta,\\
0, &  \text{ if } t\le 0, 
\end{cases}
\end{equation}
for $\delta>0$. 
%We define similarly $\mathcal{Q}_{\delta}$
%\begin{equation}\label{mathcal_Q}
%\mathcal{Q}_{\delta}(p,\alpha_{0})=\int_{\om} \sum_{\substack{i,j=1\\i\le j}} (G_{\delta}(-(p_{ij}+\alpha_{0}))+G_{\delta}(p_{ij}-\alpha_{0}))dx
%\end{equation} 
Summarizing and allowing for different regularization weights $\beta>0$, $\gamma>0$ (rather than $\beta=\gamma=\epsilon>0$),  \eqref{tgv_predualE} takes the form
\begin{equation}\label{TGV_predual_Newton}
\min_{p\in H_{0}^{2}(\om,\mathcal{S}^{d\times d})} \frac{\beta}{2} \|\Delta p\|_{L^{2}(\om,\mathcal{S}^{d\times d})}^{2} +\frac{\gamma}{2} \|p\|_{L^{2}(\om,\mathcal{S}^{d\times d})}^{2} + \frac{1}{2} \|T^{\ast}f-\di^{2}p\|_{B}^{2}+\frac{1}{\epsilon_{0}} \mathcal{Q}_{\delta}(p,\alpha_{0})+\frac{1}{\epsilon_{1}} \mathcal{P}_{\delta}(\di p,\alpha_{1}),
\end{equation}
where, for greater flexibility, we also use $\frac{1}{\epsilon_{0}}>0$  and  $\frac{1}{\epsilon_{1}}>0$, respectively, in front $\mathcal{Q}_{\delta}$ and $\mathcal{P}_{\delta}$. Note that for sufficiently small $\epsilon_{0}$, $\epsilon_{1}$, the quantities $\mathcal{Q}_{\delta}(p,\alpha_{0})$ and $\mathcal{P}_{\delta}(\di p,\alpha_{1})$ get small as well and $p$ and $\di p$ are expected to ``approximately'' satisfy the box constraints in \eqref{duality_TGV_predual}. 

The Euler-Lagrange equation for \eqref{TGV_predual_Newton} reads
\begin{equation}\label{TGV_EL}
g_{\mathrm{d}}(p,\alpha_{0},\alpha_{1}):=\beta \Delta^{2}p+\gamma p +\nabla^{2}B^{-1}\di^{2}p-\nabla^{2}B^{-1}T^{\ast}f+\frac{1}{\epsilon_{0}}Q_{\delta}(p,\alpha_{0})-\frac{1}{\epsilon_{1}}\nabla P_{\delta}(\di p,\alpha_{1})=0,
\end{equation}
in $[H_{0}^{2}(\om,\mathcal{S}^{d\times d})]^{\ast}$. 
Here, $P_{\delta}$  denotes
$P_{\delta}(q,\alpha_{1}):= G_{\delta}'(q-\alpha_{1})- G_{\delta}'(-q-
\alpha_{1})$, and
%\begin{equation}
%G_{\delta}'(t)=
%\begin{cases}
%t-\frac{\delta}{2}, & \text{ if } t\ge \delta,\\
%\frac{t^{2}}{2\delta}, & \text{ if } 0<t<\delta,\\
%0,& \text{ if } t\le 0,
%\end{cases}
%\end{equation}
$Q_{\delta}$ is analogous.% to $P_{\delta}$ with the use of $G_{\delta}'$.

\section{Two bilevel optimization schemes}\label{sec:Bilevel}

In this section we will adapt the bilevel optimization framework developed in \cite{hintermuellerPartI, hintermuellerPartII} in order to automatically select the regularization functions $\alpha_{0}$ and $\alpha_{1}$.  The main idea is to minimize a suitable upper level objective over both the image $u$ and the regularization parameters $\alpha_{0}$, $\alpha_{1}$ subject to $u$ being a solution to a (regularized)  TGV-based reconstruction problem with these regularization weights. 
%As we have already discussed one can instead  use equivalently the predual problem as a lower level problem.

It is useful to recall the definitions of  the localized residual $\mathrm{R}$ and the function $F$ as  stated  in the introduction:
 \begin{equation}\label{loc_res}
 \mathrm{R}u(x)=\int_{\om} w(x,y) (Tu-f)^{2}(y)dy,
 \end{equation}
 where $w\in L^{\infty}(\om\times \om)$ with $\int_{\om}\int_{\om}w(x,y)\,dxdy=1$ and
 \begin{equation}\label{ul_F}
 F(v):=\frac{1}{2} \int_{\om} \max(v-\overline{\sigma}^{2},0)^{2}dx+\frac{1}{2} \int_{\om} \min(v-\underline{\sigma}^{2},0)^{2}dx,
 \end{equation}
for some appropriately chosen $\underline{\sigma}^{2}, \overline{\sigma}^{2}$.
 We  next describe two bilevel schemes each one based on the two regularized TGV problems studied in the previous sections.

\subsection{Bilevel dual}
Noting that the {\it localized} residual $\mathrm{R}u$ can also be written in terms of the dual variable $p$ yielding 
 \begin{align}
\mathrm{R}u(x)&=R(\di^{2} p)(x):=\int_{\om} w(x,y)\left (TB^{-1}\di^{2} p -(TB^{-1}T^{\ast}-I)f \right )^{2}dy.\label{upper_level_div2p}
 \end{align}
The duality based bilevel TGV problem is defined as follows: 
  \begin{equation}\label{bilevelTGV}\tag{$\mathbb{P}_{\tgv}$-d}\left \{
 \begin{aligned}
 &\min\; J_{\mathrm{d}}(p,\alpha_{0},\alpha_{1}):= F(R(\di^{2} p))+\frac{\lambda_{0}}{2}\|\alpha_{0}\|_{H^{1}(\om)}^{2}
 +\frac{\lambda_{1}}{2}\|\alpha_{1}\|_{H^{1}(\om)}^{2},\\
 &\qquad \text{over } (p,\alpha)\in H_{0}^{2}(\om,\mathcal{S}^{d\times d})\times \mathcal{A}_{ad}^{0}\times \mathcal{A}_{ad}^{1},\\
& \text{subject to}\quad  p= \underset{p\in H_{0}^{2}(\om,\mathcal{S}^{d\times d})}{\operatorname{argmin}}  \frac{\beta}{2} \|\Delta p\|_{L^{2}(\om,\mathcal{S}^{d\times d})}^{2} +\frac{\gamma}{2} \|p\|_{L^{2}(\om,\mathcal{S}^{d\times d})}^{2} + \frac{1}{2} \|T^{\ast}f-\di^{2}p\|_{B}^{2}\\
 &\hspace{2.5cm}+\frac{1}{\epsilon_{0}} \mathcal{Q}_{\delta}(p,\alpha_{0})+\frac{1}{\epsilon_{1}} \mathcal{P}_{\delta}(\di p,\alpha_{1}).
 \end{aligned}\right.
 \end{equation}
Here, box constraints on $\alpha_i$ are contained in
 \begin{align}
 %\begin{split}
 \mathcal{A}_{ad}^{i}:=\{\alpha_{i}\in H^{1}(\om):\; \underline{\alpha}_{i}\le \alpha_{i} \le \overline{\alpha}_{i}\},\quad i=0,1,\label{A01_ad}
 % \mathcal{A}_{ad}^{1}&=\{\alpha_{1}\in H^{1}(\om):\; \underline{\alpha}_{1}\le \alpha_{1} \le \overline{\alpha}_{1}\},
% \end{split}
 \end{align}
with $\underline{\alpha}_{i},\overline{\alpha}_{i}\in L^{2}(\om)$ and $0<\underline{\epsilon}\le \underline{\alpha}_{i}(x)<\overline{\alpha}_{i}(x)-\overline{\epsilon}$ in $\om$ for some $\underline{\epsilon},\overline{\epsilon}>0$, $i=0,1$. Note that the $H^{1}$ regularity on the parameter functions $\alpha_{0}, \alpha_{1}$ facilitates the existence and differential sensitivity analysis as established in \cite{hintermuellerPartI, hintermuellerPartII} for the TV case. Note, however, that this setting does not guarantee a priori that these functions belong to $C(\overline{\om})$, the regularity required for applying the dualization results of the previous sections. Nevertheless, under mild data assumptions, one can make use of a regularity result of the  $H^{1}$--projection  onto the sets $ \mathcal{A}_{ad}^{0}$ and $ \mathcal{A}_{ad}^{1}$; see  \cite[Corollary 2.3]{hintermuellerPartII}. In particular, if $\underline{\alpha}_{0}$, $\overline{\alpha}_{0}$, $\underline{\alpha}_{1}$, $\overline{\alpha}_{1}$ as well as the initializations for $\alpha_{1}$ and $\alpha_{0}$ are constant functions, then along the projected gradient  iterations, compare Algorithms \ref{alg:bilevelTGV} and \ref{alg:bilevelTGV_pd},  the weights are guaranteed to belong to $H^{2}(\om)$ which (for dimension $d\leq 2$) embeds into $C(\overline{\om})$.

 We briefly note that in the TV case it can be shown \cite{structuralTV, bilevel_handbook} that $W^{1,1}$ regularity for the regularization parameter $\alpha$ suffices to establish a dualization framework. A corresponding result is not yet known for TGV, even though one expects that it could be shown by similar arguments. Hence, here we will also make use of the $H^{1}$--projection regularity result as described above.
 % In fact as we will see in the  numerics for TGV, we will treat only $\alpha_{1}$ as a spatially varying function and $\alpha_{0}$ a constant.
 
Regarding the box constraints \eqref{A01_ad} in \cite{DELOSREYES2016464} it was shown that for a PSNR-optimizing upper level objective $\tilde{J}(u,\alpha)=\|u(\balpha)-f\|_{L^{2}(\om)}^{2}$ subject to $H^{1}$ and Huber regularized TV and TGV  denoising problems, under some mild conditions on the data $f$, the optimal scalar solutions $\alpha$ and $(\alpha_{0},\alpha_{1})$ are strictly positive. As depicted in Figure \ref{fig:psnr_ul_comp} the upper level objective discussed here appears close to optimizing the PSNR, keeping the parameters strictly positive via  \eqref{A01_ad} seems, however, necessary for the time being. 

We now briefly discuss how to treat the bilevel problem \eqref{bilevelTGV}. Let $(\alpha_{0},\alpha_{1})\mapsto p(\alpha_{0},\alpha_{1})$ denote the solution map for the lower level problem, equivalently of the optimality condition \eqref{TGV_EL}. Then the problem \eqref{bilevelTGV} admits the following reduced version
\begin{alignat}{4}
&\min \;\hat{J}_{\mathrm{d}}(\alpha_{0},\alpha_{1})&&:=J_{\mathrm{d}}(p(\alpha_{0},\alpha_{1}),\alpha_{0},\alpha_{1}) \quad &&\text{over }\alpha_{0}\in \mathcal{A}_{ad}^{0},\; \alpha_{1}\in \mathcal{A}_{ad}^{1}.\label{bilevelTGV_reduced}
\end{alignat}
Similarly to the  TV case  \cite{hintermuellerPartI}, one can show that the reduced functional $\hat{J}_{\tgv}: H^{1}(\om)\times H^{1}(\om) \to \RR$ is differentiable.  
%We note that from now on we will consider $\alpha_{0}$ be a scalar, that is, in this case $ \mathcal{A}_{ad}^{0}=\{\alpha_{0}\in \RR:\; \underline{\alpha}_{0}\le \alpha_{0} \le \overline{\alpha}_{0}\}$ where $0<\underline{\alpha}_{0}, \overline{\alpha}_{0}\in \RR$ and also we set $\lambda_{0}=0$ in \eqref{bilevelTGV}.
We can then apply the KKT framework in Banach space \cite{Zowe1979}: 
\begin{equation}\label{con_min_Banach}\left \{
\begin{aligned}
&\text{minimize}\quad  T(\mathrm{x})\quad\text{over }\mathrm{x}\in X,\\
&\text{subject to}\quad \mathrm{x}\in \mathcal{C} \text{ and } g(\mathrm{x})=0,
\end{aligned}
\right.
\end{equation} 
where $V,\mathcal{A}, Z$ are Banach spaces, $X=V\times \mathcal{A}$,  $T:X\to \RR$ and $g:X\to Z$ are Fr\'echet differentiable and continuous differentiable functions, respectively, and $\mathcal{C}\subset X$ is a non-empty, closed convex set. In the bilevel TGV problem \eqref{bilevelTGV} we have $V=H_{0}^{2}(\om,\mathcal{S}^{d\times d})$, $\mathcal{A}=H^{1}(\om)\times H^{1}(\om)$, $Z=V^{\ast}$, $\mathcal{C}=V\times \mathcal{A}_{ad}^{0}\times \mathcal{A}_{ad}^{1}$, $\mathrm{x}=(p,\alpha_{0},\alpha_{1})$, $T(\mathrm{x})=J_{\tgv}(p,\alpha_{0},\alpha_{1})$ and
\[g(\mathrm{x})=g_{\mathrm{d}}(\mathrm{x}):=\beta \Delta^{2}p+\gamma p +\nabla^{2}B^{-1}\di^{2}p-\nabla^{2}B^{-1}T^{\ast}f+\frac{1}{\epsilon_{0}}Q_{\delta}(p,\alpha_{0})-\frac{1}{\epsilon_{1}}\nabla P_{\delta}(\di p,\alpha_{1}).\]

Similarly to \cite{hintermuellerPartI}, for an optimal triplet $(\tilde{p},\tilde{\alpha}_{0},\tilde{\alpha}_{1})$ we can further show that there exists an adjoint variable $q\in H_{0}^{2}(\om,\mathcal{S}^{d\times d})$ (Lagrange multiplier) satisfying the following: 
\begin{align}
\begin{split}
\langle (\di^{2})^{\ast} J_{0}'(\di^{2}\tilde{p},p) \rangle_{V^{\ast},V}
+\langle \beta \Delta q +\gamma q+ \nabla^{2}B^{-1}\di^{2}q
+\frac{1}{\epsilon_{0}} D_{1} Q_{\delta}(\tilde{p},\tilde{\alpha}_{0})q&\\
-\frac{1}{\epsilon_{1}}D_{1}\nabla P_{\delta}(\tilde{p},\tilde{\alpha}_{1})q,p
 \rangle_{V^{\ast},V}&=0,\label{TGV_as_2a}
 \end{split}\\
 \langle  \lambda_{1} (-\Delta+I)\tilde{\alpha}_{1}-\frac{1}{\epsilon_{1}}(D_{2}\nabla P_{\delta}(\tilde{p},\tilde{\alpha}_{1}))^{\ast}q,\alpha_{1}-\tilde{\alpha}_{1}\rangle_{H^{1}(\om)^{\ast}, H^{1}(\om)}&\ge 0,\label{TGV_as_2b}\\
 \langle \lambda_{0}(-\Delta+I)\tilde{\alpha}_{0}+ \frac{1}{\epsilon_{0}} (D_{2} Q_{\delta}(\tilde{p},\tilde{\alpha}_{0}))^{\ast}q,\alpha_{0}-\tilde{\alpha}_{0} \rangle_{H^{1}(\om)^{\ast}, H^{1}(\om)}&\ge 0,\label{TGV_as_2c}
\end{align}
for all $p\in V$, $\alpha_{0}\in \mathcal{A}_{ad}^{0}$ and $\alpha_{1}\in\mathcal{A}_{ad}^{1}$. Here we have used the notation $J_{0}:=F(R\cdot)$, and $D_{1}$ as well as $D_{2}$ denote derivatives with respect to the first and second arguments, respectively. The derivative of the reduced objective is then computed 
as 
\begin{align}
\begin{split}
\hat{J}_{\mathrm{d}}'(\alpha_{0},\alpha_{1})
&=(\lambda_{1}(-\Delta+I)\alpha_{1},\lambda_{0}(-\Delta+I)\alpha_{0})\\
 &\;\quad+ \left(\frac{1}{\epsilon_{0}} (D_{2} Q_{\delta}(\tilde{p},\alpha_{0})), -\frac{1}{\epsilon_{1}}(D_{2}\nabla P_{\delta}(\tilde{p},\alpha_{1}))\right)^{\ast}q(\alpha_{0},\alpha_{1}),\label{TGV_reduced}
 \end{split}
\end{align}
where again $q(\alpha_{0},\alpha_{1})$ solves \eqref{TGV_as_2a} for $\tilde{\alpha}_{0}=\alpha_{0}$, $\tilde{\alpha}_{1}=\alpha_{1}$ and $\tilde{p}=p(\alpha_{0},\alpha_{1})$.

We have $\hat{J}_{\mathrm{d}}'(\alpha_{0},\alpha_{1})\in (H^{1}(\om)\times H^{1}(\om))^{\ast}$.  In order to obtain the gradient of this functional we apply the inverse Riesz map as follows:
\begin{align}
\nabla \hat{J}_{\mathrm{d}}(\alpha_{0},\alpha_{1})&:=\left(\mathcal{R}_{H^{1}}^{-1} P_{1}\hat{J}_{\mathrm{d}}'(\alpha_{0},\alpha_{1}),\mathcal{R}_{H^{1}}^{-1} P_{2}\hat{J}_{\mathrm{d}}'(\alpha_{0},\alpha_{1})\right) \in H^{1}(\om)\times H^{1}(\om),
\end{align}
where for $(r_{1},r_{2})\in H^{1}(\om)\times H^{1}(\om)$ we have
\[\hat{J}_{\mathrm{d}}'(\alpha_{0},\alpha_{1})[r_{1},r_{2}]= P_{1}\hat{J}_{\mathrm{d}}'(\alpha_{0},\alpha_{1}) [r_1] +  P_{2}\hat{J}_{\mathrm{d}}'(\alpha_{0},\alpha_{1})[r_2],\]
with $P_{1}, P_{2}$ denoting the first and the second component of the derivative of the reduced objective. Equipped with this gradient, a gradient-related descent scheme as in \cite[Algorithm 1]{hintermuellerPartII} can be set up for our bilevel TGV problem. This will be discussed further in Section \ref{sec:numerics_algorithm_dual} below.

\subsection{An MPEC}

Utilizing the primal-dual first-order optimality characterization \eqref{reg_opt1}--\eqref{reg_opt4} of the solution to the lower level problem, we arrive at the following {\it mathematical program with equilibrium constraints} (MPEC, for short):
 \begin{equation}\label{bilevelTGV_pd}\tag{$\mathbb{P}_{\tgv}$-p.d.}\left \{
 \begin{aligned}
 &\min \;J_{\mathrm{pd}}(u,\alpha_{0},\alpha_{1}):= F(\mathrm{R}(u))+\frac{\lambda_{0}}{2}\|\alpha_{0}\|_{H^{1}(\om)}^{2}
 +\frac{\lambda_{1}}{2}\|\alpha_{1}\|_{H^{1}(\om)}^{2},\\
&\qquad \text{over } (u,\alpha_{0},\alpha_{1})\in H^{1}(\om)\times \mathcal{A}_{ad}^{0}\times \mathcal{A}_{ad}^{1},\\
& \begin{aligned}
	\text{ subject to }\qquad \qquad\qquad
	Bu-\mu \Delta u+\nabla^{\ast}q -T^{\ast}f&=0, \\
\alpha w- \alpha \Delta w -q+E^{\ast}p&=0, \\
\mathrm{max}_{\delta} (|\nabla u -w|,\gamma_{1})q-\alpha_{1}(\nabla u-w)&=0, \\
\mathrm{max}_{\delta}(|Ew|,\gamma_{0})p-\alpha_{0}Ew&=0.
	\end{aligned}
 \end{aligned}\right.
%& \text{such that}\quad Bu-\mu \Delta u+\nabla^{\ast}q -T^{\ast}f=0, \\
%&\hspace{2.0cm} \alpha w- \alpha \Delta w -q+E^{\ast}p=0,\\
%&\hspace{2.0cm} \mathrm{max}_{\delta} (|\nabla u -w|,\gamma_{1})q-\alpha_{1}(\nabla u-w)=0,\\
%&\hspace{2.0cm} \mathrm{max}_{\delta}(|Ew|,\gamma_{0})p-\alpha_{0}Ew=0,
% \end{aligned}\right.
 \end{equation}
 In order to avoid constraint degeneracy and for the sake of differentiability, we employ here a smoothed version $\max_{\delta}(\cdot,\gamma)$ of $\max$ and its derivative, denoted by $\mathcal{X}_{\delta}$, defined as follows for $r\ge 0$ and for $\frac{\delta}{2}<\gamma$:

\begin{equation*}
\mathrm{max}_{\delta}(r,\gamma)=
\begin{cases}
\gamma \\
\frac{1}{2\delta} (r+\frac{\delta}{2}-\gamma)^{2} +\gamma,\\
r  
\end{cases}
\quad 
\mathcal{X}_{\delta}(r,\gamma)=
\begin{cases}
0 &\;\;\;\; \text{ if } r\le \gamma-\frac{\delta}{2},\\
\frac{1}{\delta} (r+\frac{\delta}{2}-\gamma) &\;\;\;\; \text{ if } \gamma-\frac{\delta}{2} < r< \gamma+\frac{\delta}{2},\\
1 &\;\;\;\; \text{ if } r> \gamma+\frac{\delta}{2}.
\end{cases}
\end{equation*}

%\begin{equation*}
%\mathrm{max}_{\delta}(r,\gamma)=
%\begin{cases}
%\gamma & \text{ if } r\le \gamma-\frac{\delta}{2},\\
%\frac{1}{2\delta} (r+\frac{\delta}{2}-\gamma)^{2} +\gamma& \text{ if } \gamma-\frac{\delta}{2} < r< \gamma+\frac{\delta}{2},\\
%r  & \text{ if } r> \gamma+\frac{\delta}{2},
%\end{cases}
%\end{equation*}
%
%\begin{equation*}
%\mathcal{X}_{\delta}(r,\gamma)=
%\begin{cases}
%0 & \text{ if } r\le \gamma-\frac{\delta}{2},\\
%\frac{1}{\delta} (r+\frac{\delta}{2}-\gamma) & \text{ if } \gamma-\frac{\delta}{2} < r< \gamma+\frac{\delta}{2},\\
%1 & \text{ if } r> \gamma+\frac{\delta}{2}.
%\end{cases}
%\end{equation*}

We treat \eqref{bilevelTGV_pd} similarly to \eqref{bilevelTGV} via the KKT  framework, with $V=H^{1}(\om)$, $\mathcal{A}$, $\mathcal{C}$ as before, $X=H^{1}(\om)\times H^{1}(\om,\RR^{d})\times L^{2}(\om,\RR^{d}), L^{2}(\om,\mathcal{S}^{d\times d})$ and $Z=H^{1}(\om)^{\ast}\times H^{1}(\om,\RR^{d})^{\ast}\times L^{2}(\om,\RR^{d}), L^{2}(\om,\mathcal{S}^{d\times d})$. Here $g_{\mathrm{pd}}: X\to Z$ is defined by the optimality conditions \eqref{reg_opt1}--\eqref{reg_opt4}.

We will skip here the  proofs for the differentiability of the functions $g$ and the reduced objective $J$ as well as the existence proofs for \eqref{bilevelTGV} and \eqref{bilevelTGV_pd}. These results can be shown similarly to the corresponding assertions for TV; see \cite{hintermuellerPartI, hintermuellerPartII}.

\subsection{Newton solvers for the lower level problems}

\subsubsection{Dual TGV Newton}
Before we proceed to devising of a projected gradient algorithm for the solution of both aforementioned bilevel problems, we  discuss here two Newton algorithms for the solutions of the corresponding lower level problems. 

We first state the corresponding function space Newton method for the solution of  \eqref{TGV_EL}; see Algorithm \ref{TGV_Newton}. 

\begin{algorithm}[!h]
  \caption{\newline Function space Newton algorithm for the solution of the regularized TGV dual problem \eqref{TGV_predual_Newton} \label{TGV_Newton}}
  \begin{algorithmic}
    %\Statex {\textbf{Input}:  }
    %\Statex {\textbf{Initialise}: }
    \While{some stopping criterion is not satisfied}
     \Statex{Find $\delta p^{k}\in H_{0}^{2}(\om,\mathcal{S}^{d\times d})$ such that the following equation is satisfied in $[H_{0}^{2}(\om,\mathcal{S}^{d\times d})]^{\ast}$}:
     \Statex{
     \begin{align*}
     \nabla^{2} B^{-1}\di^{2} \delta p^{k}+\beta \Delta^{2} \delta p^{k} + \gamma \delta p
     &+\frac{1}{\epsilon_{0}} \left ( G_{\delta}''(p^{k}-\alpha_{0})+ G_{\delta}''(-p^{k}-      \alpha_{0})\right)\delta p^{k}\\
     & -\frac{1}{\epsilon_{1}}\nabla \left ( G_{\delta}''(\di p^{k}-\alpha_{1})+ G_{\delta}''(-\di p^{k}-\alpha_{1})\right)\di \delta p^{k}
     =-G(p^{k}),
     \end{align*}
     }
      \Statex{Update $p^{k+1}$:}
       \Statex{\[\qquad p^{k+1}=p^{k}+\delta p^{k}\]}
 \EndWhile
  \end{algorithmic}
\end{algorithm}
Here $G_{\delta}''$ denotes the second derivative of $G_{\delta}$ in \eqref{G_d}. Due to the regularization of $p$ in \eqref{TGV_EL} the algorithm admits a local superlinear convergence; see \cite{HIK,HiKu_pf}. Moreover, similar to \cite{HiUl} it can be shown that the solver is mesh (i.e. image resolution) independent.
%\begin{equation}
%G_{\delta}''(t)=
%\begin{cases}
%1, & \text{ if } t\ge \delta,\\
%\frac{t}{\delta}, & \text{ if } 0<t<\delta,\\
%0,& \text{ if } t\le 0,
%\end{cases}
%\end{equation}
%with that function applied pointwise. %Here $\mathbf{1}$ denotes  the vector with values equal to one. Its multiplication with $\alpha_{0}$ and $\alpha_{1}$ is regarded pointwise.

%We mention a few more words about the discrete version  of Algorithm \ref{TGV_Newton}.
 A few words on the discrete version of  Algorithm \ref{TGV_Newton} are in order. Images ($d=2$) are considered as elements of $U_{h}:=\{u\, |\,u:\om_{h}\to \RR\}$ where $\om_{h}=\{1,2,\ldots,n\}\times \{1,2,\ldots,m\}$ is a discrete cartesian grid that corresponds to the image pixels. The mesh size, defined as the distance between the grid points, is set to $h=1/\sqrt{nm}$.
We define the associated discrete function spaces
$W_{h}=U_{h}\times U_{h}$, $V_{h}=U_{h}\times U_{h} \times U_{h}$,
so that  $p\in V_{h}$ with  $p=(p^{11}, p^{12}, p^{22})$.
For the discrete gradient and divergence we have, $\nabla: W_{h}\to V_{h}$ and $\di: V_{h}\to W_{h}$ satisfying the adjoint relation $\nabla=-\di^{\top}$. We refer the reader to Appendix \ref{sec:app} for precise definitions of these operators as well as for a detailed description of the other discrete second- order differential operators, $\nabla^{2}:U_{h}\to V_{h}$, $\di^{2}: V_{h}\to U_{h}$, the vector bi-Laplacian $\Delta^{2}: V_{h}\to V_{h}$, as well as the operator $\nabla^{2}\di ^{2}:V_{h}\to V_{h}$. We note here that these operators must be defined with the correct boundary conditions in order to reflect the boundary conditions imposed on $p\in H_{0}^{2}(\om,\mathcal{S}^{2\times 2})$.

\subsubsection{Primal-Dual TGV Newton}

Next we briefly describe the primal-dual TGV Newton method for the solution of the first-order optimality conditions in Proposition \ref{lbl:optimality_reg2} written here for the denoising case, for the sake of readability only:
\begin{align}
u-\mu \Delta u - \di q -f&=0, \label{lbl:opt1}\\
\alpha w -\alpha \Delta w -q-\di p&=0, \label{lbl:opt2}\\
\mathrm{max}_{\delta}(|\nabla u -w|,\gamma_{1})q - \alpha_{1} (\nabla u -w)&=0, \label{lbl:opt3}\\
\mathrm{max}_{\delta} (|Ew|,\gamma_{0})p-\alpha_{0}E w&=0.\label{lbl:opt4}
\end{align}

For the discretized versions of the above differential operators, we use the standard five-point stencils with zero Neumann boundary conditions. Note that  these act on the primal variables $u$ and $w$,  which satisfy natural boundary conditions in contrast to the dual variable. The discretized symmetrized gradient $Ew$ is defined as $\frac{1}{2}(\nabla  w+ (\nabla w)^{\top})$.

The system of equations \eqref{lbl:opt1}--\eqref{lbl:opt4} can be shortly written as $g_{\mathrm{pd}}(\mathrm{x})=0$, where $\mathrm{x}=(u,w,q,p)$. We compute the  derivative of $g_{\mathrm{pd}}$ at a point $\mathrm{x}=(u,w,q,p)$ as the following block-matrix:

\[Dg_{\mathrm{pd}}(\mathrm{x})=Dg_{\mathrm{pd}}(u,w,q,p)=
\begin{bmatrix}
   A    & B \\
    C & D
\end{bmatrix},
\]
where
\begin{equation}\label{ABD}
A=
\begin{bmatrix}
   I-\mu\Delta    & 0 \\
    0 & \alpha(I-\Delta)
\end{bmatrix},\;\;
B=
\begin{bmatrix}
   -\di     & 0 \\
    -I & -\di
\end{bmatrix},\;\;
D=
\begin{bmatrix}
   \mathrm{max}_{\delta}(|\nabla u-w |,\gamma_{1}) & 0 \\
    0 & \mathrm{max}_{\delta}(|Ew|,\gamma_{0})
\end{bmatrix},
\end{equation}
\begin{equation}\label{CCC}
C=
\begin{bmatrix}
-\alpha_{1}\nabla+q\mathcal{X}_{\delta}(|\nabla u-w|, \gamma_{1}) \frac{\nabla u-w}{|\nabla u-w|} \cdot \nabla
& \alpha_{1}I+q\mathcal{X}_{\delta}(|\nabla u-w|, \gamma_{1}) \frac{\nabla u-w}{|\nabla u-w|}\cdot(-I)\\
0    & -\alpha_{0}E+p\mathcal{X}_{\delta}(|Ew|, \gamma_{0})\frac{Ew}{|Ew|}\cdot E    \\ 
\end{bmatrix}.
\end{equation}
Given $\mathrm{x}^k$, the Newton iteration for solving the system of equations  \eqref{lbl:opt1}--\eqref{lbl:opt4}, or $g_{\mathrm{pd}}(\mathrm{x})=0$ for short, reads
\[\mathrm{x}^{k+1}=\mathrm{x}^{k}-D\mathcal{F}(\mathrm{x}^{k})^{-1}\mathcal{F}(\mathrm{x}^{k}),\]
which can also be written as 
\begin{equation}\label{snn1_linear_system}
Dg_{\mathrm{pd}}(\mathrm{x}^{k})\mathrm{x}^{k+1}=Dg_{\mathrm{pd}}(\mathrm{x}^{k})\mathrm{x}^{k}-g_{\mathrm{pd}}(\mathrm{x}^{k}).
\end{equation}
Here it is convenient to introduce the notation
\[Dg_{\mathrm{pd}}(\mathrm{x}^{k})=Dg_{\mathrm{pd}}(u^{k},w^{k},q^{k},p^{k})=
\begin{bmatrix}
   A    & B \\
    C_{k} & D_{k}
\end{bmatrix}
\]
since only the submatrices $C$ and $D$ depend on $k$. Note that the righthand side $Dg_{\mathrm{pd}}(\mathrm{x}^{k})\mathrm{x}^{k}-g_{\mathrm{pd}}(\mathrm{x}^{k})$ of the linear system \eqref{snn1_linear_system} can be written as 
\[Dg_{\mathrm{pd}}(\mathrm{x}^{k})\mathrm{x}^{k}-\mathcal{F}(\mathrm{x}^{k})=
\left (
\begin{array}{c}
b_{1}^{k}\\
b_{2}^{k}
\end{array}
\right ),
\]
where
\[b_{1}^{k}=\left (f,0 \right)^\top,\quad \text{ and }\quad b_{2}^{k}=\left (q^{k}\mathcal{X}_{\delta}(|\nabla u^{k}-w^{k}|, \gamma_{1})|\nabla u^{k}-w^{k}|, p^{k} \mathcal{X}_{\delta}(|Ew^{k}|, \gamma_{0}) |Ew^{k}| \right)^\top.\]
Notation-wise, the components that appear in $b_{2}^{k}$ should be regarded as the diagonals of the 
corresponding diagonal matrices that we mentioned before, multiplied component-wise. By introducing the notation $\mathrm{x}_{1}^{k}=(u^{k},w^{k})^\top$, $\mathrm{x}_{2}^{k}=(q^{k},p^{k})^\top$, the Newton system \eqref{snn1_linear_system} can be written as 
\begin{equation}\label{snn1_linear_system_detailed}
\begin{bmatrix}
   A    & B \\
    C_{k} & D_{k}
\end{bmatrix}
\left (
\begin{array}{c}
\mathrm{x}_{1}^{k+1}\\
\mathrm{x}_{2}^{k+1}
\end{array}
\right )=
\left (
\begin{array}{c}
b_{1}^{k}\\
b_{2}^{k}
\end{array}
\right ).
\end{equation}
The above system can be simplified utilizing the Schur complement: First solve for the primal variables $\mathrm{x}_{1}^{k+1}=(u^{k+1},w^{k+1})$ and then recover the dual ones $\mathrm{x}_{2}^{k+1}=(q^{k+1},p^{k+1})$. This yields
\begin{align*}
(A-B D_{k}^{-1} C_{k})\mathrm{x}_{1}^{k+1}&= b_{1}^{k}-B D_{k}^{-1}b_{2}^{k},\\
\mathrm{x}_{2}^{k+1}&=D_{k}^{-1}(b_{2}^{k}-C_{k}\mathrm{x}_{1}^{k+1}).
\end{align*}
The folllowing result then holds.
\newtheorem{pos_def}[Wdiv]{Lemma}
\begin{pos_def}\label{lbl:pos_def}
If $(q^{k},p^{k})$ belong to the feasible set, i.e., $|q^{k}|\le \alpha_{1}$ and $|p^{k}|\le \alpha_{0}$ component-wise, then the matrix $S_{k}:=(A-BD_{k}^{-1}C_{k})$ is positive definite and for the minimum eigenvalues we have $\lambda_{\min}(S_{k})\ge \lambda_{\min}(A)>0$. Furthermore, $S_{k}^{-1}$ is bounded independently of $k$.
\end{pos_def}

The proof of Lemma \ref{lbl:pos_def} follows the steps of the analogous proof in \cite{stadler} and is hence omitted. 
Summarizing, the Newton method for the solution of the \eqref{lbl:opt1}-\eqref{lbl:opt4} is outlined in Algorithm \ref{TGV_Newton_pd}. Here we have followed \cite{stadler} and project in every iteration the variables $q,p$ onto the feasible sets such that the result of Lemma \ref{lbl:pos_def} holds.

\begin{algorithm}[!h]
  \caption{\newline Newton algorithm for the solution of the regularized TGV primal problem \eqref{tgv_primal_reg2} \label{TGV_Newton_pd}}
  \begin{algorithmic}
    %\Statex {\textbf{Input}:  }
    %\Statex {\textbf{Initialise}: }
    \While{some stopping criterion is not satisfied}
     \Statex{Solve the linear system for $\mathrm{x}_{1}^{k+1}=(u^{k+1},w^{k+1})$}
     \Statex{
     \[(A-B D_{k}^{-1} C_{k})\mathrm{x}_{1}^{k+1}= b_{1}^{k}-B D_{k}^{-1}b_{2}^{k}\]}
      \Statex{Update $\tilde{\mathrm{x}}_{2}^{k+1}=(\tilde{q}^{k+1},\tilde{p}^{k+1})$ as follows}
       \Statex{\[\tilde{\mathrm{x}}_{2}^{k+1}=D_{k}^{-1}(b_{2}^{k}-C_{k}\mathrm{x}_{1}^{k+1})\]}
        \Statex{Compute $q^{k+1}$, $p^{k+1}$ as  projections of $\tilde{q}^{k+1}$, $\tilde{p}^{k+1}$  onto the feasible sets $\{q:|q|\le \alpha_{1}\}$, $\{p:|p|\le \alpha_{0}\}$} 
 \EndWhile
  \end{algorithmic}
\end{algorithm}
The projections onto the feasible sets  are defined respectively as
\begin{equation}\label{projections}
q=\frac{\tilde{q}}{\max\left \{1, \frac{|\tilde{q}|}{\alpha_{1}} \right \}},\quad p=\frac{\tilde{p}}{\max\left \{1, \frac{|\tilde{p}|}{\alpha_{0}} \right \}},
\end{equation}
with the equalities above to be considered component-wise. 
%$|\cdot|$ denote the Euclidean norm of the vectors $\tilde{q}$, $\tilde{p}$ also considered pointwise.

\section{Numerical implementation}\label{sec:numerics}
In this section we will describe two projected gradient algorithms for the solution of the discretized versions of the two bilevel problems \eqref{bilevelTGV} and \eqref{bilevelTGV_pd}. Note that for most of the experiments we will keep $\alpha_{0}$ a scalar -- this is justified by the numerical results; see the relevant discussion later on. 

\subsection{The numerical algorithm for \eqref{bilevelTGV}}\label{sec:numerics_algorithm_dual}

We now describe our strategy for solving  the discretized version of the bilevel TGV problem \eqref{bilevelTGV}.
For this purpose, we introduce the discrete versions of differential operators and norms that appear in the upper level objective of \eqref{bilevelTGV}. We will make use of the discrete Laplacian with zero Neumann boundary conditions $\Delta_{N}: U_{h}\to U_{h}$ which is used to act on the weight function $\alpha_{1}$. These are the desired boundary conditions for $\alpha_{1}$ as dictated by the regularity result for the $H^{1}$--projection  in \cite[Corollary 2.3]{hintermuellerPartII}.  For that we use the standard Laplacian stencil, setting the function values of ghost grid points to be the same with  the function value of the nearest grid point in $\om_{h}$. 
%\begin{figure}[!h]
%\begin{tikzpicture}
%\node at (9.0,0.2) {$\frac{1}{h^{2}}\;\times$};
%  \node[shape=circle,draw, minimum size=13pt, inner sep=0pt, fill=white]  at (9.8,0.2) {\tiny{1}};
% \node[shape=circle,draw, minimum size=13pt, inner sep=0pt, fill=mygray]  at (10.4,0.2) {\tiny{-4}};
% \node[shape=circle,draw, minimum size=13pt, inner sep=0pt, fill=white]  at (11.0,0.2) {\tiny{1}};
% \node[shape=circle,draw, minimum size=13pt, inner sep=0pt, fill=white]  at (10.4,-0.4) {\tiny{1}};
% \node[shape=circle,draw, minimum size=13pt, inner sep=0pt, fill=white]  at (10.4,0.8) {\tiny{1}};
% \end{tikzpicture}
% %\caption{Finite difference stencils that constitute the discrete bi-Laplacian $\Delta^{2}$}
% %\label{Stencil_biLap}
% \end{figure}
%
%\noindent
%setting the function values of ghost grid points to be the same with  the function value of the nearest grid point in $\om_{h}$.
 For a function $u\in U_{h}$ we define the discrete $\ell^{2}$ norm as 
\[\|u\|_{\ell^{2}(\om_{h})}^{2}=h^{2}\sum_{(i,j)\in\om_{h}} |u_{i,j}|^{2}.\]
For the discrete  $H^{1}$ norm applied to the weight function $\alpha_{1}$ we use
\[\|\alpha_{1}\|_{H^{1}(\om_{h})}=h\sqrt{\alpha_{1}^{\top} (I-\Delta_{N})\alpha_{1}},\]
while the dual norm is defined as 
\[\|r\|_{H^{1}(\om_{h})^{\ast}}=\|(I-\Delta_{N})^{-1}r\|_{H^{1}(\om_{h})}=h\sqrt{r^{\top}(I-\Delta_{N})^{-1}r}\]
based on the $H^{1}\to H^{1}(\om)^{\ast}$ Riesz map $\alpha\mapsto r=(I-\Delta_{N})\alpha$. We will also make use if the following version of the discrete dual $H_{0}^{2}(\om_{h})^{\ast}$:
\[\|v\|_{H_{0}^{2}(\om_{h})^{\ast}}=h\sqrt{v^{\top} (I+\Delta^{2})^{-1}v}.\]
For the discrete version of the averaging filter in the definition of the localized residuals \eqref{loc_res} we use a filter of size $n_{w}\times n_{w}$, with entries of equal value whose sum is equal to one. With these definitions the discrete version of the bilevel TGV \eqref{bilevelTGV} is the following:

  \begin{equation}\tag{$\mathbb{P}_{\tgv}^{\,h}$-d}\label{bilevelTGV_h}\left \{
 \begin{aligned}
 &\text{minimize } \frac{1}{2}\|(R(\di^{2} p)-\overline{\sigma}^{2})^{+}\|_{\ell^{2}(\om_{h})}^{2}
 +\frac{1}{2}\|(\underline{\sigma}^{2}-R(\di^{2} p))^{+}\|_{\ell^{2}(\om_{h})}^{2}
 +\frac{\lambda}{2}\|\alpha_{1}\|_{H^{1}(\om_{h})}^{2},\\
 &\qquad \text{over } (p,\alpha_{0},\alpha_{1})\in V_{h}\times (\mathcal{A}_{ad}^{0})_{h}\times (\mathcal{A}_{ad}^{1})_{h},\\
 &\text{subject to}\quad \beta\Delta^{2}p+\gamma p+\nabla^{2} B^{-1}\di^{2}p  -\nabla^{2}B^{-1}  T^{\ast}f+\frac{1}{\epsilon_{0}} Q_{\delta}(p,\alpha_{0})-\frac{1}{\epsilon_{1}}\nabla P_{\delta}(\di p,\alpha_{1})=0.
 \end{aligned}\right.
 \end{equation}
Here, $(\cdot)^+$ is applied in a component-wise way and we  have 
 \begin{align*}
  (\mathcal{A}_{ad}^{0})_{h}&=
  \{\alpha_{0}\in \RR: \underline{\alpha}_{0}\le \alpha_{0}\le \overline{\alpha}_{0}\},\\
 (\mathcal{A}_{ad}^{1})_{h}&=
 \{\alpha\in U_{h}:\; \underline{\alpha}_{1}\le (\alpha_{1})_{i,j} \le \overline{\alpha}_{1},\; \text{for all }(i,j)\in \om_{h}\}.\label{A_ad_h_tgv}
 \end{align*}
 Note that the discrete penalty functions $P_{\delta}: W_{h}\to W_{h}$ and $Q_{\delta}: V_{h}\to V_{h}$ are defined straightforwardly by componentwise application of the function $G_{\delta}'$. 
 
 Regarding the choice of the lower and upper bounds for the local variance $\underline{\sigma}^{2}$ and $\overline{\sigma}^{2}$, respectively, we follow here the following rules, where $\sigma^{2}$ is the variance of the ``Gaussian''  noise contaminating the data:
\begin{equation}\label{sigma_choice}
\overline{\sigma}^{2}=\sigma^{2}\left (1+\frac{\sqrt{2}}{n_{w}} \right ), \qquad \underline{\sigma}^{2}=\sigma^{2}\left (1-\frac{\sqrt{2}}{n_{w}} \right ).
\end{equation}
 The formulae \eqref{sigma_choice} are based on the statistics of the extremes; see \cite[Section 4.2.1]{hintermuellerPartII}.

We now proceed by describing the algorithm for the numerical solution of \eqref{bilevelTGV_h}. In essence, we employ a discretized projected gradient method with Armijo line search. The discrete gradient of the reduced objective functional is computed with the help of the adjoint equation which is the discrete version of \eqref{TGV_as_2a}. We summarize this in Algorithm \ref{alg:bilevelTGV}.

 \begin{algorithm}[!h]
  \caption{\newline Discretized projected gradient method for the bilevel TGV problem \eqref{bilevelTGV_h} \label{alg:bilevelTGV}}
  \begin{algorithmic}
    \Statex {\textbf{Input}: $f$, $\underline{\alpha}_{0}$, $\overline{\alpha}_{0}$, $\underline{\alpha}_{1}$, $\overline{\alpha}_{1}$, $\overline{\sigma}$, $\underline{\sigma}$, $\lambda$, $\beta$, $\gamma$, $\epsilon_{0}$, $\epsilon_{1}$, $\delta$, $n_{w}$ $\tau_{0}^{0}$, $\tau_{1}^{0}$, $0<c<1$, $0<\theta_{-}<1\le \theta_{+}$}
    \Statex{\textbf{Initialize}: $\alpha_{0}^{0}\in(\mathcal{A}_{ad}^{0})_{h}$, $\alpha_{1}^{0}\in(\mathcal{A}_{ad}^{1})_{h}$  and set $k=0$.}
 \Repeat
     \State{Use Algorithm \ref{TGV_Newton} to compute the solution $p^{k}$ of the lower level problem 
     \[g_{\mathrm{d}}(p^{k},\alpha_{0}^{k}, \alpha_{1}^{k}):=\beta\Delta^{2}p^{k}+\gamma p^{k}+ \nabla^{2}B^{-1}\di^{2}p^{k} -\nabla^{2} B^{-1} T^{\ast}f+\frac{1}{\epsilon_{0}} Q_{\delta}(p^{k},\alpha_{0}^{k})-\frac{1}{\epsilon_{1}}\nabla P_{\delta}(\di p^{k},\alpha_{1}^{k})=0\]}
    \State{Solve the adjoint equation for $q^{k}$
    \begin{align*}
    \beta \Delta^{2} q^{k}+\gamma q^{k}+\nabla^{2} B^{-1}\di^{2} q^{k}
    &+\frac{1}{\epsilon_{0}} \left (G_{\delta}''(p^{k}-\mathbf{1}\alpha_{0}^{k})+G_{\delta}''(-p^{k}-\mathbf{1}\alpha_{0}^{k}) \right)q^{k}\\
    &-\frac{1}{\epsilon_{1}}  \nabla \left ( G_{\delta}''(\di p^{k}-\mathbf{1}\alpha_{1}^{k})+G_{\delta}''(-\di p^{k}-\mathbf{1}\alpha_{1}^{k})\right )\di q^{k}\\
    &=-2\nabla B^{-1}T^{\ast}\di^{2}p^{k}\left (w\ast\left((R(\di^{2} p^{k})-\overline{\sigma}^{2})^{+}-(\underline{\sigma}^{2}-R(\di^{2} p^{k}))^{+}\right) \right )
    \end{align*}
    } 
  \State{Compute the derivative of the reduced objective with respect to $\alpha_{0}$ and $\alpha_{1}$
  \begin{align*}
  \hat{J}_{\mathrm{d},\alpha_{0}}'(\alpha_{0}^{k},\alpha_{1}^{k})&=\frac{1}{\epsilon_{0}}[\mathbbm{1}\;\; \mathbbm{1} \;\; \mathbbm{1}] \left (-G_{\delta}''(p^{k}-\mathbf{1}\alpha_{0}^{k})+G_{\delta}''(-p^{k}-\mathbf{1}\alpha_{0}^{k})\right )q^{k},\\
  \hat{J}_{\mathrm{d},\alpha_{1}}'(\alpha_{0}^{k},\alpha_{1}^{k})&=-\frac{1}{\epsilon_{1}}[Id \;\; Id]\nabla \left (-G_{\delta}''(\di p^{k}-\mathbf{1}\alpha_{1}^{k})+G_{\delta}''(-\di p^{k}-\mathbf{1}\alpha_{1}^{k}) \right )q^{k}+ \lambda(I-\Delta_{N})\alpha_{1}^{k}.
  \end{align*} 
  } 
  \State{Compute the reduced gradients
  \begin{align*}
  \nabla_{\alpha_{0}} \hat{J}_{\mathrm{d}}(\alpha_{0}^{k},\alpha_{1}^{k})
  &=\hat{J}_{\mathrm{d},\alpha_{0}}'(\alpha_{0}^{k},\alpha_{1}^{k}),\\
  \nabla_{\alpha_{1}} \hat{J}_{\mathrm{d}}(\alpha_{0}^{k},\alpha_{1}^{k})
  &=(I-\Delta_{N})^{-1}\hat{J}_{\mathrm{d}}'(\alpha_{0}^{k},\alpha_{1}^{k})
  \end{align*}
  }

\State{Compute the trial points 
\begin{align*}
\alpha_{i}^{k+1}&=P_{(\mathcal{A}_{ad}^{i})_{h}}\big(\alpha_{i}^{k}-\tau_{i}^{k}  \nabla_{\alpha_{i}} \hat{J}_{\mathrm{d}}(\alpha_{0}^{k},\alpha_{1}^{k})\big), \quad i=0,1
%\alpha_{1}^{k+1}&=P_{(\mathcal{A}_{ad}^{1})_{h}}\big(\alpha_{1}^{k}-\tau_{1}^{k}  \nabla_{\alpha_{1}} \hat{J}_{\mathrm{d}}(\alpha_{0}^{k},\alpha_{1}^{k})\big)
\end{align*}
}
\While{
\begin{align*}
\hat{J}_{\mathrm{d}}(\alpha_{0}^{k+1},\alpha_{1}^{k+1})
&> \hat{J}_{\mathrm{d}}(\alpha_{0}^{k},\alpha_{1}^{k})\\
&+c\left (\hat{J}_{\mathrm{d},\alpha_{0}}'(\alpha_{0}^{k},\alpha_{1}^{k})^{\top}(\alpha_{0}^{k+1}-\alpha_{0}^{k})
+\hat{J}_{\mathrm{d},\alpha_{1}}'(\alpha_{0}^{k},\alpha_{1}^{k})^{\top}(\alpha_{1}^{k+1}-\alpha_{1}^{k})\right)
\end{align*}
}{ (Armijo line search)}
\State{Set $\tau_{0}^{k}:=\theta_{-}\tau_{0}^{k}$, $\tau_{1}^{k}:=\theta_{-}\tau_{1}^{k}$ and re-compute 
\begin{align*}
\alpha_{i}^{k+1}&=P_{(\mathcal{A}_{ad}^{i})_{h}}\big(\alpha_{i}^{k}-\tau_{i}^{k}  \nabla_{\alpha_{i}} \hat{J}_{\mathrm{d}}(\alpha_{0}^{k},\alpha_{1}^{k})\big),\quad i=0,1
%\alpha_{1}^{k+1}&=P_{(\mathcal{A}_{ad}^{1})_{h}}\big(\alpha_{1}^{k}-\tau_{1}^{k}  \nabla_{\alpha_{1}} \hat{J}_{\mathrm{d}}(\alpha_{0}^{k},\alpha_{1}^{k})\big)
\end{align*}
}
  \EndWhile
  \State{Update $\tau_{0}^{k+1}=\theta_{+}\tau_{0}^{k}$, $\tau_{1}^{k+1}=\theta_{+}\tau_{1}^{k}$ and $k:=k+1$}
 \Until{some stopping condition is satisfied}
  \end{algorithmic}
\end{algorithm}

For the sake of notation, here $\bf{1}$ denotes a matrix either of the form $[Id;Id]$ or $[Id;Id;Id]$ of size $nm\times 2nm$ or $nm\times 3nm$, respectively, depending on whether it is applied on $\alpha_{1}$ or $\alpha_{0}$. On the other hand, $\mathbbm{1}$ denotes a matrix of size $1\times nm$ with all entries equal to one. The projection $P_{(\mathcal{A}_{ad}^{1})_{h}}$ is computed as described in \cite[Algorithm 4]{hintermuellerPartII}, that is via the semismooth Newton method developed in \cite{HiKu_pf}. We only mention   that  the original discretized $H^{1}$--projection problem $P_{(\mathcal{A}_{ad})_{h}}(\tilde{\alpha})$ given by
\begin{equation}\label{H1_proj_h}
\left\{
\begin{aligned}
&\min \;\frac{1}{2}\|\alpha-\tilde{\alpha}\|_{H^{1}(\om_{h})}^{2}:= \frac{h}{2}(\alpha-\tilde{\alpha})^{\top}(I-\Delta_{N})(\alpha-\tilde{\alpha}),\\
&\text{over} \quad \alpha\in   (\mathcal{A}_{ad})_{h}=\{\alpha\in U_{h}:\; \underline{\alpha}\le \alpha_{i,j} \le \overline{\alpha}\},
\end{aligned}
\right.
\end{equation}
is approximated by the following penalty version: 
\begin{equation}\label{H1_proj_h}
\min_{\alpha\in U_{h}} \;\frac{1}{2} \|\alpha-\tilde{\alpha}\|_{H^{1}(\om_{h})}^{2}
+ \frac{1}{\epsilon_{\alpha}}\left (\frac{1}{2}\|(\alpha-\overline{\alpha})^{+}\|_{\ell^{2}(\om_{h})}^{2}+ \frac{1}{2}\|(\underline{\alpha}-\alpha)^{+}\|_{\ell^{2}(\om_{h})}^{2}\right ),
\end{equation}
with some small $\epsilon_{\alpha}>0$. For the projection regarding $\alpha_{0}$, we simply set $P_{(\mathcal{A}_{ad}^{1})_{h}}(\alpha_{0})=\max(\min(\alpha_{0},\overline{\alpha}_{0}),\underline{\alpha}_{0})$.

Furthermore, a path following scheme is employed for solving $g_{\mathrm{d}}(p,\alpha_{0},\alpha_{1})=0$. This done by using a decaying sequence $\epsilon_{0}=\epsilon_{0}^{\ell}$, $\epsilon_{1}=\epsilon_{1}^{\ell}$ up to a tolerance
 \[g_{\mathrm{d}}(p^{\ell+1},\alpha_{0},\alpha_{1})\le \mathrm{tol}_{(\ell)},\]
 and then setting $\epsilon_{0}^{\ell+1}:=\max(\theta_{\epsilon}\epsilon_{0}^{\ell},\epsilon_{0})$, 
 $\epsilon_{1}^{\ell+1}:=\max(\theta_{\epsilon}\epsilon_{1}^{\ell},\epsilon_{1})$ for some $0<\theta_{\epsilon}<1$, until a desired level of penalization is reached.

\subsection{The numerical algorithm for \eqref{bilevelTGV_pd}}

We now turn our attention to the  discretized bilevel problem  \eqref{bilevelTGV_pd} which, again for simplicity, is here formulated for the denoising case, only. Since for that problem we also report on numerical experiments for spatially varying $\alpha_{0}$ we formulate the problem for this general case:
  \begin{equation}\tag{$\mathbb{P}_{\tgv}^{\,h}$-p.d.}\label{bilevelTGV_h_pd}\left \{
 \begin{aligned}
 &\text{minimize }\; \frac{1}{2}\|(R(u)-\overline{\sigma}^{2})^{+}\|_{\ell^{2}(\om_{h})}^{2}
 +\frac{1}{2}\|(\underline{\sigma}^{2}-R(u))^{+}\|_{\ell^{2}(\om_{h})}^{2}
 +\frac{\lambda_{0}}{2}\|\alpha_{0}\|_{H^{1}(\om_{h})}^{2}
  +\frac{\lambda_{0}}{2}\|\alpha_{1}\|_{H^{1}(\om_{h})}^{2},\\
 &\qquad \text{over } (u,\alpha_{0},\alpha_{1})\in U_{h}\times (\mathcal{A}_{ad}^{0})_{h}\times (\mathcal{A}_{ad}^{1})_{h},\\
 %&\text{subject to the discretized version of }  \eqref{lbl:opt1}-\eqref{lbl:opt4}
 	& \begin{aligned}
	\text{ subject to }\qquad \qquad\qquad
	u-\mu \Delta u - \di q -f&=0, \\
\alpha w -\alpha \Delta w -q-\di p&=0, \\
\mathrm{max}_{\delta}(|\nabla u -w|,\gamma_{1})q - \alpha_{1} (\nabla u -w)&=0, \\
\mathrm{max}_{\delta} (|Ew|,\gamma_{0})p-\alpha_{0}E w&=0.
	\end{aligned}
 \end{aligned}\right.
 \end{equation}
Here the set $(\mathcal{A}_{ad}^{0})_{h}$ is defined similarly to $(\mathcal{A}_{ad}^{1})_{h}$ before.
% \begin{align*}
%  (\mathcal{A}_{ad}^{0})_{h}&=
% \{\alpha_{0}\in U_{h}:\; \underline{\alpha}_{0}\le (\alpha_{0})_{i,j} \le \overline{\alpha}_{0},\; \text{for all }(i,j)\in \om_{h}\}\\
% (\mathcal{A}_{ad}^{1})_{h}&=
% \{\alpha_{1}\in U_{h}:\; \underline{\alpha}_{1}\le (\alpha_{1})_{i,j} \le \overline{\alpha}_{1},\; \text{for all }(i,j)\in \om_{h}\}\label{A_ad_h_tgv}
% \end{align*}
 The constraints in \eqref{bilevelTGV_h_pd} are the discretized versions of \eqref{lbl:opt1}-\eqref{lbl:opt4}, still  denoted by $g_{\mathrm{pd}}(\mathrm{x})=0$. The upper level objective is still denoted by $J_{\mathrm{pd}}$. 
The corresponding discretized adjoint equation  
\[D_{\mathrm{x}}g_{\mathrm{pd}}(\mathrm{x}^{\ast})^{\top}=-D_{\mathrm{x}}J_{\mathrm{pd}}(\mathrm{x}),\]
where $\mathrm{x}^{\ast}:=(u^{\ast}, w^{\ast}, q^{\ast},p^{\ast})$ is the adjoint variable,  reads
\begin{equation}\label{adjoint_pd}
\begin{bmatrix}
   A^{\top}    & C^{\top} \\
    B^{\top} & D^{\top}
\end{bmatrix}
\left (
\begin{array}{c}
u^{\ast}\\
w^{\ast}\\
q^{\ast}\\
p^{\ast}
\end{array}
\right )=
\left (
\begin{array}{c}
-2(u-f)\left(w\ast\left ((R(u)-\overline{\sigma}^{2})^{+}-(\underline{\sigma}^{2}-R(u))^{+} \right)\right)\\
0\\
0\\
0
\end{array}
\right )
:=
\left(
\begin{array}{c}
b_{1}^{\ast}\\
b_{2}^{\ast}
\end{array}
\right ),
\end{equation}
 where the matrices above were defined in \eqref{ABD} and \eqref{CCC}.
The equation can be solved again for $\mathrm{x}_{1}^{\ast}:=(u^{\ast},w^{\ast})$ first and then subsequently for  $\mathrm{x}_{2}^{\ast}:=(q^{\ast},p^{\ast})$ as follows
\begin{align*}
\left( A^{\top} - C^{\top} (D^{\top})^{-1}B^{\top}  \right)x_{1}^{\ast}&=b_{1}^{\ast},\\
x_{2}^{\ast}&= (D^{\top})^{-1}(b_{2}^{\ast}-B^{\top}x_{1}^{\ast}).
\end{align*}
The derivatives of the reduced objective with respect to $\alpha_{0}$ and $\alpha_{1}$, respectively, are
\begin{align}
\hat{J}_{\mathrm{pd},\alpha_{0}}'(\alpha_{0},\alpha_{1})
&=(D_{\alpha_{0}}g_{\mathrm{pd}})^{\top} x^{\ast} + D_{\alpha_{0}}J_{\mathrm{pd}}(\alpha_{0},\alpha_{1})\\
&=
\begin{bmatrix}
   Id   & Id & 2 Id
\end{bmatrix}
\left(
\begin{array}{c c c c}
0 &    &  &\\
   & 0 &  &\\
   &    & 0 & \\
   &	 & & -diag(Ew) 
\end{array}
\right )
\left (
\begin{array}{c}
u^{\ast}\\
w^{\ast}\\
q^{\ast}\\
p^{\ast}
\end{array}
\right ) + \lambda_{0}(Id-\Delta_{N})\alpha_{0}
\label{derJa0}\\
%&=
%\begin{bmatrix}
%   Id   & Id & 2 Id
%\end{bmatrix}
%\left(
%\begin{array}{c c c c}
%0 &    &  &\\
%   & 0 &  &\\
%   &    & & \\
%   &	 & & -diag(Ew) 
%\end{array}
%\right )
%\left (
%\begin{array}{c}
%u^{\ast}\\
%w^{\ast}\\
%q^{\ast}\\
%p^{\ast}
%\end{array}
%\right ) + \lambda_{0}(Id-\Delta_{N})\alpha_{0}
%\nonumber\\
&=- 
\begin{bmatrix}
   Id   & Id & 2 Id
\end{bmatrix}
diag(Ew) p^{\ast}
+\lambda_{0}(Id-\Delta_{N})\alpha_{0},
\nonumber
\end{align}

\begin{align}
\hat{J}_{\mathrm{pd},\alpha_{1}}'(\alpha_{0},\alpha_{1})
&=(D_{\alpha_{1}}g_{\mathrm{pd}})^{\top} x^{\ast}+ D_{\alpha_{1}}J_{\mathrm{pd}}(\alpha_{0},\alpha_{1})\label{derJa1}\\
&=
\begin{bmatrix}
   Id   & Id 
\end{bmatrix}
\left(
\begin{array}{c c c c}
0 &    &  &\\
   & 0 &  &\\
   &    & -diag(Du-w)& \\
   &	 & &  0
\end{array}
\right )
\left (
\begin{array}{c}
u^{\ast}\\
w^{\ast}\\
q^{\ast}\\
p^{\ast}
\end{array}
\right )
+\lambda_{1}(Id-\Delta_{N})\alpha_{1},
\nonumber
\\
&=- 
\begin{bmatrix}
   Id   & Id 
\end{bmatrix}
diag(Du-w) q^{\ast}
+\lambda_{1}(Id-\Delta_{N})\alpha_{1},\nonumber
\end{align}
where $\mathrm{x}=(u,w,q,p)$ solves  $g_{\mathrm{pd}}(\mathrm{x})=0$ for $\alpha_{0}, \alpha_{1}$. The corresponding reduced gradients are  
\begin{align}
\nabla_{\alpha_{i}} \hat{J}_{\mathrm{pd}}(\alpha_{0},\alpha_{1})=(I-\Delta_{N})^{-1}\hat{J}_{\mathrm{pd},\alpha_{i}}'(\alpha_{0},\alpha_{1}),\quad i=0,1. \label{nablaJa0}
%\nabla_{\alpha_{1}} \hat{J}_{\mathrm{pd}}(\alpha_{0},\alpha_{1})=(Id-\Delta_{N})^{-1}\hat{J}_{\mathrm{pd},\alpha_{1}}'(\alpha_{0},\alpha_{1}).\label{nablaJa1}
\end{align}
We note that in the case of a scalar $\alpha_{0}$, we set $\lambda_{0}=0$. Then, $\hat{J}_{\mathrm{pd},\alpha_{0}}'(\alpha_{0},\alpha_{1})=-[\mathbbm{1}\;\; \mathbbm{1}\;\; 2 \mathbbm{1}] diag (Ew)p^{\ast}$,  and  $\nabla_{\alpha_{0}} \hat{J}_{\mathrm{pd}}(\alpha_{0},\alpha_{1})=\hat{J}_{\mathrm{pd},\alpha_{0}}'(\alpha_{0},\alpha_{1})$.

In summary, the projected gradient  algorithm for the solutions of \eqref{bilevelTGV_h_pd} is described in Algorithm \ref{alg:bilevelTGV_pd}. The projections $P_{(\mathcal{A}_{ad}^{0})_{h}}$ and $P_{(\mathcal{A}_{ad}^{1})_{h}}$ are computed  as before, using  \cite[Algorithm 4]{hintermuellerPartII}.

 \begin{algorithm}[!h]
  \caption{\newline Discretized projected gradient method for  the bilevel TGV problem \eqref{bilevelTGV_h_pd} \label{alg:bilevelTGV_pd}}
  \begin{algorithmic}
    \Statex {\textbf{Input}: $f$, $\underline{\alpha}_{0}$, $\overline{\alpha}_{0}$, $\underline{\alpha}_{1}$, $\overline{\alpha}_{1}$, $\overline{\sigma}$, $\underline{\sigma}$, $\lambda_{0}$, $\lambda_{1}$, $\alpha$, $\mu$, $\gamma_{0}$, $\gamma_{1}$, $\delta$, $n_{w}$ $\tau_{0}^{0}$, $\tau_{1}^{0}$, $0<c<1$, $0<\theta_{-}<1\le \theta_{+}$}
    \Statex{\textbf{Initialize}: $\alpha_{0}^{0}\in(\mathcal{A}_{ad}^{0})_{h}$, $\alpha_{1}^{0}\in(\mathcal{A}_{ad}^{1})_{h}$  and set $k=0$.}
 \Repeat
     \State{Use the Algorithm  \ref{TGV_Newton_pd} to compute the solution $\mathrm{x}^{k}=(u^{k},w^{k},q^{k},p^{k})$ of the lower level problem 
     \[g_{\mathrm{pd}}(u^{k},w^{k},q^{k},p^{k})=0\]}
    \State{Solve the adjoint equation \eqref{adjoint_pd} for $(u^{\ast},w^{\ast},q^{\ast}, p^{\ast})$
    } 
  \State{Compute the derivative of the reduced objective with respect to $\alpha_{0}$ and $\alpha_{1}$ as in \eqref{derJa0} and \eqref{derJa1}
  } 
  \State{Compute the reduced gradients
  \begin{align*}
  \nabla_{\alpha_{i}} \hat{J}_{\mathrm{pd}}(\alpha_{0}^{k},\alpha_{1}^{k})
  &=(I-\Delta_{N})^{-1}\hat{J}_{\mathrm{pd},\alpha_{i}}'(\alpha_{0}^{k},\alpha_{1}^{k}),\quad i=0,1
%  \nabla_{\alpha_{1}} \hat{J}_{\mathrm{pd}}(\alpha_{0}^{k},\alpha_{1}^{k})
%  &=(Id-\Delta_{N})^{-1}\hat{J}_{\mathrm{pd},\alpha_{1}}'(\alpha_{0}^{k},\alpha_{1}^{k})
  \end{align*}
  }
\State{Compute the trial points 
\begin{align*}
\alpha_{i}^{k+1}&=P_{(\mathcal{A}_{ad}^{i})_{h}}\big(\alpha_{i}^{k}-\tau_{i}^{k}  \nabla_{\alpha_{i}} \hat{J}_{\mathrm{pd}}(\alpha_{0}^{k},\alpha_{1}^{k})\big),\quad i=0,1
%\alpha_{1}^{k+1}&=P_{(\mathcal{A}_{ad}^{1})_{h}}\big(\alpha_{1}^{k}-\tau_{1}^{k}  \nabla_{\alpha_{1}} \hat{J}_{\mathrm{pd}}(\alpha_{0}^{k},\alpha_{1}^{k})\big)
\end{align*}
}
\While{
\begin{align*}
\hat{J}_{\mathrm{pd}}(\alpha_{0}^{k+1},\alpha_{1}^{k+1})
&> \hat{J}_{\mathrm{pd}}(\alpha_{0}^{k},\alpha_{1}^{k})\\
&+c\left (\hat{J}_{\mathrm{pd},\alpha_{0}}'(\alpha_{0}^{k},\alpha_{1}^{k})^{\top}(\alpha_{0}^{k+1}-\alpha_{0}^{k})
+\hat{J}_{\mathrm{pd},\alpha_{1}}'(\alpha_{0}^{k},\alpha_{1}^{k})^{\top}(\alpha_{1}^{k+1}-\alpha_{1}^{k})\right)
\end{align*}
}{ (Armijo line search)}
\State{Set $\tau_{0}^{k}:=\theta_{-}\tau_{0}^{k}$, $\tau_{1}^{k}:=\theta_{-}\tau_{1}^{k}$ and re-compute 
\begin{align*}
\alpha_{i}^{k+1}&=P_{(\mathcal{A}_{ad}^{i})_{h}}\big(\alpha_{i}^{k}-\tau_{i}^{k}  \nabla_{\alpha_{i}} \hat{J}_{\mathrm{pd}}(\alpha_{0}^{k},\alpha_{1}^{k})\big),\quad i=0,1
%\alpha_{1}^{k+1}&=P_{(\mathcal{A}_{ad}^{1})_{h}}\big(\alpha_{1}^{k}-\tau_{1}^{k}  \nabla_{\alpha_{1}} \hat{J}_{\mathrm{pd}}(\alpha_{0}^{k},\alpha_{1}^{k})\big)
\end{align*}
}
  \EndWhile
  \State{Update $\tau_{0}^{k+1}=\theta_{+}\tau_{0}^{k}$, $\tau_{1}^{k+1}=\theta_{+}\tau_{1}^{k}$ and $k:=k+1$}
 \Until{some stopping condition is satisfied}
  \end{algorithmic}
\end{algorithm}

\subsection{Numerical examples in denoising}

We now discuss some weighted TGV numerical examples, with regularization weights produced automatically by Algorithms \ref{alg:bilevelTGV}  and \ref{alg:bilevelTGV_pd}. We are particularly interested in the degree of improvement over the scalar TGV examples.
% where the scalar regularization weights are chosen ``manually'' to optimize the standard image quality measures PSNR and SSIM.
We are also interested in whether the statistics-based upper level objective enforces an automatic choice of regularization parameters that ultimately leads to a reduction of the staircasing effect. Our TGV results are also compared with the bilevel weighted TV method of \cite{hintermuellerPartI, hintermuellerPartII}.
 The associated test images are depicted in Figure \ref{fig:testimages} with resolution $n=m=256$. The first one is the well-known ``Cameraman'' image which essentially consists of a combination of piecewise constant parts and texture. The next two images, ``Parrot'' and ``Turtle'' contain large piecewise affine type  areas, thus they are more suitable for the TGV prior. The final image ``hatchling'' is characterized by highly oscillatory patterns of various kinds, depicting sand in various degrees of focus. 
\begin{figure}[h!]
	\centering
	\begin{minipage}[t]{0.23\textwidth}\centering
	\includegraphics[width=0.95\textwidth]{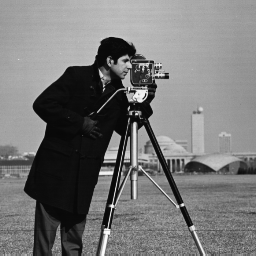}\\[1pt]
	Cameraman
	\end{minipage}
	\begin{minipage}[t]{0.23\textwidth}\centering
	\includegraphics[width=0.95\textwidth]{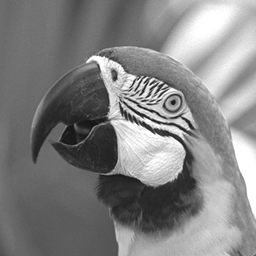}\\[1pt] Parrot
	\end{minipage}
	\begin{minipage}[t]{0.23\textwidth}\centering
	\includegraphics[width=0.95\textwidth]{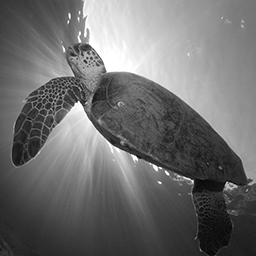}\\[1pt Turtle]
	
	\end{minipage}
	\begin{minipage}[t]{0.23\textwidth}\centering
	\includegraphics[width=0.95\textwidth]{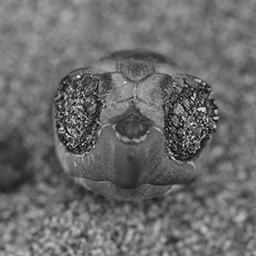}\\[2pt] Hatchling
	\end{minipage}
	\caption{Test images, resolution $256\times 256$.}
	\label{fig:testimages}
\end{figure}

\emph{Parameter values for \eqref{bilevelTGV_h}}: For the lower level dual TGV problem we used $\beta=10^{-3}$, $\gamma=0$, $\delta=10^{-6}$, $\epsilon_{0}=10^{-12}$, $\epsilon_{1}=10^{-12}$. Initially the lower problem is solved for $\epsilon_{0}^{0}=10^{3}$, $\epsilon_{1}^{0}=10^{3}$ and each of these successively decreased by the same factor $\theta_{\epsilon}=0.05$ down to final values $\epsilon_{0}=\epsilon_{1}=10^{-12}$.

MATLAB's backslash was used for the solution of the linear systems. We set $\underline{\alpha}_{0}=10^{-7}$, $\overline{\alpha}_{0}=10^{-2}$, and $\lambda=10^{-11}$, while for the $H^{1}$--projection we used  $\epsilon_{\alpha}=10^{-10}$, and $\underline{\alpha}_{1}=10^{-7}$, $\overline{\alpha}_{1}=10^{-2}$. A normalized  $n_{w}\times n_{w}$ filter for $w$ (i.e., with entries $1/n_{w}^{2}$), with $n_{w}=7$ was used. The local variance barriers $\underline{\sigma}^{2}$ and $\overline{\sigma}^{2}$ were set according to \eqref{sigma_choice}.  For our noisy images we have $\sigma^{2}=10^{-2}$, and thus the corresponding values for $(\underline{\sigma},\overline{\sigma})$ are $(0.00798, 0.01202).$
For the Armijo line search the parameters had the values $\tau_{0}^{0}=1$, $\tau_{1}^{0}=10^{-12}$, while $c=10^{-8}$, $\theta_{-}=0.25$, $\theta_{+}=2$.

\emph{Parameter values for \eqref{bilevelTGV_h_pd}}: For the lower level primal-dual TGV problem we used $\mu=0.1$, $\alpha=1$, $\delta=10^{-5}$, $\gamma_{0}=\gamma_{1}=10^{-3}$. 
We note that here we chose a mesh size $h=1$.
%{\bf\color{red}We note that here we chose a mesh size $h=1$, as this resulted in a better scaling of the Newton system -- contradicts our discussion!}.
 For the $H^{1}$--projection, we set $\epsilon_{\alpha}=10^{-6}$, and we also weighted the discrete Laplacian $\Delta_{N}$ with $6\times 10^{4}$. For the lower and upper bounds of $\alpha_{0}$ and $\alpha_{1}$ we set here $\underline{\alpha}_{0}=10^{-2}$, $\overline{\alpha}_{0}=10$ and $\underline{\alpha}_{1}=10^{-4}$, $\overline{\alpha}_{1}=10$. We also set $\lambda_{1}=10^{-11}$ and when we spatially varied $\alpha_{0}$ we also set $\lambda_{0}=10^{-11}$.   We used the same  filter $w$ and  local variance barriers as before. For the Armijo line search the parameters were $\tau_{0}^{0}=0.05$, $\tau_{1}^{0}=100$, $c=10^{-9}$, $\theta_{-}=0.25$, $\theta_{+}=2$. We solved each lower level problem until the residual of each of the optimality conditions \eqref{lbl:opt1}--\eqref{lbl:opt4} had Euclidean norm less than $10^{-4}$. Again,  MATLAB's backslash was used for the solution of the linear systems.  

We note that the initialization of the algorithms needs some attention. As it was done in \cite{hintermuellerPartII} for the TV case, $\alpha_{0}^{0}$ and $\alpha_{0}^{1}$ must be large enough in order to produce cartoon-like images, providing the local variance estimator with useful information. However, if $\alpha_{0}$ is initially too large then there is a danger of falling into the regime, in which the TGV functional and hence the solution map of (at least the non-regularized) lower level problem does not depend on $\alpha_{0}$.  In that case the derivative of the reduced functional with respect to $\alpha_{0}$ will be close to zero, thus making no or little progress with respect to its optimal choice.  Indeed this was confirmed after some numerical experimentation. Note that  an analogous phenomenon can occur also in the case where $\alpha_{0}$ is much smaller than $\alpha_{1}$. In that case it is the effect of $\alpha_{1}$ which vanishes. This has been shown theoretically in \cite[Proposition 2]{tgv_asymptotic} for dimension one, but numerical experiments indicate that this phenomenon persists also in higher dimensions. In our examples we used and $\alpha_{1}^{0}=9\times 10^{-4}$ and $\alpha_{0}^{0}= 3.125\times 10^{-6}$ for \eqref{bilevelTGV_h} and $\alpha_{1}^{0}=0.25$ and $\alpha_{0}^{0}= 0.2$ for \eqref{bilevelTGV_h_pd}. Regarding the termination of the projected gradient algorithm, we used  a fixed number of iterations, $n=30$ for \eqref{bilevelTGV_h} and $n=40$ for \eqref{bilevelTGV_h_pd}. Neither the upper level objective nor the argument changed significantly after running the algorithm for more iterations; see for instance Figure \ref{fig:Jvalues}.
The same holds true for the corresponding PSNR and SSIM values. We also note that a termination criterion as in \cite{hintermuellerPartII} based on the proximity measures $\|P_{(\mathcal{A}_{ad}^{i})_{h}}\left(\alpha_{i}^{k}-\nabla _{\alpha_{i}} \hat{J} (\alpha_{0}^{k},\alpha_{1}^{k})\right)-\alpha_{i}^{k}\|_{H^{1}(\om_{h})}$, $i=0,1$, is also possible here.

We note that due to the line search, the number of times that the lower level problem has to be solved is more than the number of projected gradient iterations. For instance for the four examples of \eqref{bilevelTGV_h_pd} of Figure  \ref{fig:weightedTGV_01} the lower level problem had to be solved 57, 57, 57, and 59 times respectively (40 projected gradient iterations). Typically  8-12 Newton iterations were needed per each lower level problem. 

 \begin{figure}
	 \begin{minipage}[t]{0.48\textwidth}\centering
		\resizebox{0.95\textwidth}{!}{
		% This file was created by matlab2tikz.
%
%The latest updates can be retrieved from
%  http://www.mathworks.com/matlabcentral/fileexchange/22022-matlab2tikz-matlab2tikz
%where you can also make suggestions and rate matlab2tikz.
%
\begin{tikzpicture}

\begin{axis}[%
width=6.028in,
height=4.754in,
at={(1.011in,0.642in)},
scale only axis,
xmin=0,
xmax=30,
ymode=log,
ymin=1e-07,
ymax=1e-05,
yminorticks=true,
ylabel style={font=\color{white!15!black}},
ylabel={\LARGE{$J_{\mathrm{d}}$ values}},
xlabel style={font=\color{white!15!black}},
xlabel={\LARGE{Projected gradient iterations}},
axis background/.style={fill=white},
legend style={legend cell align=left, align=left, draw=white!15!black}
]
\addplot [color=red, mark=o, mark options={solid, red}]
  table[row sep=crcr]{%
1	7.92025418523312e-06\\
2	1.51344947460644e-06\\
3	1.13172138991821e-06\\
4	1.0526806219983e-06\\
5	9.84499478093348e-07\\
6	9.6583625633586e-07\\
7	9.42163873274314e-07\\
8	9.26162347805311e-07\\
9	9.1148770532944e-07\\
10	8.99312396916726e-07\\
11	8.77435070870293e-07\\
12	8.40028382444426e-07\\
13	8.1386271769096e-07\\
14	7.85787403319441e-07\\
15	7.73212118345409e-07\\
16	7.68025160745699e-07\\
17	7.58592449251271e-07\\
18	7.50807172242724e-07\\
19	7.39288095553082e-07\\
20	7.35205323466642e-07\\
21	7.28318166231559e-07\\
22	7.24459320783163e-07\\
23	7.18444014204643e-07\\
24	7.14700256026075e-07\\
25	7.0988946417872e-07\\
26	7.05863451310332e-07\\
27	7.02250893772741e-07\\
28	6.97766013181739e-07\\
29	6.94416847751218e-07\\
30	6.90104631775679e-07\\
};
\addlegendentry{\Large{Cameraman}}

\addplot [color=green, mark=o, mark options={solid, green}]
  table[row sep=crcr]{%
1	6.65345935634544e-06\\
2	6.77908093043133e-07\\
3	5.72081660391083e-07\\
4	5.42202463578948e-07\\
5	4.96041167915544e-07\\
6	4.45896783618417e-07\\
7	4.34118425922962e-07\\
8	4.04498206576398e-07\\
9	3.90398998679469e-07\\
10	3.75267803370281e-07\\
11	3.64436149603147e-07\\
12	3.53669963915777e-07\\
13	3.46827267613679e-07\\
14	3.35299939007164e-07\\
15	3.19393393361473e-07\\
16	3.18185806937104e-07\\
17	3.11000300558638e-07\\
18	3.07985239065732e-07\\
19	3.05543215059237e-07\\
20	3.01332317788882e-07\\
21	2.96755215938297e-07\\
22	2.93575346178062e-07\\
23	2.91978722966538e-07\\
24	2.89432823027451e-07\\
25	2.87246948577717e-07\\
26	2.84470583595698e-07\\
27	2.83250291154434e-07\\
28	2.81452100039515e-07\\
29	2.78867774160883e-07\\
30	2.78565331612873e-07\\
};
\addlegendentry{\Large{Parrot}}

\addplot [color=blue, mark=o, mark options={solid, blue}]
  table[row sep=crcr]{%
1	6.9394483327094e-06\\
2	4.90193979800763e-06\\
3	1.03740490428506e-06\\
4	4.45245538798915e-07\\
5	4.3606100945152e-07\\
6	4.30482486127051e-07\\
7	4.22364854700092e-07\\
8	4.10234298169838e-07\\
9	4.07224660251881e-07\\
10	4.00222321826385e-07\\
11	3.95440261109315e-07\\
12	3.90338721172028e-07\\
13	3.85856509177952e-07\\
14	3.83427571822911e-07\\
15	3.79785858335904e-07\\
16	3.73361600233068e-07\\
17	3.7048146321775e-07\\
18	3.69531626991805e-07\\
19	3.69060601014483e-07\\
20	3.60705158041514e-07\\
21	3.58377122497501e-07\\
22	3.55059663407019e-07\\
23	3.53365307196353e-07\\
24	3.5077426134247e-07\\
25	3.48843885905963e-07\\
26	3.45146155466283e-07\\
27	3.38402713147035e-07\\
28	3.31387638361471e-07\\
29	3.29537946107342e-07\\
30	3.25247979590751e-07\\
};
\addlegendentry{\Large{Turtle}}

\addplot [color=black, mark=o, mark options={solid, black}]
  table[row sep=crcr]{%
1	7.17100556552085e-06\\
2	1.23854914228852e-06\\
3	1.08241336018669e-06\\
4	1.0592810865747e-06\\
5	9.54481539708672e-07\\
6	9.15418993573692e-07\\
7	8.71629875209475e-07\\
8	8.33397444260807e-07\\
9	7.84162061300719e-07\\
10	7.58225443877901e-07\\
11	7.27359839347395e-07\\
12	7.04974381338192e-07\\
13	6.65318737247608e-07\\
14	6.14752638439505e-07\\
15	5.95673100078165e-07\\
16	5.83982136611549e-07\\
17	5.63217444620514e-07\\
18	5.45603864233703e-07\\
19	5.22635851066064e-07\\
20	5.1444200332927e-07\\
21	5.02381601316868e-07\\
22	4.94608216107242e-07\\
23	4.8493086449051e-07\\
24	4.77482766853051e-07\\
25	4.68851056113842e-07\\
26	4.6239147909933e-07\\
27	4.54476941454008e-07\\
28	4.49251857991626e-07\\
29	4.41836691439026e-07\\
30	4.37966836607012e-07\\
};
\addlegendentry{\Large{Hatchling}}

\end{axis}
\end{tikzpicture}%
}
	\end{minipage}	
	\begin{minipage}[t]{0.48\textwidth}\centering	
		\resizebox{0.95\textwidth}{!}{
		% This file was created by matlab2tikz.
%
%The latest updates can be retrieved from
%  http://www.mathworks.com/matlabcentral/fileexchange/22022-matlab2tikz-matlab2tikz
%where you can also make suggestions and rate matlab2tikz.
%
\begin{tikzpicture}

\begin{axis}[%
width=6.028in,
height=4.754in,
at={(1.011in,0.642in)},
scale only axis,
xmin=0,
xmax=40,
ymode=log,
ymin=0.003,
ymax=4,
yminorticks=true,
ylabel style={font=\color{white!15!black}},
ylabel={\LARGE{$J_{\mathrm{pd}}$ values}},
xlabel style={font=\color{white!15!black}},
xlabel={\LARGE{Projected gradient iterations}},
axis background/.style={fill=white},
legend style={legend cell align=left, align=left, draw=white!15!black}
]
\addplot [color=red, mark=o, mark options={solid, red}]
  table[row sep=crcr]{%
1	1.84841216104632\\
2	0.696435829956872\\
3	0.0340406133234087\\
4	0.0285943845410081\\
5	0.0199530855729707\\
6	0.0198423813114976\\
7	0.0187687206607354\\
8	0.0184577289874712\\
9	0.0180760585029656\\
10	0.0179067446108614\\
11	0.0175366624917547\\
12	0.017308230691463\\
13	0.0168974632134879\\
14	0.0167112352719757\\
15	0.0165788811610914\\
16	0.0164113817471497\\
17	0.0161422884114486\\
18	0.0160426662669468\\
19	0.0159990877656258\\
20	0.015958926049297\\
21	0.0157426434467411\\
22	0.0156745797881065\\
23	0.0155544348280554\\
24	0.0154931220637163\\
25	0.0154820957654815\\
26	0.0154386140920273\\
27	0.0152758080763588\\
28	0.0152207294514238\\
29	0.015115620973279\\
30	0.0150228254669821\\
31	0.0148889529032339\\
32	0.0148430080046342\\
33	0.0147539588372936\\
34	0.0146046957682597\\
35	0.0145641019246968\\
36	0.0145247344386612\\
37	0.0144477821737675\\
38	0.0143005866053737\\
39	0.0140319961143433\\
40	0.0139457183792569\\
};
\addlegendentry{\Large{Cameraman}}

\addplot [color=green, mark=o, mark options={solid, green}]
  table[row sep=crcr]{%
1	1.24149826549551\\
2	0.515876075165367\\
3	0.120865562415\\
4	0.0186532526051514\\
5	0.0152011100800669\\
6	0.0143472570002263\\
7	0.0141466171927817\\
8	0.0138829501264856\\
9	0.0137594474289833\\
10	0.0135174379914446\\
11	0.0134438160562508\\
12	0.0133386288888268\\
13	0.0132735199325537\\
14	0.0132140494342015\\
15	0.0131393411259071\\
16	0.0130968408848996\\
17	0.0130552943835091\\
18	0.0129789808757056\\
19	0.0128723054280741\\
20	0.0128277231967711\\
21	0.0127988005852667\\
22	0.0127459655113215\\
23	0.0126956270304976\\
24	0.0126356807979988\\
25	0.0126120345868281\\
26	0.0125713028507724\\
27	0.0125703036704617\\
28	0.0124837754330539\\
29	0.0124638189110912\\
30	0.0124383489024112\\
31	0.0124152541168594\\
32	0.0123999973573687\\
33	0.012371824809377\\
34	0.0123547609084117\\
35	0.0123393527673036\\
36	0.0123099758071778\\
37	0.0122631644515754\\
38	0.0122462070722556\\
39	0.0122325696308575\\
40	0.0122071683885609\\
};
\addlegendentry{\Large{Parrot}}

\addplot [color=blue, mark=o, mark options={solid, blue}]
  table[row sep=crcr]{%
1	0.5658567810425\\
2	0.47474916504711\\
3	0.233631152945399\\
4	0.216762926257541\\
5	0.171904758994717\\
6	0.0427071478260397\\
7	0.0388375858894744\\
8	0.0182592414074006\\
9	0.0152366617152323\\
10	0.0149520973706874\\
11	0.0146510650208991\\
12	0.014446445487838\\
13	0.0142093072736291\\
14	0.0140728347728681\\
15	0.0138766238305132\\
16	0.0136745927095213\\
17	0.0135113386456451\\
18	0.01341958475044\\
19	0.0132868581612123\\
20	0.013071828294027\\
21	0.0129865674326082\\
22	0.0128705507872063\\
23	0.0127647617846798\\
24	0.0126052631747444\\
25	0.0125496161943959\\
26	0.0124109193898952\\
27	0.0123173302897316\\
28	0.012184637407874\\
29	0.0121438434613726\\
30	0.0119903380312925\\
31	0.0119019128182433\\
32	0.0117939522186329\\
33	0.0117760641611864\\
34	0.0115896161773863\\
35	0.0115112372108981\\
36	0.0114957372396083\\
37	0.0113463105044151\\
38	0.0112869609849009\\
39	0.0112202534305624\\
40	0.0111317410230551\\
};
\addlegendentry{\Large{Turtle}}

\addplot [color=black, mark=o, mark options={solid, black}]
  table[row sep=crcr]{%
1	0.859921023725285\\
2	0.516505120297241\\
3	0.516055891488347\\
4	0.514766777012052\\
5	0.511024265094735\\
6	0.496666007325591\\
7	0.263550551367947\\
8	0.0336260323321901\\
9	0.0237447302186058\\
10	0.0197148891266367\\
11	0.0176440163763841\\
12	0.0159234472038393\\
13	0.0152939048641668\\
14	0.0148888174687519\\
15	0.0143666607118204\\
16	0.0141959692120003\\
17	0.0140642578254204\\
18	0.0139213997672532\\
19	0.0137432675582621\\
20	0.0136491001384465\\
21	0.0134823800032121\\
22	0.0133337525320166\\
23	0.0131581952014422\\
24	0.0130969246630868\\
25	0.0129874436976769\\
26	0.0129624594190046\\
27	0.0127570704882738\\
28	0.012707067974945\\
29	0.0126283323630932\\
30	0.0125801330688227\\
31	0.0125155810660578\\
32	0.0124653437506712\\
33	0.0124158038581694\\
34	0.0123605052046173\\
35	0.0123289914943226\\
36	0.0122638803754335\\
37	0.0122422044202015\\
38	0.0121723964867417\\
39	0.0121572163433941\\
40	0.0120856979089995\\
};
\addlegendentry{\Large{Hatchling}}

\end{axis}
\end{tikzpicture}%
}
	\end{minipage}
\caption{Upper level objective values vs projected gradient iterations for the  problems \eqref{bilevelTGV_h} (left) and \eqref{bilevelTGV_h_pd} (right) of Figure \ref{fig:weightedTGV_01}. Note the different scaling which is due to the different values for the mesh size $h$ used for the two methods
%, i.e., $h=1/\sqrt{nm}$ for \eqref{bilevelTGV_h} and $h=1$ for  \eqref{bilevelTGV_h_pd}
}
\label{fig:Jvalues}		
\end{figure}
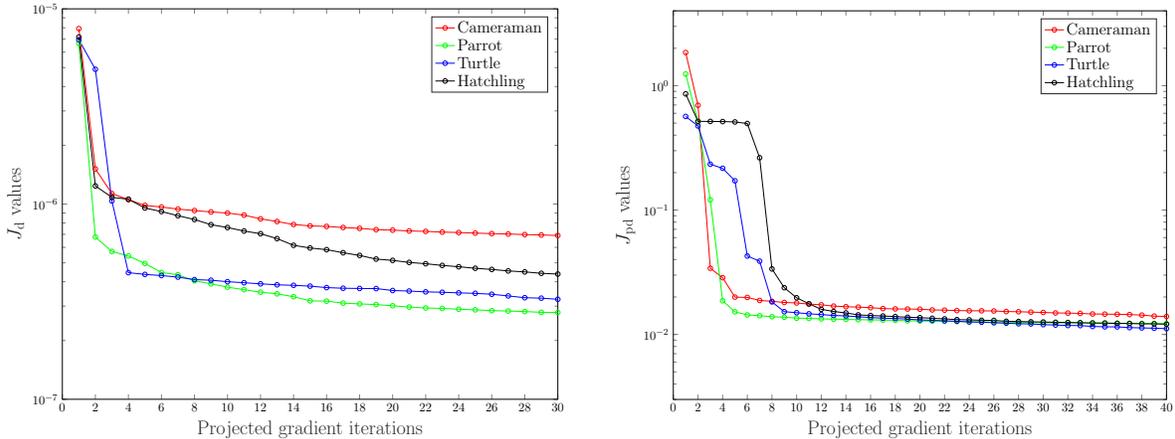

\begin{figure}[h!]
	\centering	
	%\cmmnt{ 
	\begin{minipage}[t]{0.23\textwidth}
	\includegraphics[width=0.95\textwidth]{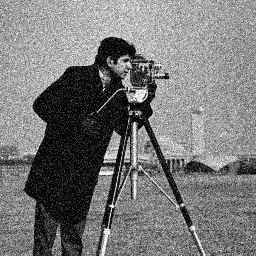}
	\end{minipage}
	%}
	\begin{minipage}[t]{0.23\textwidth}
	\includegraphics[width=0.95\textwidth]{parrot_noisy_01}
	\end{minipage}
	\begin{minipage}[t]{0.23\textwidth}
	\includegraphics[width=0.95\textwidth]{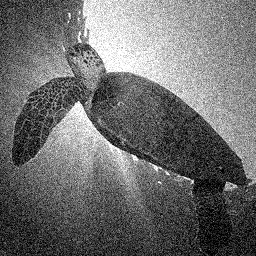}
	\end{minipage}
	\begin{minipage}[t]{0.23\textwidth}
	\includegraphics[width=0.95\textwidth]{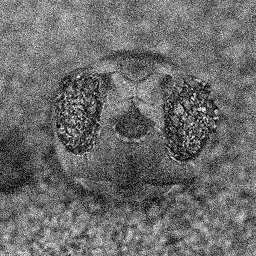}
	\end{minipage}
	
	 \vspace{-6pt}
	\begin{minipage}[t]{0.23\textwidth}
	\centering \scalebox{.65}{PSNR=20.00, SSIM=0.3304}
	\end{minipage}%}
	\begin{minipage}[t]{0.23\textwidth}
	\centering \scalebox{.65}{PSNR=20.04, SSIM=0.2773}
	\end{minipage}
	\begin{minipage}[t]{0.23\textwidth}
	 \centering \scalebox{.65}{PSNR=19.99, SSIM=0.2448}
	\end{minipage}
	\begin{minipage}[t]{0.23\textwidth}
	\centering \scalebox{.65}{PSNR=20.00, SSIM=0.3349}
	\end{minipage}\vspace{2pt}
	
	\begin{minipage}[t]{0.23\textwidth}
	\includegraphics[width=0.95\textwidth]{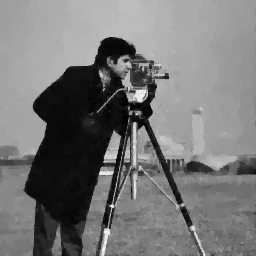}
	\end{minipage}
	\begin{minipage}[t]{0.23\textwidth}
	\includegraphics[width=0.95\textwidth]{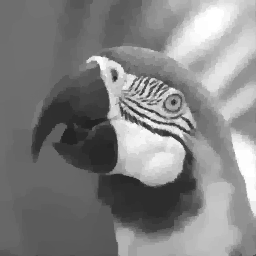}
	\end{minipage}
	\begin{minipage}[t]{0.23\textwidth}
	\includegraphics[width=0.95\textwidth]{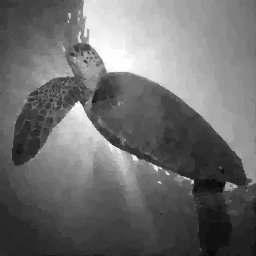}
	\end{minipage}
	\begin{minipage}[t]{0.23\textwidth}
	\includegraphics[width=0.95\textwidth]{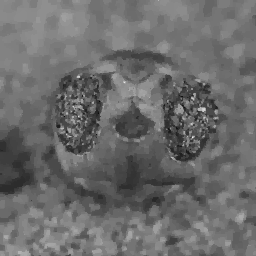}
	\end{minipage}
	
	\vspace{-6pt}
	\begin{minipage}[t]{0.23\textwidth}
	\centering \scalebox{.65}{PSNR=\textbf{27.85}, SSIM=\textbf{0.8259}}
	\end{minipage}
	\begin{minipage}[t]{0.23\textwidth}
	\centering \scalebox{.65}{PSNR=28.96, SSIM=0.8477}
	\end{minipage}
	\begin{minipage}[t]{0.23\textwidth}
	 \centering \scalebox{.65}{PSNR=29.60, SSIM=0.8176}
	\end{minipage}
	\begin{minipage}[t]{0.23\textwidth}
	\centering \scalebox{.65}{PSNR=27.55, SSIM=0.7750}
	\end{minipage}\vspace{2pt}

	%\cmmnt{ 
	\begin{minipage}[t]{0.23\textwidth}
	\includegraphics[width=0.95\textwidth]{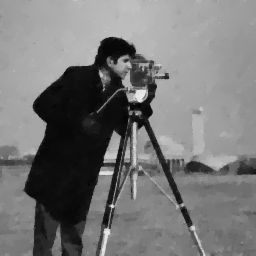}
	\end{minipage}%}
	\begin{minipage}[t]{0.23\textwidth}
	\includegraphics[width=0.95\textwidth]{parrot_scalarTGVpd}
	\end{minipage}
	\begin{minipage}[t]{0.23\textwidth}
	\includegraphics[width=0.95\textwidth]{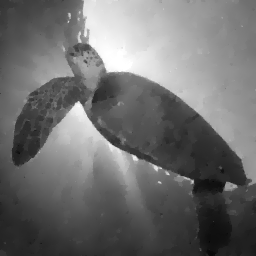}
	\end{minipage}
	\begin{minipage}[t]{0.23\textwidth}
	\includegraphics[width=0.95\textwidth]{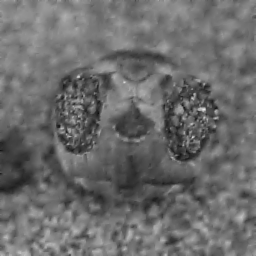}
	\end{minipage}
	
	\vspace{-6pt}
	\begin{minipage}[t]{0.23\textwidth}
	\centering \scalebox{.65}{PSNR=26.87, SSIM=0.8070}
	\end{minipage}%}
	\begin{minipage}[t]{0.23\textwidth}
	\centering \scalebox{.65}{PSNR=28.61, SSIM=0.8588}
	\end{minipage}
	\begin{minipage}[t]{0.23\textwidth}
	 \centering \scalebox{.65}{PSNR=29.24, SSIM=0.8273}
	\end{minipage}
	\begin{minipage}[t]{0.23\textwidth}
	\centering \scalebox{.65}{PSNR=27.82, SSIM=\textbf{0.8108}}
	\end{minipage}\vspace{2pt}
	
	%\cmmnt{ 
	\begin{minipage}[t]{0.23\textwidth}
	\includegraphics[width=0.95\textwidth]{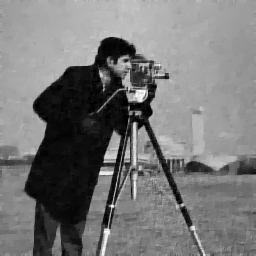}
	\end{minipage}%}
	\begin{minipage}[t]{0.23\textwidth}
	\includegraphics[width=0.95\textwidth]{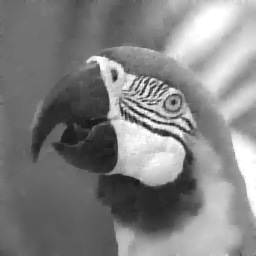}
	\end{minipage}
	\begin{minipage}[t]{0.23\textwidth}
	\includegraphics[width=0.95\textwidth]{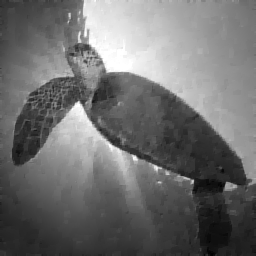}
	\end{minipage}
	\begin{minipage}[t]{0.23\textwidth}
	\includegraphics[width=0.95\textwidth]{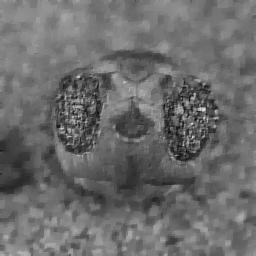}
	\end{minipage}
	
	\vspace{-6pt}
	\begin{minipage}[t]{0.23\textwidth}
	\centering \scalebox{.65}{PSNR=27.50, SSIM=0.8061}
	\end{minipage}%}
	\begin{minipage}[t]{0.23\textwidth}
	\centering \scalebox{.65}{PSNR=29.36, SSIM=\textbf{0.8653}}
	\end{minipage}
	\begin{minipage}[t]{0.23\textwidth}
	 \centering \scalebox{.65}{PSNR=29.10, SSIM=0.8231}
	\end{minipage}
	\begin{minipage}[t]{0.23\textwidth}
	\centering \scalebox{.65}{PSNR=27.67, SSIM=0.7884}
	\end{minipage}\vspace{2pt}
	
	\begin{minipage}[t]{0.23\textwidth}
	\includegraphics[width=0.95\textwidth]{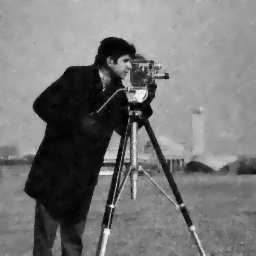}
	\end{minipage}
	\begin{minipage}[t]{0.23\textwidth}
	\includegraphics[width=0.95\textwidth]{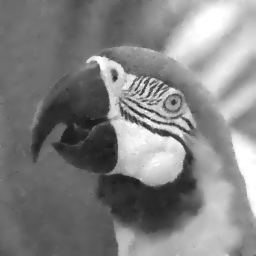}
	\end{minipage}
	\begin{minipage}[t]{0.23\textwidth}
	\includegraphics[width=0.95\textwidth]{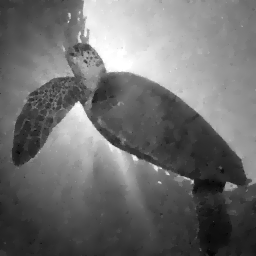}
	\end{minipage}
	\begin{minipage}[t]{0.23\textwidth}
	\includegraphics[width=0.95\textwidth]{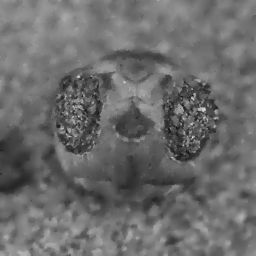}
	\end{minipage}	
	
        \vspace{-6pt}
	\begin{minipage}[t]{0.23\textwidth}
	\centering \scalebox{.65}{PSNR=27.42, SSIM=0.8077}
	\end{minipage}%}
	\begin{minipage}[t]{0.23\textwidth}
	\centering \scalebox{.65}{PSNR=\textbf{29.47}, SSIM=0.8628}
	\end{minipage}
	\begin{minipage}[t]{0.23\textwidth}
	 \centering \scalebox{.65}{PSNR=\textbf{29.63}, SSIM=\textbf{0.8305}}
	\end{minipage}
	\begin{minipage}[t]{0.23\textwidth}
	\centering \scalebox{.65}{PSNR=\textbf{28.01}, SSIM=0.8037}
	\end{minipage}
	
%	%\cmmnt{ 
%	\begin{minipage}[t]{0.23\textwidth}
%	\includegraphics[width=0.95\textwidth]{figures/weightedTGV_01/alpha_dorsal/cam_alpha.png}
%	\end{minipage}%}
%	\begin{minipage}[t]{0.23\textwidth}
%	\includegraphics[width=0.95\textwidth]{figures/weightedTGV_01/alpha_dorsal/parrot_alpha.png}
%	\end{minipage}
%	\begin{minipage}[t]{0.23\textwidth}
%	\includegraphics[width=0.95\textwidth]{figures/weightedTGV_01/alpha_dorsal/orestis_alpha.png}
%	\end{minipage}
%	\begin{minipage}[t]{0.23\textwidth}
%	\includegraphics[width=0.95\textwidth]{figures/weightedTGV_01/alpha_dorsal/hatchling_alpha.png}
%	\end{minipage}

\caption{First row: noisy images. Second row: bilevel TV.  Third row: Best scalar TGV  (SSIM). Fourth row: bilevel TGV--dual. Fifth row: bilevel TGV--primal-dual}
% Fifth row: bilevel TGV-primal dual}
\label{fig:weightedTGV_01}
\end{figure}	

\begin{figure}[h!]
	\centering	
	%\cmmnt{ 
	\begin{minipage}[t]{0.23\textwidth}
	\includegraphics[width=0.95\textwidth, trim={3cm 4.5cm 2cm 0.5cm},clip]{cam_noisy_01}
	\end{minipage}
	%}
	\begin{minipage}[t]{0.23\textwidth}
	\includegraphics[width=0.95\textwidth, trim={3.5cm 4.0cm 1.5cm 1.0cm},clip]{parrot_noisy_01}
	\end{minipage}
	\begin{minipage}[t]{0.23\textwidth}
	\includegraphics[width=0.95\textwidth, trim={1.5cm 4.0cm 3.5cm 1.0cm},clip]{orestis_noisy_01}
	\end{minipage}
	\begin{minipage}[t]{0.23\textwidth}
	\includegraphics[width=0.95\textwidth, trim={5.0cm 3.0cm 0.0cm 2.0cm},clip]{hatchling_noisy_01}
	\end{minipage}
	
	 \vspace{-5pt}
	\begin{minipage}[t]{0.95\textwidth}
	\centering \footnotesize{Noisy}
	\end{minipage}\vspace{4pt}
	
	\begin{minipage}[t]{0.23\textwidth}
	\includegraphics[width=0.95\textwidth, trim={3cm 4.5cm 2cm 0.5cm},clip]{cam_wTV}
	\end{minipage}
	\begin{minipage}[t]{0.23\textwidth}
	\includegraphics[width=0.95\textwidth, trim={3.5cm 4.0cm 1.5cm 1.0cm},clip]{parrot_wTV}
	\end{minipage}
	\begin{minipage}[t]{0.23\textwidth}
	\includegraphics[width=0.95\textwidth, trim={1.5cm 4.0cm 3.5cm 1.0cm},clip]{orestis_wTV}
	\end{minipage}
	\begin{minipage}[t]{0.23\textwidth}
	\includegraphics[width=0.95\textwidth, trim={5.0cm 3.0cm 0.0cm 2.0cm},clip]{hatchling_wTV}
	\end{minipage}
	
	 \vspace{-5pt}
	\begin{minipage}[t]{0.95\textwidth}
	\centering \footnotesize{Bilevel TV}
	\end{minipage}\vspace{4pt}
	
	\begin{minipage}[t]{0.23\textwidth}
	\includegraphics[width=0.95\textwidth, trim={3cm 4.5cm 2cm 0.5cm},clip]{cam_scalarTGVpd}
	\end{minipage}
	\begin{minipage}[t]{0.23\textwidth}
	\includegraphics[width=0.95\textwidth, trim={3.5cm 4.0cm 1.5cm 1.0cm},clip]{parrot_scalarTGVpd}
	\end{minipage}
	\begin{minipage}[t]{0.23\textwidth}
	\includegraphics[width=0.95\textwidth, trim={1.5cm 4.0cm 3.5cm 1.0cm},clip]{orestis_scalarTGVpd}
	\end{minipage}
	\begin{minipage}[t]{0.23\textwidth}
	\includegraphics[width=0.95\textwidth,  trim={5.0cm 3.0cm 0.0cm 2.0cm},clip]{hatchling_scalarTGV}
	\end{minipage}
	
	\vspace{-5pt}
	\begin{minipage}[t]{0.95\textwidth}
	\centering \footnotesize{Best scalar TGV reconstructions (SSIM)}
	\end{minipage}\vspace{4pt}
	
	\begin{minipage}[t]{0.23\textwidth}
	\includegraphics[width=0.95\textwidth, trim={3cm 4.5cm 2cm 0.5cm},clip]{cam_wTGV}
	\end{minipage}%}
	\begin{minipage}[t]{0.23\textwidth}
	\includegraphics[width=0.95\textwidth, trim={3.5cm 4.0cm 1.5cm 1.0cm},clip]{parrot_wTGV}
	\end{minipage}
	\begin{minipage}[t]{0.23\textwidth}
	\includegraphics[width=0.95\textwidth, trim={1.5cm 4.0cm 3.5cm 1.0cm},clip]{orestis_wTGV}
	\end{minipage}
	\begin{minipage}[t]{0.23\textwidth}
	\includegraphics[width=0.95\textwidth,  trim={5.0cm 3.0cm 0.0cm 2.0cm},clip]{hatchling_wTGV}
	\end{minipage}
	
	\vspace{-5pt}
	\begin{minipage}[t]{0.95\textwidth}
	\centering \footnotesize{Bilevel TGV-dual}
	\end{minipage}\vspace{4pt}
	
	\begin{minipage}[t]{0.23\textwidth}
	\includegraphics[width=0.95\textwidth, trim={3cm 4.5cm 2cm 0.5cm},clip]{cam_wTGVpd}
	\end{minipage}
	\begin{minipage}[t]{0.23\textwidth}
	\includegraphics[width=0.95\textwidth, trim={3.5cm 4.0cm 1.5cm 1.0cm},clip]{parrot_wTGVpd}
	\end{minipage}
	\begin{minipage}[t]{0.23\textwidth}
	\includegraphics[width=0.95\textwidth, trim={1.5cm 4.0cm 3.5cm 1.0cm},clip]{orestis_wTGVpd}
	\end{minipage}
	\begin{minipage}[t]{0.23\textwidth}
	\includegraphics[width=0.95\textwidth,  trim={5.0cm 3.0cm 0.0cm 2.0cm},clip]{hatchling_wTGVpd}
	\end{minipage}	
	
	\vspace{-5pt}
	\begin{minipage}[t]{0.95\textwidth}
	\centering \footnotesize{Bilevel TGV--primal-dual}
	\end{minipage}
\caption{Details of the reconstructions shown in Figure \ref{fig:weightedTGV_01}}
\label{fig:weightedTGV_01_Details}
\end{figure}

{\small
\setlength\extrarowheight{2pt}
\begin{table}[h!]
  \centering
  \begin{tabular}{|c|c|c|c|c|}
  %\hline
%$\sigma^{2}=0.01$ &  $\alpha\tv$ best PSNR &   $\alpha\tv$ best SSIM & wTV & $\sigma^{2}=4\times 10^{-4}$ &  $\alpha\tv$ best PSNR &   $\alpha\tv$ best SSIM &  wTV\\\hline
  \hline
  $\sigma^{2}=0.01$                 			&Cameraman             			& Parrot                   	               &  Turtle 			   	    & Hatchling \\\hline
 scalar TV \tiny{(PSNR)}     			& $27.54$, $0.7857$  		        & $28.88$, $0.8119$ 	       & $29.27$, $0.7924$, 	    & $27.57$, $0.7597$\\\hline
 scalar TV \tiny{(SSIM)}	      			& $27.19$, $0.8064$  		        & $28.51$, $0.8421$	               & $29.11$, $0.8044$  	             & $27.46$, $0.7687$\\\hline
bilevel $\tv$  			     			& $\bf{27.85}$, $\bf{0.8259}$            & $28.96$, $0.8477$                & $29.60$, $0.8176$              & $27.55$, $0.7750$\\\hline
 scalar  TGV-dual  \tiny{(PSNR)}    		& $27.38$, $0.7730$    			& $29.07$, $0.8438$	               & $28.97$, $0.8032$	            & $28.00$, $0.8032$\\\hline
 scalar  TGV--dual  \tiny{(SSIM)}     		& $26.95$, $0.8043$				& $28.61$, $0.8575$		       & $28.70$, $0.8200$              & $27.82$, $\bf{0.8108}$ \\\hline
bilevel TGV--dual  	  		        		& $27.50$, $0.8061$		      	        & $29.36$, $\bf{0.8653}$         & $29.10$, $0.8231$               & $27.67$, $0.7884$ \\\hline
scalar  TGV--primal-dual  \tiny{(PSNR)}    & $27.23$, $0.7873$				& $29.10$, $0.8325$		      &	 $29.40$, $0.8230$		    & $27.88$, $0.7991$     \\\hline
scalar  TGV--primal-dual  \tiny{(SSIM)}	& $26.87$, $0.8070$				& $28.61$, $0.8588$		      & $29.24$, $0.8273$		    & $27.71$, $0.8024$     \\\hline
bilevel TGV--primal-dual 				& $27.42$, $0.8077$				& $\bf{29.47}$, $0.8628$	      &	 $\bf{29.63}$, \bf{0.8305}	    &  $\bf{28.01}$, $0.8037$     \\\hline
  \end{tabular}\vspace{8pt}
  \caption{PSNR and SSIM comparisons for the images of Figure \ref{fig:weightedTGV_01}. Every cell contains the corresponding PSNR and SSIM value}
  \label{tab:psnr_ssim}
\end{table}
}

For the first series of examples we  keep the parameter $\alpha_{0}$  scalar, whose value nevertheless is determined by the bilevel algorithms. We depict the examples in Figure \ref{fig:weightedTGV_01}. The first row shows the  noisy images, while the second contains the bilevel TV results  \cite{hintermuellerPartI}. The third row depicts the best scalar TGV results with respect to SSIM, either using the dual or the primal-dual approach -- whichever had the largest value -- where we have computed the optimal scalars $\alpha_{0}, \alpha_{1}$ with a manual grid method. The fourth and the fifth rows show the results of \eqref{bilevelTGV_h} and \eqref{bilevelTGV_h_pd} respectively. Detailed sections  of all the images of Figure \ref{fig:weightedTGV_01} are highlighted in Figure \ref{fig:weightedTGV_01_Details}. The  weight functions $\alpha_{1}$ for the bilevel TV and the  bilevel TGV algorithms are shown in Figure \ref{weights_contour}. In Table \ref{tab:psnr_ssim} we report all PSNR and SSIM values of the best scalar methods (scalar TV, scalar TGV--dual, scalar TGV--primal-dual) with respect to both   quality measures, as well as the corresponding values of the three bilevel algorithms. We next comment on the results for each image.\\[3pt]
\noindent
\emph{Cameraman}: Here both the best PSNR and SSIM are obtained by the bilevel TV algorithm. This is probably not surprising due to the piecewise constant nature of this image. However, both bilevel TGV algorithms improve upon their scalar versions with respect to both measures. It is interesting to observe the two different spatial weights $\alpha_{1}$ produced by the two bilevel TGV algorithms, see the last two functions at the first column of Figure \ref{weights_contour}. The dual TGV algorithm, solving the anisotropic version of TGV, has the tendency to blur thin objects that have a 45 degree orientation with respect to the pixel grid, like for instance the middle part of the cameraman's tripod. We see that the weight $\alpha_{1}$ drops significantly at this area aiming to reduce this effect. Otherwise both bilevel algorithms preserve better the detailed area of the camera with the weights having small values there.\\[3pt]
\noindent
\emph{Parrot}: Here the best results with respect to both PSNR and SSIM are achieved by the two bilevel TGV algorithms, \eqref{bilevelTGV_h_pd} and \eqref{bilevelTGV_h}, respectively.  There is significant improvement over all TV methods, which is due to the parameters being chosen in a way such that the staircasing effect diminishes. Furthermore, we  observe improvement over the scalar TGV results especially around the parrot's eye, where the weights $\alpha_{1}$ drop significantly; see the second column of Figure \ref{weights_contour}. \\[3pt]
\noindent
\emph{Turtle}: We get analogous results here as well, with the bilevel TGV \eqref{bilevelTGV_h_pd} producing the best results both with respect to PSNR and SSIM. There a significant reduction of the staircasing effect, while the weight $\alpha_{1}$ drops in the detailed areas of the image (head and flipper of the turtle).\\[3pt]
\noindent
\emph{Hatchling}: In this image, the best PSNR is achieved by \eqref{bilevelTGV_h_pd}, but only marginally. In fact, the best SSIM is achieved by the scalar version of the dual TGV algorithm also with a comparable PSNR. Similarly at least with respect to PSNR, the scalar TV is marginally better than bilevel TV. We attribute this to the fact that the natural oscillatory features of the image are interpreted as noise by the upper level objective. Nevertheless, all the bilevel methods are able to locate and preserve better the eyes area, i.e., sand in focus, with the weight $\alpha_{1}$ dropping there significantly.

%\cmmnt{
\begin{figure}[h!]
	\centering	
	\begin{minipage}[t]{0.23\textwidth}
	  	\resizebox{0.95\textwidth}{!}{
%	\begin{tikzpicture}
%		\begin{axis}[xmin=0,xmax=255,ymin=0,ymax=255,width=9cm,
%		zticklabel style={
%          	  /pgf/number format/fixed,
%            	/pgf/number format/precision=1,
%          	  /pgf/number format/fixed zerofill
%      			  },
%		height=9.25cm,grid=both, zmin=0.00015]
%		\addplot3 graphics[points={%
%(17.9949,203.9081,0.00041755) => (120.45,307.1475)
%(212.6781,181.815,0.00035962) => (305.14,292.0913)
%(105.5411,145.7219,0.00021256) => (245.9187,121.4537)
%(234.1194,27.8209,0.00023354) => (458.7138,144.54)
%		}]
%		{cam_tv_trans2};
%		\end{axis}	
%	\end{tikzpicture}
	\begin{tikzpicture}
	\node[anchor=south west,inner sep=0] at (0,0) {\includegraphics{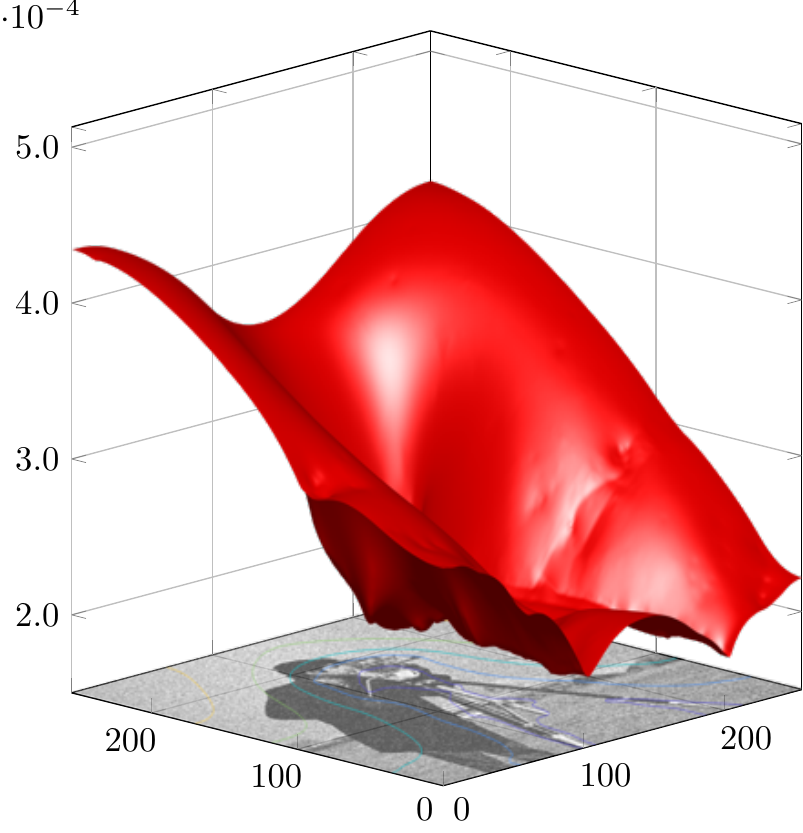}};
	\end{tikzpicture}
}
	\end{minipage}
		\begin{minipage}[t]{0.23\textwidth}
	  	\resizebox{0.95\textwidth}{!}{
%	\begin{tikzpicture}
%		\begin{axis}[xmin=0,xmax=255,ymin=0,ymax=255,width=9cm,
%		zticklabel style={
%          	  /pgf/number format/fixed,
%            	/pgf/number format/precision=1,
%          	  /pgf/number format/fixed zerofill
%      			  },
%		height=9.25cm,grid=both, zmin=0.0001000]
%		\addplot3 graphics[points={%
%(6.452,208.7468,0.00050973) => (106.3975,291.0875)
%(12.1975,9.9079,0.00061589) => (286.0687,310.1587)
%(84.5284,86.6081,0.00048899) => (280.0462,270.0087)
%(168.4054,115.083,0.00023713) => (326.2188,149.5588)
%		}]
%		{parrot_tv_trans2};
%		\end{axis}	
%	\end{tikzpicture}
	\begin{tikzpicture}
	\node[anchor=south west,inner sep=0] at (0,0) {\includegraphics{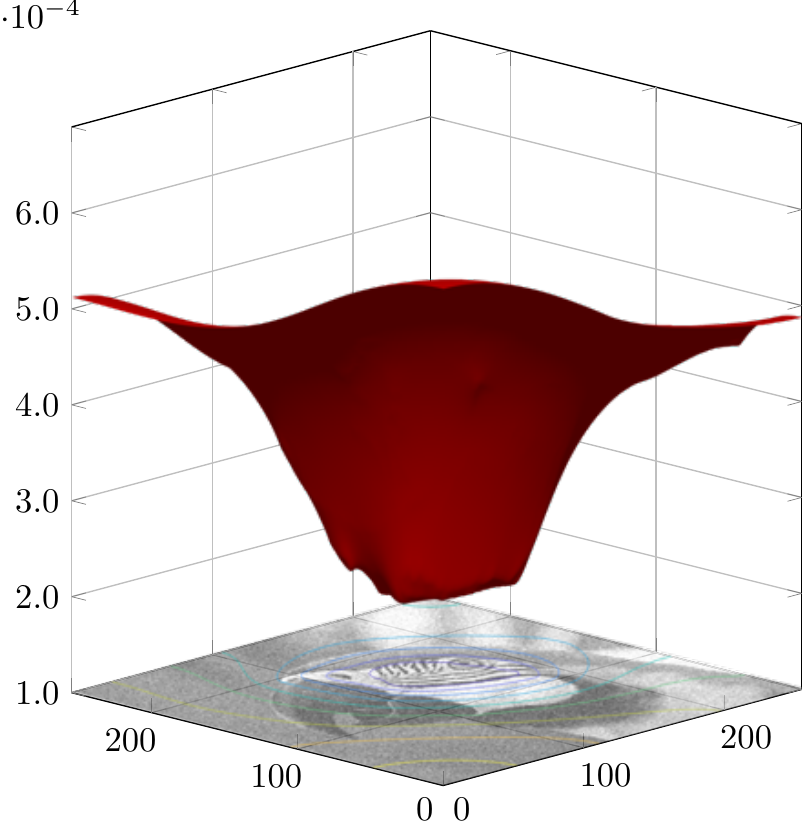}};
	\end{tikzpicture}
	}
	\end{minipage}
		\begin{minipage}[t]{0.23\textwidth}
	  	\resizebox{0.95\textwidth}{!}{
%	\begin{tikzpicture}
%		\begin{axis}[xmin=0,xmax=255,ymin=0,ymax=255,width=9cm,
%		zticklabel style={
%          	  /pgf/number format/fixed,
%            	/pgf/number format/precision=1,
%          	  /pgf/number format/fixed zerofill
%      			  },
%		height=9.25cm,grid=both, zmin=0.00015]
%		\addplot3 graphics[points={%
%(29.5323,224.6306,0.00028101) => (111.4162,155.5812)
%(220.2364,219.933,0.00046472) => (278.0387,330.2337)
%(76.5611,90.2891,0.00030866) => (270.0087,156.585)
%(230.0169,63.2737,0.00020543) => (424.5863,111.4162)
%		}]
%		{orestis_tv_trans2};
%		\end{axis}	
%	\end{tikzpicture}
	\begin{tikzpicture}
	\node[anchor=south west,inner sep=0] at (0,0) {\includegraphics{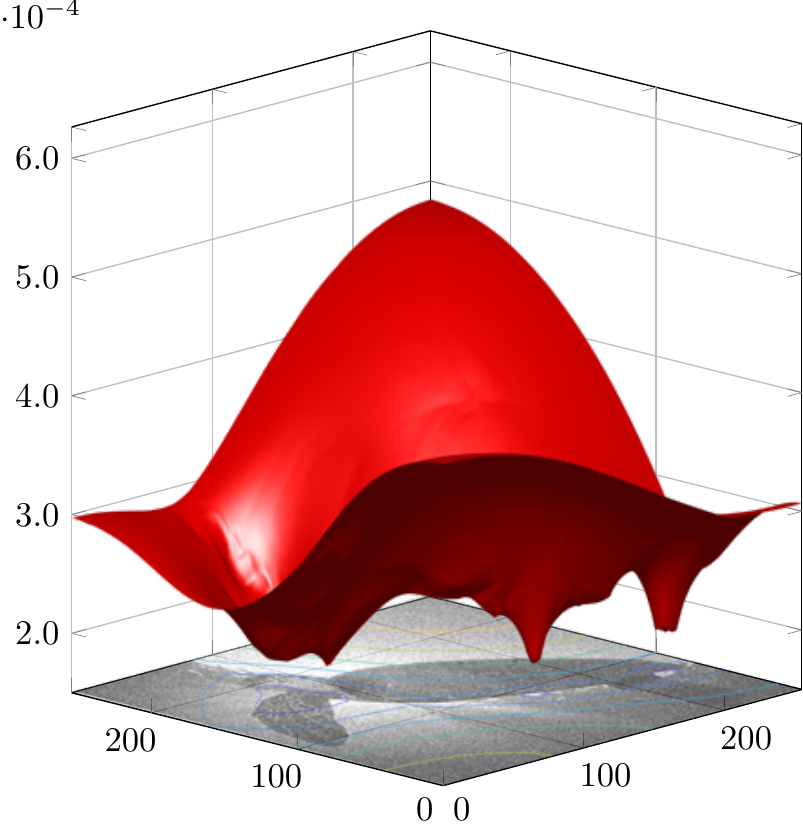}};
	\end{tikzpicture}
	}
	\end{minipage}
		\begin{minipage}[t]{0.23\textwidth}
	  	\resizebox{0.95\textwidth}{!}{
%	\begin{tikzpicture}
%		\begin{axis}[xmin=0,xmax=255,ymin=0,ymax=255,width=9cm,
%		zticklabel style={
%          	  /pgf/number format/fixed,
%            	/pgf/number format/precision=1,
%          	  /pgf/number format/fixed zerofill
%      			  },
%		height=9.25cm,grid=both, zmin=0.00015]
%		\addplot3 graphics[points={%
%(47.1942,232.1054,0.00035702) => (120.45,276.0312)
%(250.7675,226.5904,0.00035409) => (298.1137,318.1888)
%(52.4415,126.4154,0.00023492) => (217.8137,131.4912)
%(210.0773,27.4404,0.00030488) => (438.6387,215.8062)
%		}]
%		{hatchling_tv_trans2};
%		\end{axis}	
%	\end{tikzpicture}
	\begin{tikzpicture}
	\node[anchor=south west,inner sep=0] at (0,0) {\includegraphics{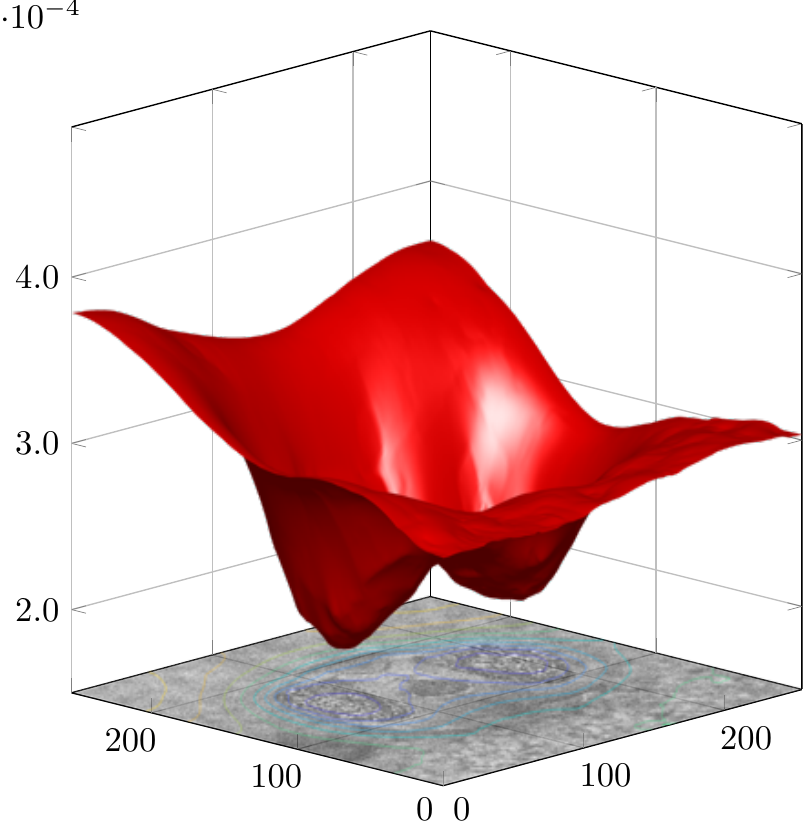}};
	\end{tikzpicture}
	}
	\end{minipage}

	\begin{minipage}[t]{0.23\textwidth}
	  	\resizebox{0.95\textwidth}{!}{
%	\begin{tikzpicture}
%		\begin{axis}[xmin=0,xmax=255,ymin=0,ymax=255,width=9cm,
%		zticklabel style={
%          	  /pgf/number format/fixed,
%            	/pgf/number format/precision=1,
%          	  /pgf/number format/fixed zerofill
%      			  },
%		height=9.25cm,grid=both, zmin=0.0000]
%		\addplot3 graphics[points={%
%		(44.7203,220.5818,0.00024119) => (128.48,292.0913)
%		(222.5813,208.5059,0.00025007) => (291.0875,338.2637)
%		(198.9708,94.3592,0.00018854) => (370.3837,249.9338)
%		(153.4041,68.6252,5.6184e-05) => (354.3237,109.4087)
%		}]
%		{cam_tgv_trans2};
%		\end{axis}	
%	\end{tikzpicture}
	\begin{tikzpicture}
	\node[anchor=south west,inner sep=0] at (0,0) {\includegraphics{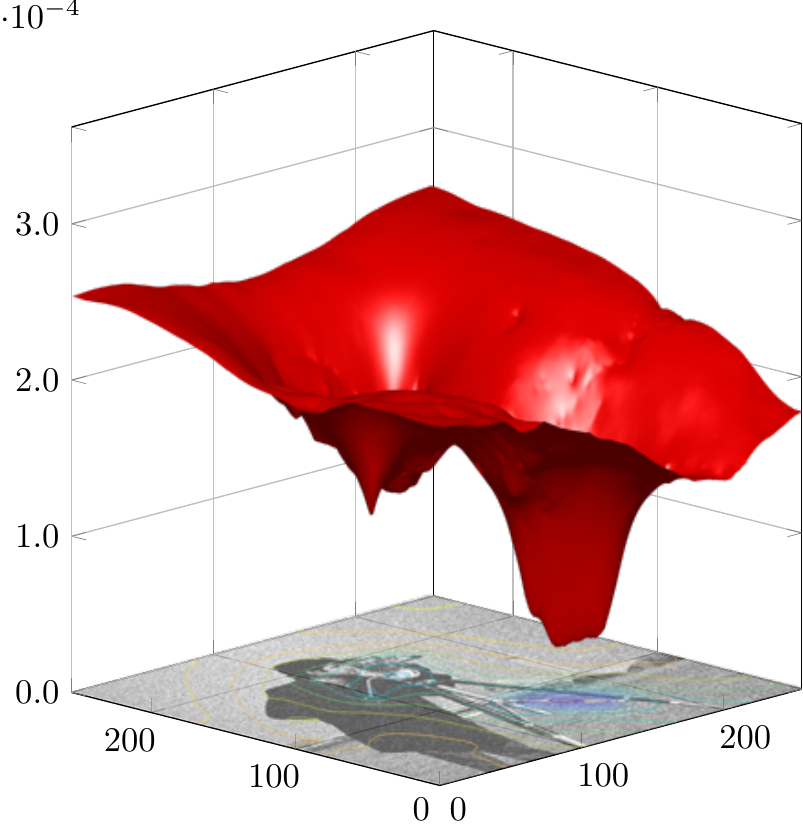}};
	\end{tikzpicture}	
	}
	\end{minipage}
		\begin{minipage}[t]{0.23\textwidth}
	  	\resizebox{0.95\textwidth}{!}{
%	\begin{tikzpicture}
%		\begin{axis}[xmin=0,xmax=255,ymin=0,ymax=255,width=9cm,
%		zticklabel style={
%          	  /pgf/number format/fixed,
%            	/pgf/number format/precision=1,
%          	  /pgf/number format/fixed zerofill
%      			  },
%		height=9.25cm,grid=both, zmin=0.00005000]
%		\addplot3 graphics[points={%
%(129.838,142.9232,0.00012616) => (269.005,129.4837)
%(66.6772,218.9488,0.00028863) => (148.555,263.9862)
%(221.5188,210.9053,0.00031536) => (288.0762,319.1925)
%(176.6583,18.0046,0.00031954) => (418.5638,270.0087)
%		}]
%		{parrot_tgv_trans2};
%		\end{axis}	
%	\end{tikzpicture}
	\begin{tikzpicture}
	\node[anchor=south west,inner sep=0] at (0,0) {\includegraphics{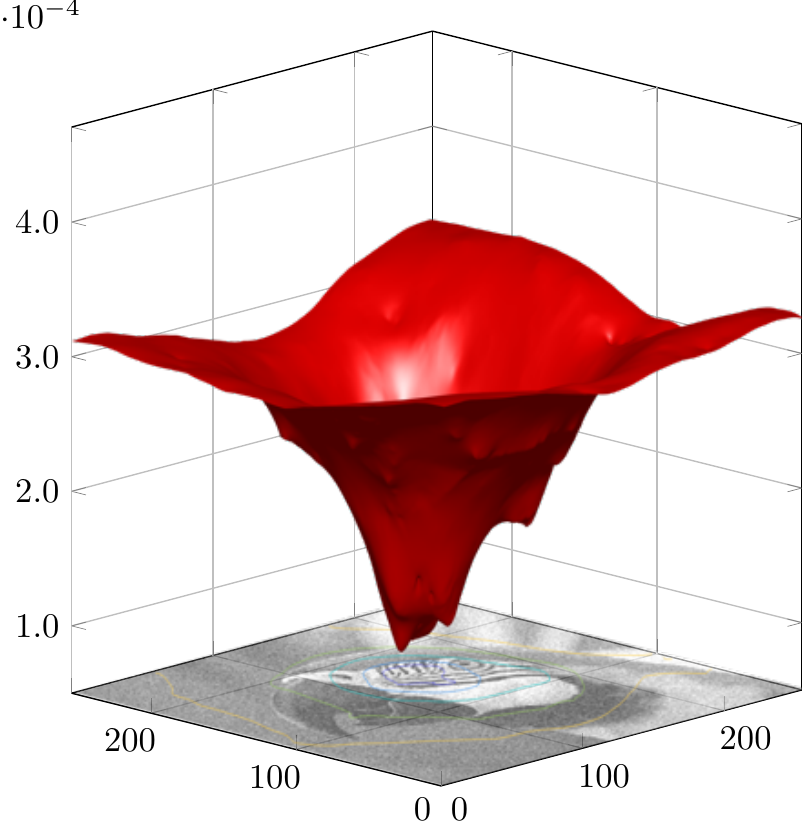}};
	\end{tikzpicture}	
	}
	\end{minipage}
		\begin{minipage}[t]{0.23\textwidth}
	  	\resizebox{0.95\textwidth}{!}{
%	\begin{tikzpicture}
%		\begin{axis}[xmin=0,xmax=255,ymin=0,ymax=255,width=9cm,
%		zticklabel style={
%          	  /pgf/number format/fixed,
%            	/pgf/number format/precision=1,
%          	  /pgf/number format/fixed zerofill
%      			  },
%		height=9.25cm,grid=both, zmin=0.00005]
%		\addplot3 graphics[points={%
%(8.658,207.4516,0.00024322) => (109.4087,235.8813)
%(205.5902,217.218,0.0002584) => (268.0013,297.11)
%(177.748,41.8934,0.00021325) => (398.4887,209.7837)
%(38.7887,148.6277,0.00012909) => (186.6975,122.4575)
%		}]
%		{orestis_tgv_trans2};
%		\end{axis}	
%	\end{tikzpicture}
	\begin{tikzpicture}
	\node[anchor=south west,inner sep=0] at (0,0) {\includegraphics{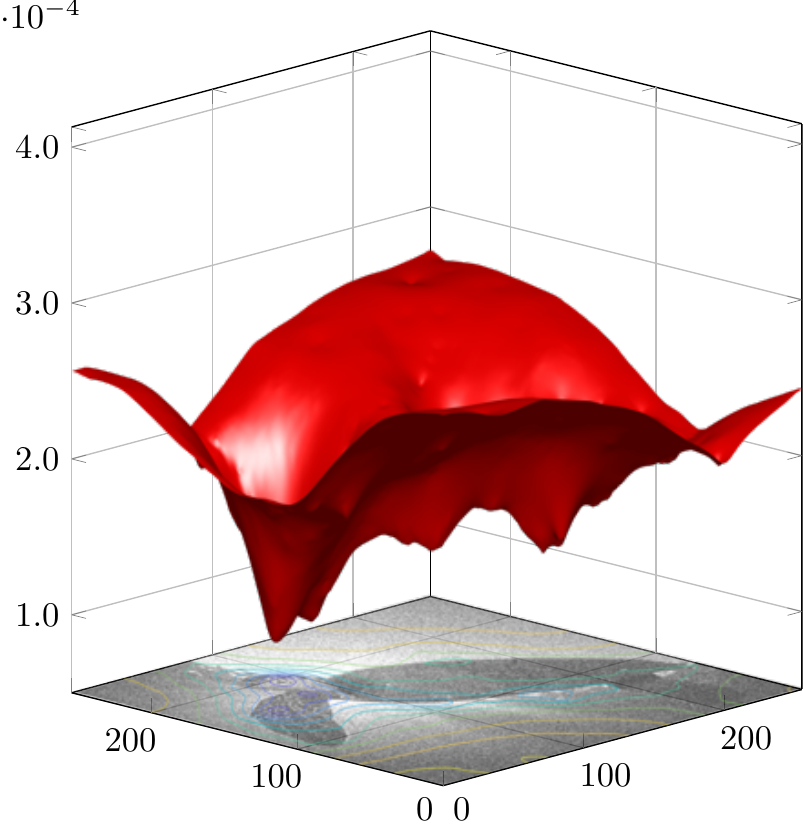}};
	\end{tikzpicture}
	}
	\end{minipage}
		\begin{minipage}[t]{0.23\textwidth}
	  	\resizebox{0.95\textwidth}{!}{
%	\begin{tikzpicture}
%		\begin{axis}[xmin=0,xmax=255,ymin=0,ymax=255,width=9cm,
%		zticklabel style={
%          	  /pgf/number format/fixed,
%            	/pgf/number format/precision=1,
%          	  /pgf/number format/fixed zerofill
%      			  },
%		height=9.25cm,grid=both, zmin=0.0000]
%		\addplot3 graphics[points={%
%(219.4267,221.4451,0.00024888) => (276.0312,339.2675)
%(51.3426,216.7004,0.00023836) => (137.5137,290.0838)
%(53.0496,129.2863,0.00014419) => (215.8062,182.6825)
%(165.7634,69.1422,0.000211) => (364.3612,257.9638)
%		}]
%		{hatchling_tgv_trans2};
%		\end{axis}	
%	\end{tikzpicture}
	\begin{tikzpicture}
	\node[anchor=south west,inner sep=0] at (0,0) {\includegraphics{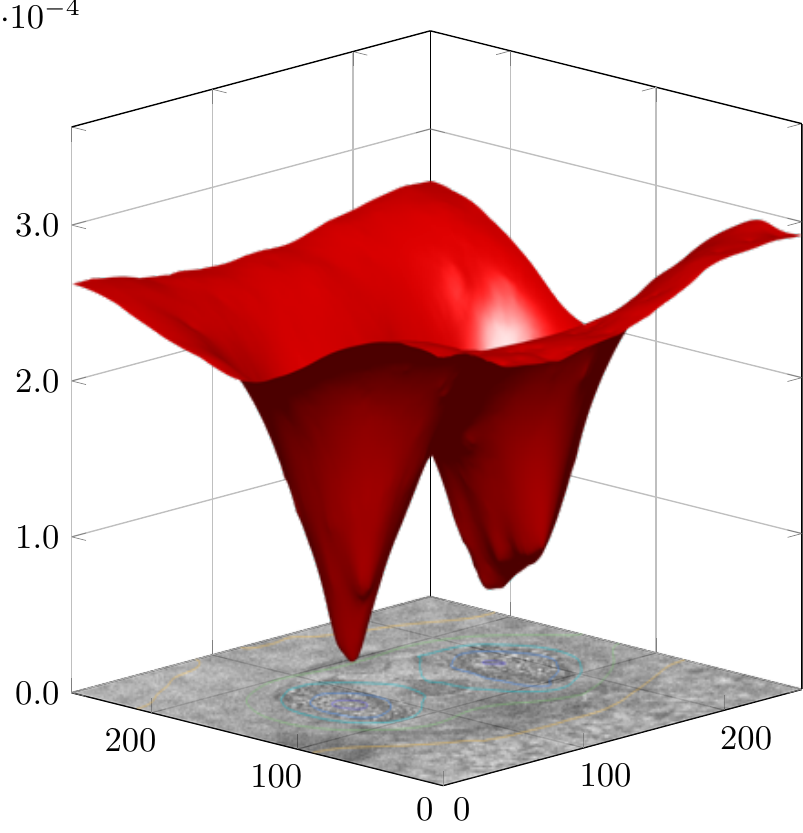}};
	\end{tikzpicture}
	}
	\end{minipage}	
	
\begin{minipage}[t]{0.23\textwidth}
	  	\resizebox{0.95\textwidth}{!}{
%	\begin{tikzpicture}
%		\begin{axis}[xmin=0,xmax=255,ymin=0,ymax=255,width=9cm,
%		zticklabel style={
%          	  /pgf/number format/fixed,
%            	/pgf/number format/precision=1,
%          	  /pgf/number format/fixed zerofill
%      			  },
%		height=9.25cm,grid=both, zmin=0.04, zmax=0.09]
%		\addplot3 graphics[points={%
%(29.1129,231.0518,0.083288) => (105.3937,324.2113)
%(221.0216,230.9307,0.083349) => (269.005,368.3763)
%(102.94,143.0676,0.0656) => (245.9187,214.8025)
%(235.3597,34.3021,0.072706) => (454.6987,263.9862)
%		}]
%		{cam_TGVpd};
%		\end{axis}	
%	\end{tikzpicture}
	\begin{tikzpicture}
	\node[anchor=south west,inner sep=0] at (0,0) {\includegraphics{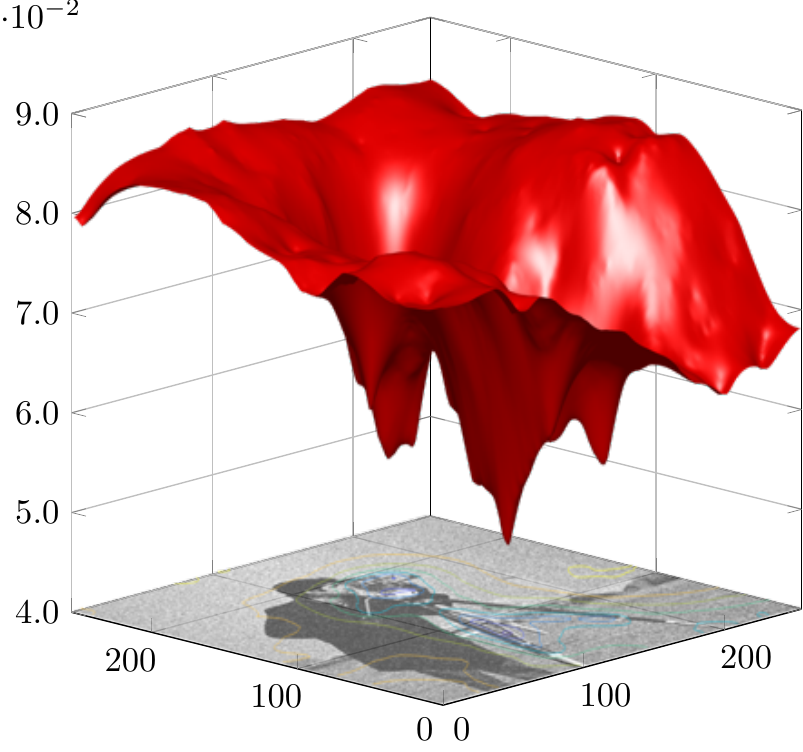}};
	\end{tikzpicture}
	
	}
	\end{minipage}	
\begin{minipage}[t]{0.23\textwidth}
	  	\resizebox{0.95\textwidth}{!}{
%	\begin{tikzpicture}
%		\begin{axis}[xmin=0,xmax=255,ymin=0,ymax=255,width=9cm,
%		zticklabel style={
%          	  /pgf/number format/fixed,
%            	/pgf/number format/precision=1,
%          	  /pgf/number format/fixed zerofill
%      			  },
%		zticklabels={0.02, 0.04, 0.06, 0.08, 0.1},	  
%		height=9.25cm,grid=both, zmin=0.04, zmax=0.1]
%		\addplot3 graphics[points={%
%(135.8037,158.9606,0.050378) => (259.9712,124.465)
%(52.229,224.4075,0.090117) => (131.4912,319.1925)
%(227.3278,198.2457,0.093726) => (303.1325,371.3875)
%(228.0298,13.9476,0.093962) => (465.74,332.2412)
%		}]
%		{parrot_TGVpd};
%		\end{axis}	
%	\end{tikzpicture}
	\begin{tikzpicture}
	\node[anchor=south west,inner sep=0] at (0,0) {\includegraphics{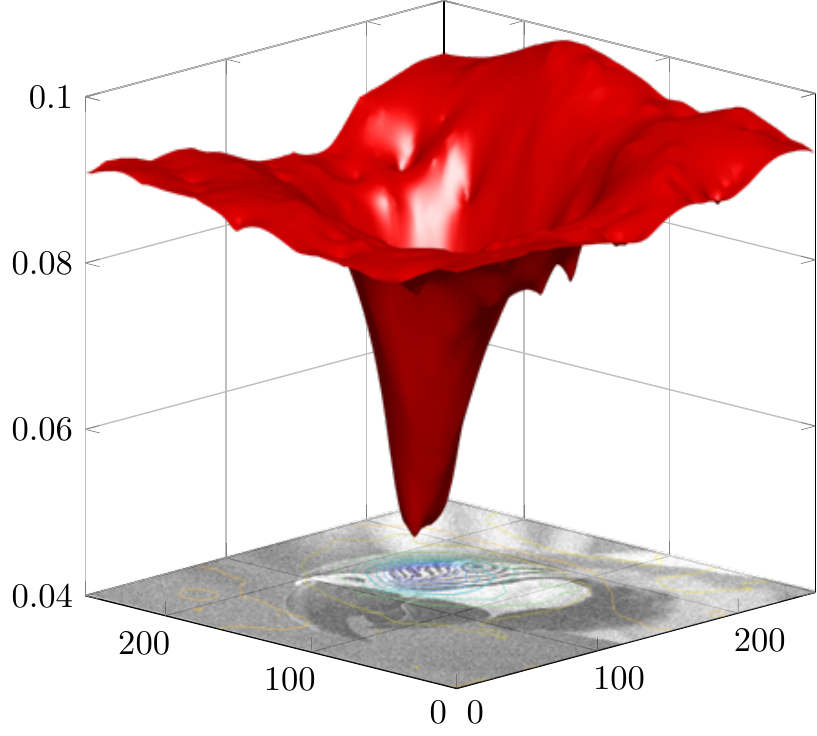}};
	\end{tikzpicture}
	}
	\end{minipage}	
\begin{minipage}[t]{0.23\textwidth}
	  	\resizebox{0.95\textwidth}{!}{
%	\begin{tikzpicture}
%		\begin{axis}[xmin=0,xmax=255,ymin=0,ymax=255,width=9cm,
%		zticklabel style={
%          	  /pgf/number format/fixed,
%            	/pgf/number format/precision=1,
%          	  /pgf/number format/fixed zerofill
%      			  },
%		zticklabels={0.04, 0.05, 0.06, 0.07, 0.08, 0.09, 0.1},	  
%		height=9.25cm,grid=both, zmin=0.05, zmax=0.1]
%		\addplot3 graphics[points={%
%(41.5992,151.3418,0.059819) => (186.6975,108.405)
%(98.7842,231.6849,0.077734) => (164.615,246.9225)
%(215.4277,134.2299,0.085886) => (349.305,301.125)
%(224.4557,35.6163,0.078473) => (443.6575,236.885)
%		}]
%		{orestis_TGVpd};
%		\end{axis}	
%	\end{tikzpicture}
	\begin{tikzpicture}
	\node[anchor=south west,inner sep=0] at (0,0) {\includegraphics{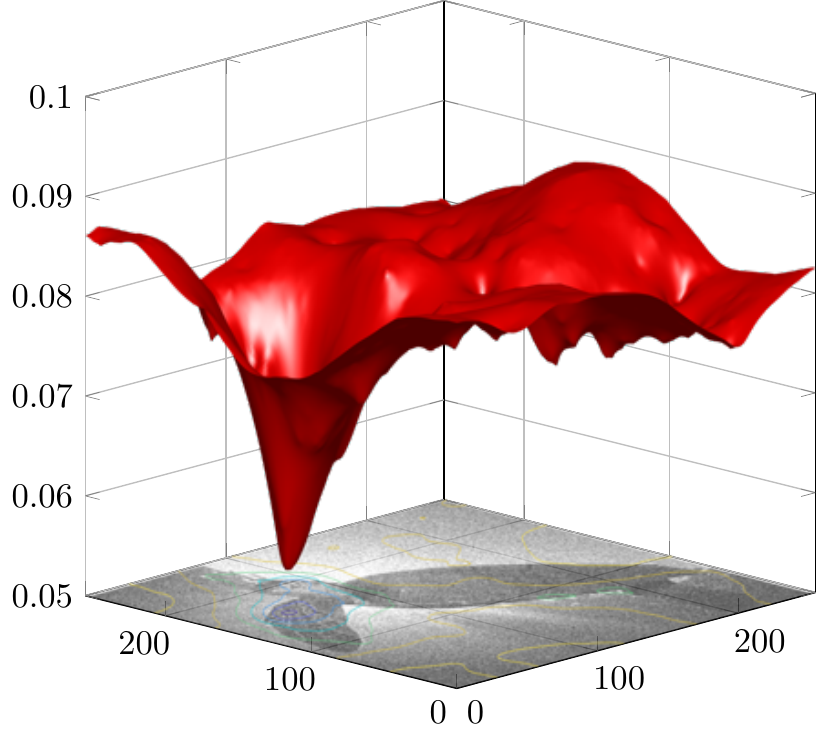}};
	\end{tikzpicture}
	}
	\end{minipage}
	\begin{minipage}[t]{0.23\textwidth}
	  	\resizebox{0.95\textwidth}{!}{
%	\begin{tikzpicture}
%		\begin{axis}[xmin=0,xmax=255,ymin=0,ymax=255,width=9cm,
%		zticklabel style={
%          	  /pgf/number format/fixed,
%            	/pgf/number format/precision=1,
%          	  /pgf/number format/fixed zerofill
%      			  },
%		zticklabels={0.02, 0.04, 0.06,  0.08, 0.1},	  
%		height=9.25cm,grid=both, zmin=0.04, zmax=0.1]
%		\addplot3 graphics[points={%
%(41.2114,219.4769,0.081429) => (126.4725,272.0162)
%(219.718,221.7264,0.089041) => (276.0312,351.3125)
%(172.719,86.1356,0.075302) => (355.3275,241.9038)
%(65.0695,130.6174,0.058828) => (224.84,144.54)
%		}]
%		{hatchling_TGVpd};
%		\end{axis}	
%	\end{tikzpicture}
	\begin{tikzpicture}
	\node[anchor=south west,inner sep=0] at (0,0) {\includegraphics{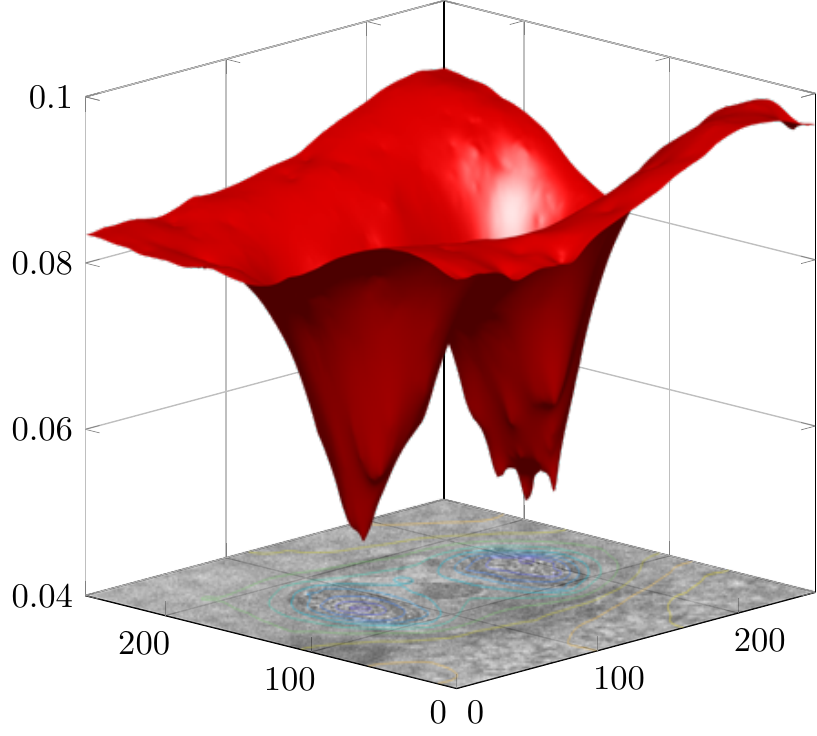}};
	\end{tikzpicture}
	}
	\end{minipage}					
	
\caption{First row: the computed regularization functions $\alpha$ for bilevel TV.\\ Second row: the computed regularization functions $\alpha_{1}$ for bilevel TGV--dual. Third row:  the computed regularization functions $\alpha_{1}$ for bilevel TGV--primal-dual.}	
\label{weights_contour}
\end{figure}
%}

Finally, we show an example where also the weight $\alpha_{0}$ varies spatially. For simplicity we use here only the primal-dual version \eqref{bilevelTGV_h_pd}. We note that by spatially varying both TGV parameters, the reduced problem becomes highly non-convex with many combinations of these parameters leading to similar values for the upper level objective. In order to deal with this, we use the following initialization strategy, which according to our numerical experiments, produces satisfactory results. We keep the spatial weight $\alpha_{1}$ fixed,  as it has been computed from the previous experiments, see the last row of Figure \ref{weights_contour}, and we optimize only with respect to a spatially varying  $\alpha_{0}$. As initialization for $\alpha_{0}$, we set it constant, equal to 5.

%\cmmnt{
\begin{figure}[h!]
	\centering		
	\begin{minipage}[t]{0.23\textwidth}
	  	\resizebox{0.95\textwidth}{!}{
	\begin{tikzpicture}
		\begin{axis}[xmin=0,xmax=255,ymin=0,ymax=255,width=9cm, view={-46}{15},
		zticklabel style={
          	  /pgf/number format/fixed,
            	/pgf/number format/precision=1,
          	  /pgf/number format/fixed zerofill
      			  },
		height=9.25cm,grid=both, zmin=0.000, zmax=0.45]
		\addplot3[fill=blue]
		coordinates{
		(0,0,0.0753)
		(0,255,0.0753)
		(255,255,0.0753)
		(255,0,0.0753)
		};
		\end{axis}	
	\end{tikzpicture}}
	\end{minipage}
	\begin{minipage}[t]{0.23\textwidth}
	  	\resizebox{0.95\textwidth}{!}{
	\begin{tikzpicture}
		\begin{axis}[xmin=0,xmax=255,ymin=0,ymax=255,width=9cm, view={-46}{15},
		zticklabel style={
          	  /pgf/number format/fixed,
            	/pgf/number format/precision=1,
          	  /pgf/number format/fixed zerofill
      			  },
		height=9.25cm,grid=both, zmin=0.000, zmax=0.55]
		\addplot3[fill=blue]
		coordinates{
		(0,0,0.0911)
		(0,255,0.0911)
		(255,255,0.0911)
		(255,0,0.0911)
		};
		\end{axis}	
	\end{tikzpicture}}
	\end{minipage}
		\begin{minipage}[t]{0.23\textwidth}
	  	\resizebox{0.95\textwidth}{!}{
	\begin{tikzpicture}
		\begin{axis}[xmin=0,xmax=255,ymin=0,ymax=255,width=9cm, view={-46}{15},
		zticklabel style={
          	  /pgf/number format/fixed,
            	/pgf/number format/precision=1,
          	  /pgf/number format/fixed zerofill
      			  },
		height=9.25cm,grid=both, zmin=0.000, zmax=0.4]
		\addplot3[fill=blue]
		coordinates{
		(0,0,0.1057)
		(0,255,0.1057)
		(255,255,0.1057)
		(255,0,0.1057)
		};
		\end{axis}	
	\end{tikzpicture}}
	\end{minipage}
	\begin{minipage}[t]{0.23\textwidth}
	  	\resizebox{0.95\textwidth}{!}{
	\begin{tikzpicture}
		\begin{axis}[xmin=0,xmax=255,ymin=0,ymax=255,width=9cm, view={-46}{15},
		zticklabel style={
          	  /pgf/number format/fixed,
            	/pgf/number format/precision=2,
          	  /pgf/number format/fixed zerofill
      			  },
		height=9.25cm,grid=both, zmin=0.000, zmax=0.25]
		\addplot3[fill=blue]
		coordinates{
		(0,0,0.0803)
		(0,255,0.0803)
		(255,255,0.0803)
		(255,0,0.0803)
		};
		\end{axis}	
	\end{tikzpicture}}
	\end{minipage}
	
		 \vspace{-6pt}
	\begin{minipage}[t]{0.23\textwidth}
	\centering \scalebox{.65}{PSNR=27.42, SSIM=0.8077}
	\end{minipage}%}
	\begin{minipage}[t]{0.23\textwidth}
	\centering \scalebox{.65}{PSNR=29.47, SSIM=0.8628}
	\end{minipage}
	\begin{minipage}[t]{0.23\textwidth}
	 \centering \scalebox{.65}{PSNR=$\mathbf{29.63}$, SSIM=0.8305}
	\end{minipage}
	\begin{minipage}[t]{0.23\textwidth}
	\centering \scalebox{.65}{PSNR=$\mathbf{28.01}$, SSIM=$\mathbf{0.8037}$}
	\end{minipage}\vspace{8pt}	
	
	\begin{minipage}[t]{0.23\textwidth}
	  	\resizebox{0.95\textwidth}{!}{
%	\begin{tikzpicture}
%		\begin{axis}[xmin=0,xmax=255,ymin=0,ymax=255,width=9cm,
%		zticklabel style={
%          	  /pgf/number format/fixed,
%            	/pgf/number format/precision=1,
%          	  /pgf/number format/fixed zerofill
%      			  },
%		height=9.25cm,grid=both, zmin=0.000, zmax=0.45]
%		\addplot3 graphics[points={%
%(26.0186,211.6564,0.16089) => (120.45,166.6225)
%(27.4852,122.8683,0.080012) => (199.7463,93.3487)
%(166.1742,158.6015,0.16516) => (286.0687,189.7087)
%(225.7687,216.1512,0.35858) => (286.0687,345.29)
%		}]
%		{cam_alpha0_trans.png};
%		\end{axis}	
%	\end{tikzpicture}
	\begin{tikzpicture}
	\node[anchor=south west,inner sep=0] at (0,0) {\includegraphics{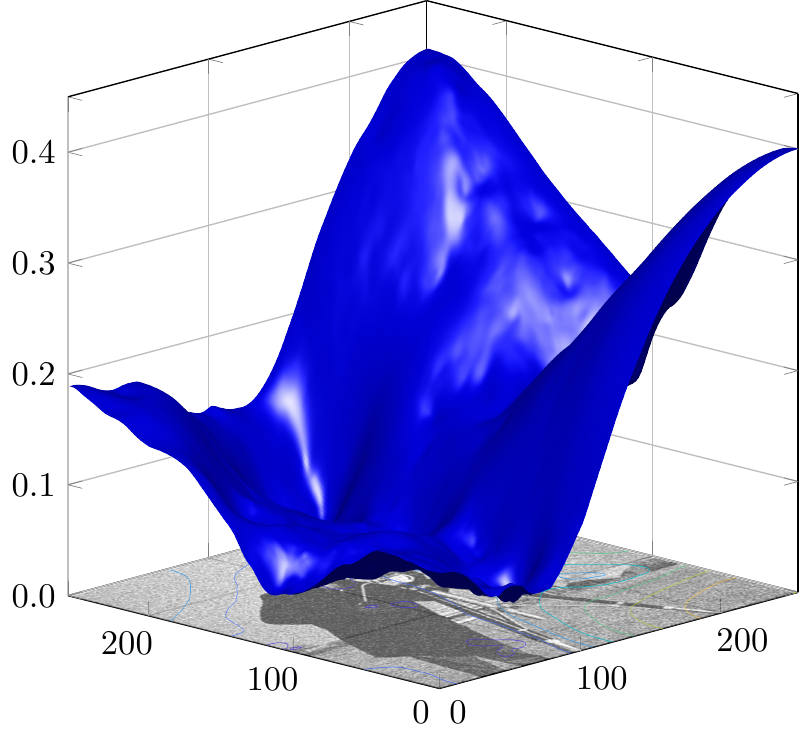}};
	\end{tikzpicture}
	
	}
	\end{minipage}
	\begin{minipage}[t]{0.23\textwidth}
	  	\resizebox{0.95\textwidth}{!}{
%	\begin{tikzpicture}
%		\begin{axis}[xmin=0,xmax=255,ymin=0,ymax=255,width=9cm,
%		zticklabel style={
%          	  /pgf/number format/fixed,
%            	/pgf/number format/precision=1,
%          	  /pgf/number format/fixed zerofill
%      			  },
%		height=9.25cm,grid=both, zmin=0.000, zmax=0.55]
%		\addplot3 graphics[points={%
%(63.5405,255.8837,0.45674) => (113.4237,327.2225)
%(126.9842,234.939,0.24436) => (185.6937,220.825)
%(92.5223,40.6618,0.21013) => (327.2225,151.5662)
%(192.0373,28.2886,0.089363) => (422.5787,105.3937)
%		}]
%		{parrot_alpha0_trans};
%		\end{axis}	
%	\end{tikzpicture}
	\begin{tikzpicture}
	\node[anchor=south west,inner sep=0] at (0,0) {\includegraphics{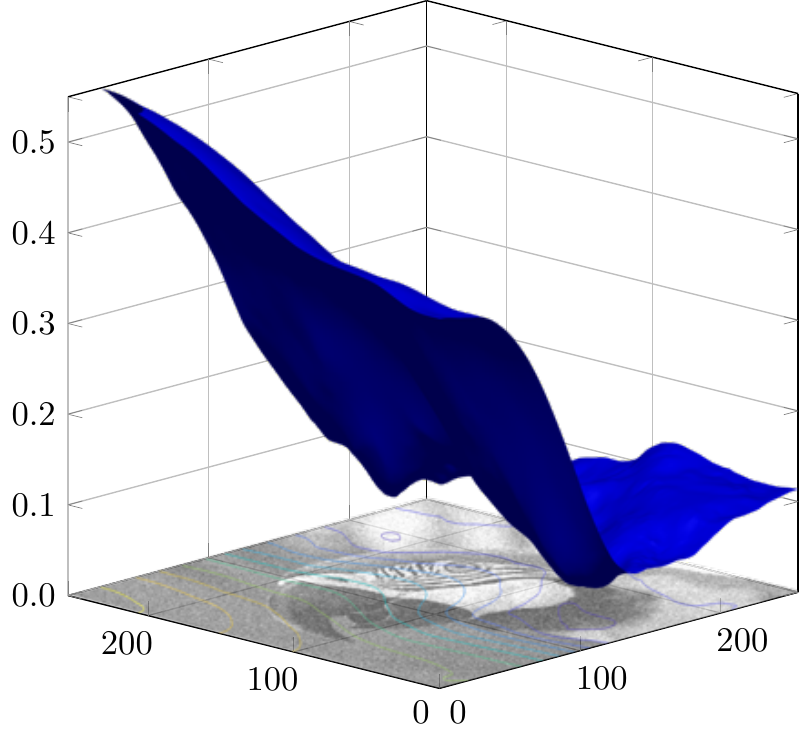}};
	\end{tikzpicture}
	
	}
	\end{minipage}
	\begin{minipage}[t]{0.23\textwidth}
	  	\resizebox{0.95\textwidth}{!}{
%	\begin{tikzpicture}
%		\begin{axis}[xmin=0,xmax=255,ymin=0,ymax=255,width=9cm,
%		zticklabel style={
%          	  /pgf/number format/fixed,
%            	/pgf/number format/precision=1,
%          	  /pgf/number format/fixed zerofill
%      			  },
%		height=9.25cm,grid=both, zmin=0.000, zmax=0.4]
%		\addplot3 graphics[points={%
%(233.3628,218.9174,0.34764) => (290.0838,369.38)
%(30.8767,232.3335,0.12884) => (106.3975,161.6037)
%(35.6826,5.1833,0.20971) => (310.1587,173.6487)
%(65.1622,99.8775,0.10379) => (251.9412,121.4537)
%		}]
%		{orestis_alpha0_trans};
%		\end{axis}	
%	\end{tikzpicture}
		\begin{tikzpicture}
	\node[anchor=south west,inner sep=0] at (0,0) {\includegraphics{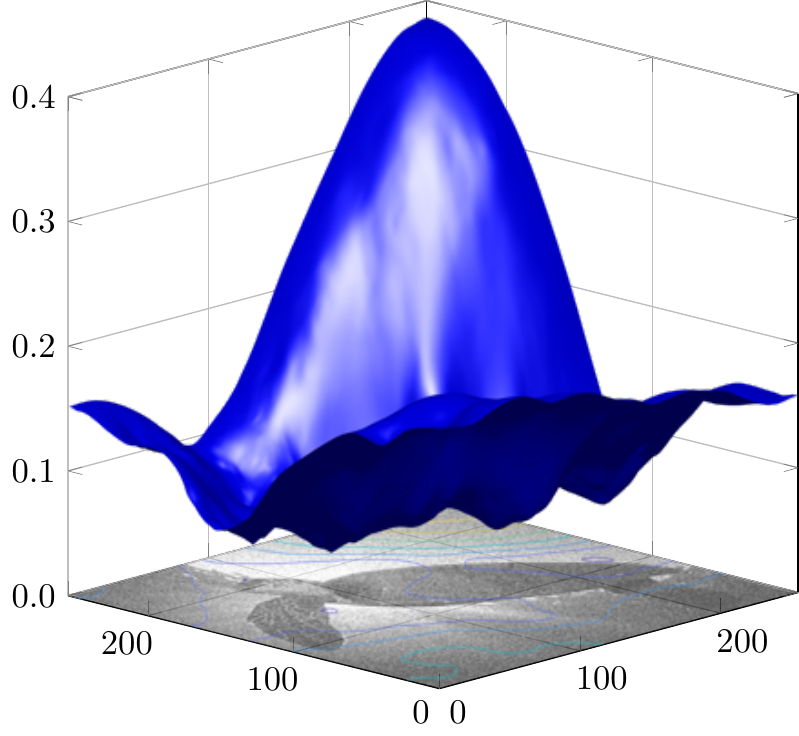}};
	\end{tikzpicture}

	}
	\end{minipage}
	\begin{minipage}[t]{0.23\textwidth}
	  	\resizebox{0.95\textwidth}{!}{
%	\begin{tikzpicture}
%		\begin{axis}[xmin=0,xmax=255,ymin=0,ymax=255,width=9cm,
%		zticklabel style={
%          	  /pgf/number format/fixed,
%            	/pgf/number format/precision=1,
%          	  /pgf/number format/fixed zerofill
%      			  },
%		height=9.25cm,grid=both, zmin=0.000, zmax=0.25]
%		\addplot3 graphics[points={%
%(23.6746,240.2223,0.11629) => (93.3487,204.765)
%(187.6217,190.7314,0.096011) => (276.0312,206.7725)
%(101.2141,52.4808,0.055868) => (324.2113,108.405)
%(238.5958,4.74,0.19785) => (482.8037,300.1212)
%		}]
%		{hatchling_alpha0_trans};
%		\end{axis}	
%	\end{tikzpicture}
		\begin{tikzpicture}
	\node[anchor=south west,inner sep=0] at (0,0) {\includegraphics{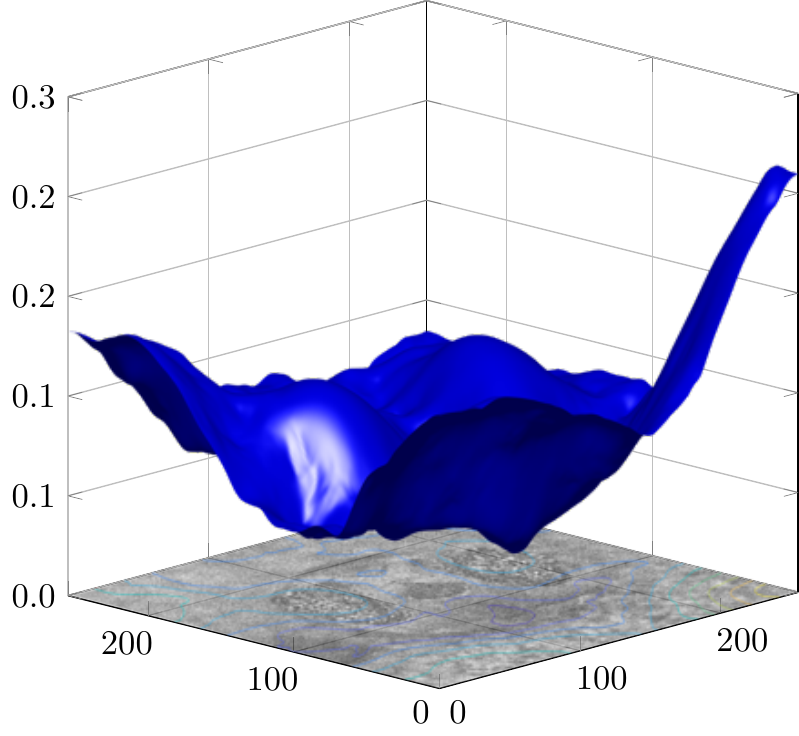}};
	\end{tikzpicture}
	}
	\end{minipage}	
	
	 \vspace{-6pt}
	\begin{minipage}[t]{0.23\textwidth}
	\centering \scalebox{.65}{PSNR=$\mathbf{27.44}$, SSIM=$\mathbf{0.8098}$}
	\end{minipage}
	\begin{minipage}[t]{0.23\textwidth}
	\centering \scalebox{.65}{PSNR=$\mathbf{29.50}$, SSIM=$\mathbf{0.8640}$}
	\end{minipage}
	\begin{minipage}[t]{0.23\textwidth}
	 \centering \scalebox{.65}{PSNR=$\mathbf{29.63}$, SSIM=$\mathbf{0.8310}$}
	\end{minipage}
	\begin{minipage}[t]{0.23\textwidth}
	\centering \scalebox{.65}{PSNR=27.96, SSIM=0.8012}
	\end{minipage}\vspace{8pt}	

	\begin{minipage}[t]{0.23\textwidth}
	\includegraphics[width=0.95\textwidth, trim={0cm 0cm 0cm 0cm},clip]{parrot_wTGVpd}
	\end{minipage}
	\begin{minipage}[t]{0.23\textwidth}
	\includegraphics[width=0.95\textwidth, trim={0cm 0cm 0cm 0cm},clip]{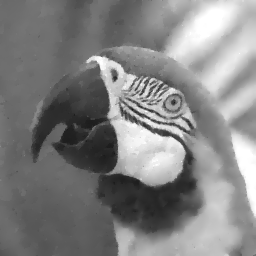}
	\end{minipage}
	\begin{minipage}[t]{0.23\textwidth}
	\includegraphics[width=0.95\textwidth, trim={0cm 3cm 6.0cm 3.0cm},clip]{parrot_wTGVpd}
	\end{minipage}
	\begin{minipage}[t]{0.23\textwidth}
	\includegraphics[width=0.95\textwidth,  trim={0cm 3.0cm 6.0cm 3.0cm},clip]{parrot_wTGVpd_spatiala0}
	\end{minipage}
	 \vspace{-6pt}
	\begin{minipage}[t]{0.23\textwidth}
	\centering \small{spatial $\alpha_{1}$, scalar $\alpha_{0}$}
	\end{minipage}
	\begin{minipage}[t]{0.23\textwidth}
	\centering \small{spatial $\alpha_{1}$, spatial $\alpha_{0}$}
	\end{minipage}
	\begin{minipage}[t]{0.23\textwidth}
	 \centering \small{spatial $\alpha_{1}$, scalar $\alpha_{0}$ \\(detail)}
	\end{minipage}
	\begin{minipage}[t]{0.23\textwidth}
	\centering \small{spatial $\alpha_{1}$, spatial $\alpha_{0}$ \\(detail)}
	\end{minipage}\vspace{8pt}		
	
\caption{Experiments with optimizing over a  spatially varying $\alpha_{0}$. Top row: the automatically computed scalar parameters $\alpha_{0}$, that correspond to the images of the last row of Figure \ref{fig:weightedTGV_01}. Middle row: the automatically computed spatially varying parameters $\alpha_{0}$, where $\alpha_{1}$ has been kept fixed (last row of Figure \ref{weights_contour}). 
The weight $\alpha_{0}$ is adapted to piecewise constant parts having there large values and hence promoting TV like behaviour, see for instance the parrot image at the last row.  On the contrary $\alpha_{0}$ has low values in piecewise smooth parts promoting a TGV like behaviour reducing the staircasing.}	
\label{fig:alpha0}
\end{figure}
%}

In Figure \ref{fig:alpha0} we depict the computed spatially varying parameters $\alpha_{0}$ as well as the corresponding PSNR and SSIM values. Observe that the shape of $\alpha_{0}$ is  different to the one of $\alpha_{1}$, compare the last row of Figure \ref{weights_contour} to the second row of Figure \ref{fig:alpha0}. This implies that a non-constant ratio of $\alpha_{0}/\alpha_{1}$ is preferred throughout the image domain.  Secondly, by spatially varying $\alpha_{0}$ we only get a  slight improvement with respect to PSNR and SSIM in all images, apart from the last one. However, it is interesting to observe the spatial adaptation of $\alpha_{0}$ with respect to piecewise constant versus piecewise smooth areas. The values of $\alpha_{0}$ are high in large piecewise constant areas, like the background of cameraman, the left area of the parrot image, as well as the top-right corner of the turtle image. This is  not so surprising as large values of $\alpha_{0}$ imply a large ratio $\alpha_{0}/\alpha_{1}$ and  a promotion of TV like behaviour in those areas. We can observe this in more detail at the parrot image, see last row of Figure \ref{fig:alpha0}. On the contrary, the values of $\alpha_{0}$ are kept small in piecewise smooth areas like  the right part of the parrot image and the sun rays around the turtle's body. This results in low ratio $\alpha_{0}/\alpha_{1}$  and thus to a more TGV like behaviour, reducing the staircasing effect. This is another indication of the fact that by minimizing the statistics-based upper level objective one is able not only to better preserve detailed areas  but also to finely adjust the TGV parameters such that the staircasing is reduced.

\section{Conclusion}

In this work we have adapted the bilevel optimization framework of \cite{hintermuellerPartI, hintermuellerPartII} for automatically computing spatially dependent regularization parameters for the TGV regularizer. For that we first examined two variants of the TGV regularization problem establishing  rigorous dualization frameworks that form the basis for their algorithmic treatment via  Newton methods. We showed that the bilevel optimization framework with the statistics/localized residual based upper level objective is able to automatically produce spatially varying parameters that not only adapt to the level of detail in the image but also reduce the staircasing effect. 

Future continuation of this work includes adaptation of the bilevel TGV framework for advanced inverse problems tasks, i.e., Magnetic Resonance Imaging (MRI) and Positron Emission Tomography (PET) reconstruction as well as in multimodal medical imaging problems where structural TV based regularizers (edge aligning) have been suggested. Adaptation of the framework for different noise distributions e.g. Poisson, Salt \& Pepper as well as combination of those  \cite{calatronimixed,CalatroniPapafitsoros}, should also be investigated. A fine structural analysis of the weighted TGV regularized solutions in the spirit of \cite{mine_spatial, jalalzai2014discontinuities} would be also of interest.

\subsection*{Acknowledgements}  This work is funded by the Deutsche Forschungsgemeinschaft (DFG, German Research Foundation) under Germany's Excellence Strategy -- The Berlin Mathematics Research Center MATH+ (EXC-2046/1, project ID: 390685689). M.H and K.P. acknowledge support of Institut Henri Poincar\'e (UMS 839 CNRS-Sorbonne Universit\'e), and LabEx CARMIN (ANR-10-LABX-59-01). H.S.  acknowledges the financial support of Alexander von Humboldt Foundation. This work is partially supported by NSF grant  DMS-2012391.

\appendix

\section{Proof of Proposition \ref{lbl:weighted_TGV_Wdiv}}\label{sec:app.proof}
\begin{proof}
The proof follows \cite[Proposition 3.3]{TGV_decompression_part1}. Denote by $C_{\balpha}$, $K_{\balpha}$ the following convex sets 
\begin{align}
C_{\balpha}&=\left \{\di^{2}\phi:\;  \phi\in C_{c}^{\infty}(\om,\mathcal{S}^{d\times d}),\; |\phi(x)|_{r}\le \alpha_{0}(x),\; |\di\phi(x)|_{r}\le \alpha_{1}(x), \text { for all }x\in\om\right \},\\
K_{\balpha}&=\left \{\di^{2}p:\;  p  \in  W_{0}^{d}(\di^{2};\om),\; |p(x)|_{r}\le \alpha_{0}(x),\; |\di p (x)|_{r}\le \alpha_{1}(x), \text { for a.e. }x\in\om\right \}.\label{lbl:K_balpha}
\end{align}
It suffices to show that 
\begin{equation}
\overline{C_{\balpha}}^{L^{d}(\om)}=K_{\balpha}.\label{density_weighted_TGV}
\end{equation}
We first show that $K_{\balpha}$ is closed in $L^{d}(\om)$. Let $g\in K_{\balpha}$ and assume that there exists $(p_{n})_{n\in\NN}\subset W_{0}^{d}(\di^{2};\om)$ where every $p_{n}$ satisfies the convex constraints and $\di^{2} p_{n}\to g$ in $L^{d}(\om)$. By boundedness of $\alpha_{0},\alpha_{1}$ we have that there exist $h_{0}\in L^{d}(\om,\mathcal{S}^{d\times d})$, $h_{1}\in L^{d}(\om,\RR^{d})$ and a subsequence of $(p_{n_{k}})_{k\in\NN}$ such that 
\[p_{n_{k}} \rightharpoonup h_{0} \quad \text{and} \quad \di p_{n_{k}} \rightharpoonup h_{1},\]
in $L^{d}(\om)$ and $L^{d}(\om,\RR^{d})$ respectively. Using that, we have for every 
%$\phi\in C_{c}^{\infty}(\om;\mathcal{S}^{d\times d})$
$\phi\in C_{c}^{\infty}(\om,\RR^{d})$
\begin{equation}\label{h0_div}
\int_{\om} \nabla \phi \cdot h_{0}\,dx=\lim_{k\to\infty} \int_{\om} \nabla \phi  \cdot p_{n_{k}} \,dx=-\lim_{k\to\infty} \int_{\om} \phi \cdot \di p_{n_{k}}\,dx= - \int_{\om} \phi \cdot h_{1}\,dx,
\end{equation}
thus $h_{1}=\di h_{0}$. Similarly we derive that $g=\di h_{1}= \di^{2} h_{0}$ and hence $h_{0}\in W^{d}(\di^{2};\om)$. Finally note that the set 
\[\left \{ (h,\di h, \di^{2} h):\; h\in W_{0}^{d}(\di^{2};\om), \; |h(x)|_{r}\le \alpha_{0}(x),\; |\di h (x)|_{r}\le \alpha_{1}(x), \text { for a.e. }x\in\om  \right \},\]
is a norm-closed and convex subset of $ L^{d}\left(\om,(\mathcal{S}^{d\times d}\times \RR^{d}\times \RR)\right)$ and hence weakly closed. Since  $(p_{n_{k}},\di p_{n_{k}}, \di^{2}p_{n_{k}})_{k\in\NN}$ belongs to that set,  converging weakly to $(h_{0},\di h_{0}, \di^{2}h_{0})$ we get that the latter also belongs there. Thus, $K_{\balpha}$ is closed in $L^{d}(\om)$ and since  $C_{\balpha}\subset K_{\balpha}$, we get  $\overline{C_{\balpha}}^{L^{d}(\om)}\subset K_{\balpha}$.

It remains to show the other direction, i.e.,  $K_{\balpha}\subset \overline{C_{\balpha}}^{L^{d}(\om)}$. Towards that, note first that the functional $\tgv_{\balpha}^{2}(\om): L^{d/d-1}(\om)\to \overline{\RR}$, can also be written as \[\tgv_{\balpha}^{2}(u)=\mathcal{I}_{C_{\balpha}}^{\ast}(u).\]
Using  the convexity of $C_{\balpha}$ one gets
\[\tgv_{\balpha}^{2^{\,\ast}}(v)=\mathcal{I}_{C_{\balpha}}^{\ast\ast}(v)=\mathcal{I}_{\overline{C_{\balpha}}^{L^{d}(\om)}}(v).\] 
Secondly, note that due to the lower bounds on $\alpha_{0},\alpha_{1}$, for $u\in L^{d/d-1}(\om)$, we have that $\tgv_{\balpha}^{2}(u)<\infty$ if and only if $u\in \bv(\om).$ Indeed this holds from the equivalence of the (scalar) $\|\cdot\|_{\mathrm{BGV}}$ with $\|\cdot\|_{\bv(\om)}$ and from the estimate
\[\tgv_{\underline{\alpha},\underline{\alpha}}^{2}(u)\le \tgv_{\balpha}^{2}(u)\le \|\alpha_{1}\|_{\infty} \tv (u),\]
for every $u\in L^{d/d-1}(\om)$. This means that if for  $\di^{2}p\in K_{\balpha}$ it holds
\begin{equation}\label{ineq_for_bv}
\int_{\om} u\,\di^{2}p\,dx \le \tgv_{\balpha}^{2}(u),\quad \text{for all }u\in \bv(\om), 
\end{equation}
then in fact the inequality \eqref{ineq_for_bv} will hold for every $u\in L^{d/d-1}(\om)$ and thus  
$\tgv_{\balpha}^{2^{\,\ast}}(\di^{2}p)=0$ which implies $\di^{2}p\in  \overline{C_{\balpha}}^{L^{d}(\om)}$. Thus in order to finish the proof it suffices to show \eqref{ineq_for_bv} for every $\di^{2}p\in K_{\balpha}$. In view of Proposition \ref{lbl:prop_weighted_TGV_min} it suffices to show 
\begin{equation}\label{ineq_for_bv_2}
\int_{\om} u\,\di^{2}p\,dx \le \min_{w\in\mathrm{BD}(\om)} \int_{\om} \alpha_{1}\,d |Du-w|_{r^{\ast}}+ \int_{\om}\alpha_{0}\, d|\mathcal{E}w|_{r^{\ast}}
\end{equation}
for all $u\in\bv(\om)$. The first step towards that is to show that for every $w\in \mathrm{BD}(\om)$ and for every $p\in W_{0}^{d}(\di^{2};\om)$ with $|p(x)|_{r}\le \alpha_{0}(x)$ for a.e. $x\in\om$, it holds
\begin{equation}\label{ind_step_1}
\int_{\om}w \, \di p\,dx \le \int_{\om} \alpha_{0}\, d|\mathcal{E}w|_{r^{\ast}}.
\end{equation}
Indeed, note first that from \eqref{ibp_2} and using the H\"older inequality, we get for every $\phi\in C^{\infty}(\overline{\om},\RR^{d})$
\begin{equation}\label{bd_before_strict}
\left |\int_{\om} \phi\cdot \di p\, dx \right| =\left | \int_{\om} p \cdot E\phi \,dx \right |\le \int_{\om} |p|_{r}|E\phi |_{r^{\ast}}\,dx\le \int_{\om}\, \alpha_{0}\,d|E\phi|_{r^{\ast}}.
\end{equation}
Recall now that every $w\in\mathrm{BD}(\om)$ can be strictly approximated by a sequence $(\phi_{n})_{n\in\NN}\subset  C^{\infty}(\overline{\om},\RR^{d})$, that is $\phi_{n}\to w$ in $L^{d/d-1}(\om,\RR^{d})$ and $|E\phi|_{r^{\ast}}(\om)\to |\mathcal{E}w|_{r^{\ast}}(\om)$, see \cite[Proposition 2.10]{TGV_decompression_part1}. Furthermore, using that, along with Reshetnyak's continuity theorem \cite[Theorem 2.39]{AmbrosioBV} we also get that 
\[\int_{\om} \alpha_{0}\, d|E\phi_{n}|_{r^{\ast}}\to \int_{\om} \alpha_{0}\, d|\mathcal{E}w|_{r^{\ast}},\quad\text{as }n\to\infty ..\]
Using that fact, by taking limits in \eqref{bd_before_strict} we obtain \eqref{ind_step_1}. Finally in order to obtain \eqref{ineq_for_bv} let $p\in W_{0}^{d}(\di^{2};\om)$ with $|p(x)|_{r}\le \alpha_{0}(x)$ and $|\di p (x)|_{r}\le\alpha_{1}(x)$ for a.e. $x\in\om$, and let $\phi\in C^{\infty}(\overline{\om},\RR)$. Then by using \eqref{ibp_3} and \eqref{ind_step_1},  we have for every $w\in \mathrm{BD}(\om)$
\begin{align*}
\int_{\om} \phi\,\di^{2}p\,dx
&\le\left | \int_{\om} \nabla \phi\cdot \di p\,dx \right |=\left | \int_{\om} (\nabla \phi -w)\cdot \di p\,dx + \int_{\om}w\,\di p\,dx \right |\\
&\le  \int_{\om} |\nabla \phi-w|_{r^{\ast}}|\di p|_{r}\,dx+ \int_{\om} \alpha_{0} \,d|\mathcal{E}w|_{r^{\ast}}\\
&\le \int_{\om} \alpha_{1} |\nabla \phi-w|_{r^{\ast}}\,dx+ \int_{\om} \alpha_{0} \,d|\mathcal{E}w|_{r^{\ast}}.
\end{align*}
Similarly as before given $u\in\bv(\om)$ and $w\in \mathrm{BD}(\om)$,  there exists a sequence $(\phi_{n})_{n\in\NN}\subset C^{\infty}(\overline{\om},\RR)$ such that $\phi_{n}\to u$ in $L^{d/d-1}(\om)$ and $|\nabla \phi_{n}-w|(\om)\to |Du-w|(\om)$, see again \cite[Proposition 2.10]{TGV_decompression_part1}. Using again the Reshetnyak's continuity theorem and taking limits  we get that for every $w\in \mathrm{BD}(\om)$
\[\int_{\om}u\,\di^{2}p\,dx\le \int_{\om} \alpha_{1}\,d |Du-w|_{r^{\ast}}+ \int_{\om}\alpha_{0}\, d|\mathcal{E}w|_{r^{\ast}}.
\] 
By taking the minimum over $w\in\mathrm{BD}(\om)$,  we obtain \eqref{ineq_for_bv}.

\end{proof}

\section{}\label{sec:app}
We provide here a few more details about the discrete differential operators involved in the implementation of Algorithm \ref{TGV_Newton}. Recall that  for the discrete gradient and divergence we have $\nabla: W_{h}\to V_{h}$ and $\di: V_{h}\to W_{h}$, where $\nabla=-\di^{\top}$ holds.
For $p\in V_{h}$, the divergence is defined  as
\begin{align*}
(\di p)_{i,j}^{1}
=\frac{1}{h} (p_{i,j}^{11}-p_{i-1,j}^{11}+p_{i,j}^{12}-p_{i,j-1}^{12}), \quad 
(\di p)_{i,j}^{2}
=\frac{1}{h} (p_{i,j}^{12}-p_{i-1,j}^{12}+p_{i,j}^{22}-p_{i,j-1}^{22}), \quad (i,j)\in \om_{h}.
\end{align*}
Here we set  zero values at the ghost points. For the second-order gradient $\nabla^{2}u:U_{h}\to V_{h}$ we have $\nabla^{2}=(D_{xx} u, D_{xy} u, D_{yy} u)$, where $D_{xx}, D_{xy}, D_{xy}$ are operators $U_{h}\to V_{h}$ and are defined using the following stencils with zero values at ghost points:
%\begin{center}
\begin{figure}[!h]
\begin{tikzpicture}
 \node at (3.6,0.2) {$D_{xy}$};
\node at (4.6,0.2) {$\frac{1}{h^{2}}\;\times$};
 \node[shape=circle,draw,  minimum size=13pt, inner sep=0pt, fill=white]  at (5.4,0.2) {-\tiny{$\tfrac{1}{2}$}};
 \node[shape=circle,draw, minimum size=13pt, inner sep=0pt, fill=mygray]  at (6.0,0.2) {\tiny{1}};
 \node[shape=circle,draw, minimum size=13pt, inner sep=0pt, fill=white]  at (6.6,0.2) {\tiny{-$\tfrac{1}{2}$}};
 \node[shape=circle,draw,  minimum size=13pt, inner sep=0pt, fill=white]  at (5.4,0.8) {\tiny{$\tfrac{1}{2}$}};
 \node[shape=circle,draw,  minimum size=13pt, inner sep=0pt, fill=white]  at (6.0,0.8) {\tiny{-$\tfrac{1}{2}$}};
 \node[shape=circle,draw, minimum size=13pt, inner sep=0pt, fill=white]  at (6.6,-0.4) {\tiny{$\tfrac{1}{2}$}};
 \node[shape=circle,draw, minimum size=13pt, inner sep=0pt, fill=white]  at (6.0,-0.4) {\tiny{-$\tfrac{1}{2}$}};
   
 \node at (8.4,0.2) {$D_{xx}$};
\node at (9.4,0.2) {$\frac{1}{h^{2}}\;\times$};
 \node[shape=circle,draw,  minimum size=13pt, inner sep=0pt, fill=mygray]  at (10.2,0.2) {\tiny{-$2$}};
 \node[shape=circle,draw,  minimum size=13pt, inner sep=0pt, fill=white]  at (10.2,0.8) {\tiny{$1$}};
 \node[shape=circle,draw,  minimum size=13pt, inner sep=0pt, fill=white]  at (10.2,-0.4) {\tiny{$1$}};
     
  \node at (12.0,0.2) {$D_{yy}$};
\node at (13.0,0.2) {$\frac{1}{h^{2}}\;\times$};
 \node[shape=circle,draw,  minimum size=13pt, inner sep=0pt, fill=white]  at (13.8,0.2) {\tiny{$1$}};   
  \node[shape=circle,draw,  minimum size=13pt, inner sep=0pt, fill=mygray]  at (14.4,0.2) {\tiny{-$2$}};  
  \node[shape=circle,draw,  minimum size=13pt, inner sep=0pt, fill=white]  at (15.0,0.2) {\tiny{$1$}};      
 \end{tikzpicture}
 %\caption{Symmetric finite difference stencil for the second order derivative operators}
 %\label{Stencil_second_order}
 \end{figure}
%\end{center}

\noindent
 Note that the use of symmetric differences for the mixed derivative results in a symmetric matrix representing $Dxy$. 
 All the resulting operators $D_{xx}, D_{xy}, D_{yy}$ are then symmetric. For the discrete second divergence $\di^{2}:V_{h}\to U_{h}$, we have $\di^{2}p=D_{xx}p^{11}+2D_{xy}p^{12}+D_{yy}p^{22}$. The vector bi-Laplacian is an operator $V_{h}\to V_{h}$ where $p\mapsto (\Delta^{2}p^{11}, \Delta^{2}p^{12}, \Delta^{2}p^{22})$ with $\Delta^{2}=D_{xxxx}+D_{yyyy}+D_{xxyy}+D_{yyxx}$.  The resulting stencil for $\Delta^{2}$ is as shown below. 
\begin{figure}[!h]
\begin{tikzpicture} 
  \node at (12.6,0.2) {$\Delta^{2}$};
  \node at (13.6,0.2) {$\frac{1}{h^{4}}\;\times$};
  \node[shape=circle,draw, minimum size=13pt, inner sep=0pt, fill=white]  at (14.4,0.2) {\tiny{1}};
  \node[shape=circle,draw, minimum size=13pt, inner sep=0pt, fill=white]  at (15.0,0.2) {\tiny{-8}};
  \node[shape=circle,draw, minimum size=13pt, inner sep=0pt, fill=mygray]  at (15.6,0.2) {\tiny{20}};
  \node[shape=circle,draw, minimum size=13pt, inner sep=0pt, fill=white]  at (16.2,0.2) {\tiny{-8}};
  \node[shape=circle,draw, minimum size=13pt, inner sep=0pt, fill=white]  at (16.8,0.2) {\tiny{1}};
   \node[shape=circle,draw, minimum size=13pt, inner sep=0pt, fill=white]  at (15.0,-0.4) {\tiny{2}};
  \node[shape=circle,draw, minimum size=13pt, inner sep=0pt, fill=white]  at (15.6,-0.4) {\tiny{-8}};
  \node[shape=circle,draw, minimum size=13pt, inner sep=0pt, fill=white]  at (16.2,-0.4) {\tiny{2}};
   \node[shape=circle,draw, minimum size=13pt, inner sep=0pt, fill=white]  at (15.0,0.8) {\tiny{2}};
  \node[shape=circle,draw, minimum size=13pt, inner sep=0pt, fill=white]  at (15.6,0.8) {\tiny{-8}};
  \node[shape=circle,draw, minimum size=13pt, inner sep=0pt, fill=white]  at (16.2,0.8) {\tiny{2}};
    \node[shape=circle,draw, minimum size=13pt, inner sep=0pt, fill=white]  at (15.6,-1) {\tiny{1}};
        \node[shape=circle,draw, minimum size=13pt, inner sep=0pt, fill=white]  at (15.6,1.4) {\tiny{1}};
 \end{tikzpicture}
 %\caption{Finite difference stencils that constitute the discrete bi-Laplacian $\Delta^{2}$}
 %\label{Stencil_biLap}
 \end{figure}
 
 \noindent
 In order to reflect the boundary conditions of $H_{0}^{2}(\om,\mathcal{S}^{2\times 2})$, the bi-Laplacian  must be endowed with  both zero Neumann and zero Dirichlet boundary conditions.  Again this is enforced by considering any ghost points (up to two of them in the boundary), to have zero value. Finally we discuss the dicretization of the operator $\nabla^{2}\di^{2}: V_{h}\to V_{h}$, which is equal to
 \begin{align*}
 (\nabla^{2}\di^{2}p)^{11}
 &= Dxxxx p^{11}+ 2Dxxxy p^{12}+ Dxxyy p^{22},\\
  (\nabla^{2}\di^{2}p)^{12}
&= Dxyxx p^{11} + 2Dxyxy p^{12} + Dxyyy p^{22},\\
  (\nabla^{2}\di^{2}p)^{22}
 &=Dyyxx p^{11}+ 2Dyyxy p^{12} + Dyyyy p^{22}, 
 \end{align*}
 where in fact it  holds $Dxxxy=Dxyxx$, $Dxxyy=Dxyxy=Dyyxx$ and $Dxyyy=Dyyxy$. For these fourth order discretized differential operators we use the stencils 
 \begin{figure}[!h]
\begin{tikzpicture}
 \node at (0.0,0.2) {$D_{xxxx}$};
   \node at (1.2,0.2) {$\frac{1}{h^{4}}\;\times$};
  \node[shape=circle,draw, minimum size=13pt, inner sep=0pt, fill=white]  at (2,-1) {\tiny{1}};
 \node[shape=circle,draw, minimum size=13pt, inner sep=0pt, fill=white]  at (2,-0.4) {\tiny{-4}};
 \node[shape=circle,draw, minimum size=13pt, inner sep=0pt, fill=mygray]  at (2,0.2) {\tiny{6}};
 \node[shape=circle,draw, minimum size=13pt, inner sep=0pt, fill=white]  at (2,0.8) {\tiny{-4}};
 \node[shape=circle,draw, minimum size=13pt, inner sep=0pt, fill=white]  at (2,1.4) {\tiny{1}};

\node at (3.7,0.2) {$D_{yyyy}$};
   \node at (4.9,0.2) {$\frac{1}{h^{4}}\;\times$};
 \node[shape=circle,draw, minimum size=13pt, inner sep=0pt, fill=white]  at (5.7,0.2) {\tiny{1}};
 \node[shape=circle,draw, minimum size=13pt, inner sep=0pt, fill=white]  at (6.3,0.2) {\tiny{-4}};
 \node[shape=circle,draw, minimum size=13pt, inner sep=0pt, fill=mygray]  at (6.9,0.2) {\tiny{6}};
 \node[shape=circle,draw, minimum size=13pt, inner sep=0pt, fill=white]  at (7.5,0.2) {\tiny{-4}};
 \node[shape=circle,draw, minimum size=13pt, inner sep=0pt, fill=white]  at (8.1,0.2) {\tiny{1}};

 %\node at (10.6,0.0) {$D_{xxyy}$};
  \node at (10,0.2) {$D_{yyxx}$};
     \node at (11,0.2) {$\frac{1}{h^{4}}\;\times$};
  \node[shape=circle,draw, minimum size=13pt, inner sep=0pt, fill=white]  at (11.8,0.2) {\tiny{-2}};
 \node[shape=circle,draw, minimum size=13pt, inner sep=0pt, fill=mygray]  at (12.4,0.2) {\tiny{4}};
 \node[shape=circle,draw, minimum size=13pt, inner sep=0pt, fill=white]  at (13.0,0.2) {\tiny{-2}};
  \node[shape=circle,draw, minimum size=13pt, inner sep=0pt, fill=white]  at (11.8,-0.4) {\tiny{1}};
 \node[shape=circle,draw, minimum size=13pt, inner sep=0pt, fill=white]  at (12.4,-0.4) {\tiny{-2}};
 \node[shape=circle,draw, minimum size=13pt, inner sep=0pt, fill=white]  at (13.0,-0.4) {\tiny{1}};
  \node[shape=circle,draw, minimum size=13pt, inner sep=0pt, fill=white]  at (11.8,0.8) {\tiny{1}};
 \node[shape=circle,draw, minimum size=13pt, inner sep=0pt, fill=white]  at (12.4,0.8) {\tiny{-2}};
 \node[shape=circle,draw, minimum size=13pt, inner sep=0pt, fill=white]  at (13.0,0.8) {\tiny{1}};
 
 \end{tikzpicture}
 \vspace{7pt}
 %\caption{Finite difference stencils that constitute the discrete bi-Laplacian $\Delta^{2}$}
 %\label{Stencil_biLap}
 %\end{figure}
 
%and \\

%\begin{figure}[!h]
\begin{tikzpicture}
 \node at (0.3,0.2) {$D_{xxxy}$};
 \node at (1.5,0.2) {$\frac{1}{h^{4}}\;\times$};
 \node[shape=circle,draw,  minimum size=13pt, inner sep=0pt, fill=white]  at (2.3,0.2) {\tiny{$\tfrac{3}{2}$}};
 \node[shape=circle,draw,  minimum size=13pt, inner sep=0pt, fill=mygray]  at (2.9,0.2) {\tiny{-$3$}};
 \node[shape=circle,draw,  minimum size=13pt, inner sep=0pt, fill=white]  at (3.5,0.2) {\tiny{$\tfrac{3}{2}$}};
 \node[shape=circle,draw,  minimum size=13pt, inner sep=0pt, fill=white]  at (2.3,0.8) {\tiny{-$\tfrac{3}{2}$}};
 \node[shape=circle,draw,  minimum size=13pt, inner sep=0pt, fill=white]  at (2.9,0.8) {\tiny{$2$}};
 \node[shape=circle,draw,  minimum size=13pt, inner sep=0pt, fill=white]  at (3.5,0.8) {\tiny{-$\tfrac{1}{2}$}};
 \node[shape=circle,draw,  minimum size=13pt, inner sep=0pt, fill=white]  at (2.3,-0.4) {\tiny{-$\tfrac{1}{2}$}};
 \node[shape=circle,draw,  minimum size=13pt, inner sep=0pt, fill=white]  at (2.9,-0.4) {\tiny{$2$}};
 \node[shape=circle,draw,  minimum size=13pt, inner sep=0pt, fill=white]  at (3.5,-0.4) {\tiny{-$\tfrac{3}{2}$}};
 \node[shape=circle,draw,  minimum size=13pt, inner sep=0pt, fill=white]  at (2.3,1.4) {\tiny{$\tfrac{1}{2}$}};
 \node[shape=circle,draw,  minimum size=13pt, inner sep=0pt, fill=white]  at (2.9,1.4) {\tiny{-$\tfrac{1}{2}$}};
  \node[shape=circle,draw,  minimum size=13pt, inner sep=0pt, fill=white]  at (2.9,-1)  {\tiny{-$\tfrac{1}{2}$}};
  \node[shape=circle,draw,  minimum size=13pt, inner sep=0pt, fill=white]  at (3.5,-1)  {\tiny{$\tfrac{1}{2}$}};
  
 \node at (5.4,0.2) {$D_{xyyy}$};
  \node at (6.6,0.2) {$\frac{1}{h^{4}}\;\times$};
 \node[shape=circle,draw,  minimum size=13pt, inner sep=0pt, fill=white]  at (7.4,0.2) {\tiny{-$\tfrac{1}{2}$}};
 \node[shape=circle,draw,  minimum size=13pt, inner sep=0pt, fill=white]  at (8.0,0.2) {\tiny{$2$}};
 \node[shape=circle,draw,  minimum size=13pt, inner sep=0pt, fill=mygray]  at (8.6,0.2) {\tiny{-$3$}};
 \node[shape=circle,draw,  minimum size=13pt, inner sep=0pt, fill=white]  at (9.2,0.2) {\tiny{$2$}};
 \node[shape=circle,draw,  minimum size=13pt, inner sep=0pt, fill=white]  at (9.8,0.2) {\tiny{-$\tfrac{1}{2}$}};
  \node[shape=circle,draw,  minimum size=13pt, inner sep=0pt, fill=white]  at (7.4,0.8) {\tiny{$\tfrac{1}{2}$}};
 \node[shape=circle,draw,  minimum size=13pt, inner sep=0pt, fill=white]  at (8.0,0.8) {\tiny{-$\tfrac{3}{2}$}};
 \node[shape=circle,draw,  minimum size=13pt, inner sep=0pt, fill=white]  at (8.6,0.8) {\tiny{$\tfrac{3}{2}$}};
 \node[shape=circle,draw,  minimum size=13pt, inner sep=0pt, fill=white]  at (9.2,0.8) {\tiny{-$\tfrac{1}{2}$}};
  \node[shape=circle,draw,  minimum size=13pt, inner sep=0pt, fill=white]  at (8.0,-0.4) {\tiny{-$\tfrac{1}{2}$}};
 \node[shape=circle,draw,  minimum size=13pt, inner sep=0pt, fill=white]  at (8.6,-0.4) {\tiny{$\tfrac{3}{2}$}};
 \node[shape=circle,draw,  minimum size=13pt, inner sep=0pt, fill=white]  at (9.2,-0.4) {\tiny{-$\tfrac{3}{2}$}};
 \node[shape=circle,draw,  minimum size=13pt, inner sep=0pt, fill=white]  at (9.8,-0.4) {\tiny{$\tfrac{1}{2}$}};  
 \end{tikzpicture}
 %\caption{Symmetric finite difference stencil for the fourth order derivative operators}
 %\label{Stencil_grad2div2}
 \end{figure}
 %\vspace{3em}
 
 \noindent
We use again the same rule to enforce zero Neumann and  Dirichlet boundary conditions. 
Note that the matrix representing  $\nabla^{2}\di ^{2}$  will not be symmetric due to the factor of $2$ multiplying the terms that correspond to $p^{12}$. That  leads to a non-symmetric linear system corresponding to the equation in Algorithm \ref{TGV_Newton}. However, having a symmetric matrix is desirable as this benefits from efficient iterative solvers for linear systems, e.g., conjugate gradients. A remedy for that is instead of solving for the vector $(p^{11}, p^{12}, p^{22})$ to do so for $(p^{11}, 2p^{12}, p^{22})$. This eliminates the 2-factor in the $p^{12}$ part of the matrix that represents $\nabla^{2}\di^{2}$. In that case the other operators must be modified, for instance the vector bi-Laplacian must take the form $(\Delta^{2}, \frac{1}{2}\Delta^{2}, \Delta^{2})$, and similarly for the other differential operators. The functions $Q_{\delta}$ and $P_{\delta}$ must be also accordingly modified in this case.

\bibliographystyle{amsplain}
\bibliography{kostasbib}
\end{document}